\documentclass[11pt]{article}
\usepackage[top=1in, bottom=1in, left=1in, right=1in]{geometry}

\usepackage{./dsm5_header}

\setTitle{Residual based adaptivity and PWDG methods for the Helmholtz equation}

\usepackage{footnote} 

\addAuthor{Shelvean Kapita\footnote[1]{Department of Mathematics, University of Delaware, 210 S. College Ave., Newark, Delaware}}
\addAuthor{Peter Monk\footnotemark[1]}
\addAuthor{T. Warburton\footnote[2]{Department of Computational and Applied Mathematics, Rice University, 6100 Main Street MS--134, Houston, Texas 77005--1892}.}

\usepackage{color} 
\usepackage{graphicx}
\usepackage{bbm}
\usepackage{amsmath}
\usepackage{amssymb}
\usepackage{amsfonts}   
\usepackage{mathtools}
\usepackage{multirow}
\usepackage{array}
\input epsf.sty
\newcommand{\jmp}[1]{\left[\!\left[#1\right]\!\right]}                     
\newcommand{\avg}[1]{\left\{\!\!\left\{#1\right\}\!\!\right\}}                   
\usepackage{pgfplots}
\usepackage{float}
\usepackage{subfigure}

\newcommand{\cEI}{\cal{E}_I}
\newcommand{\cEA}{\cal{E}_A}
\newcommand{\cED}{\cal{E}_D}

\newtheorem{theorem}{Theorem}[section]
\newtheorem{lemma}[theorem]{Lemma}

\newenvironment{remark}[1][Remark]{\begin{trivlist}
\item[\hskip \labelsep {\bfseries #1}]}{\end{trivlist}}

\newcommand{\bfv}{\mathbf{v}}

\definecolor{lightgray}{gray}{0.75}

\begin{document}

\placeTitlePage

\begin{outline}
We present a study of two residual a posteriori error indicators for the Plane Wave Discontinuous Galerkin (PWDG) method for the Helmholtz equation.  In particular we study the $h$-version of PWDG in which the number of plane wave directions per element is kept fixed.  First we use a slight modification of the appropriate a priori analysis to determine a residual indicator.  Numerical tests show that this is reliable but pessimistic in that the ratio between the true error and the indicator increases as the mesh is refined.  We therefore introduce a new analysis based on the observation that 
sufficiently many plane waves can approximate piecewise linear functions as the mesh is refined.   Numerical results demonstrate an improvement in the efficiency of the
indicators.
\end{outline}

\tableofcontents

\section{Introduction}
\label{introduction.sec}
We shall investigate the use of an adaptive plane wave discontinuous Galerkin (PWDG) method for approximating the solution of the Helmholtz equation with mixed boundary conditions. In particular, given a bounded Lipschitz polyhedral domain $\Omega\subset\mathbb{R}^2$ with boundary $\Gamma$ consisting of two disjoint components
$\Gamma_D$ and $\Gamma_A$ and unit outward normal $\nu$ we want to approximate the solution $u$ of
\begin{eqnarray}
\Delta u + \kappa^2 u&=&0\quad\mbox{ in }\Omega\label{eq:helmholtz},\\
\frac{\partial u}{\partial \nu}-i\kappa u&=&g_A\quad\mbox{ on }\Gamma_A\label{eq:impedance},\\
u&=&0\quad\mbox{ on }\Gamma_D.\label{eq:dirichlet}
\end{eqnarray}
Here the wave number $\kappa>0$ and $g_A\in L^2(\Gamma)$ is a given function.  This problem is often considered 
because the Robin boundary condition (\ref{eq:impedance}) is a simple absorbing boundary condition, so the problem serves as a simplified model for scattering from a bounded domain (in the scattering example, $g_A$ is determined by the incident field).  In particular we note
that this problem is considered in \cite{HMP13} which has motivated part of our study.
We could also include piecewise constant coefficients in the partial differential
equation without any complication of the algorithm, but the proofs we shall present require constant coefficients.  

The PWDG method we shall consider is a generalization of the Ultra Weak Variational Formulation  of the Helmholtz equation due to Cessenat and Despr\'es  \cite{cessenat_phd,despres}.  This method uses piecewise solutions of the Helmholtz equation in a non-standard variational scheme on a finite element mesh to approximate the trace of  $u$ and the normal derivative of $u$ on edges in the mesh.  In \cite{buf07} this was recognized to be a equivalent to a discontinuous Galerkin method, and this observation was then used
to prove error estimates under restrictive conditions on the domain and mesh.  At the same time, a generalized discontinuous Galerkin method based on possibly mesh dependent penalty
parameters was analyzed using a classical approach in \cite{git07}, and later using the approach of \cite{buf07} in \cite{HMP11} where error estimates for a $p$-version exhibiting wave number dependence and a more precise estimate of approximation 
properties.  In \cite{HMPS12} exponential convergence of the $hp$ PWDG method to smooth solutions is proved.  Particularly important for this paper is the analysis in \cite{HMP13} where special penalty parameters are chosen that allow the derivation of an error estimate even on highly refined grids.

For background on PWDG and other methods using plane wave solutions for the Helmholtz equation a useful paper
is \cite{MelenkEsterhazy12}. For computational aspects, see \cite{HMP13I} and for a dispersion analysis see \cite{GH14}.  The UWVF or its PWDG generalization have been applied to Maxwell's equations \cite{cessenat_phd,HMM06,HMP13M}, to the
linear elastic Navier equation \cite{LHM13} and to the biharmonic problem \cite{LHM14}.

We are interested in deriving a posteriori error indicators based on residuals to drive the PWDG method adaptively to a solution.  To our knowledge this is the first study of adaptivity for these methods. Ideally this study would include adaptivity in the number and direction of the basis functions per element (like $p$-adaptivity for polynomial
methods) and also mesh refinement or $h$-adaptivity.  The former type of adaptivity is currently difficult to attain and is not the subject
of this paper.  Instead we shall concentrate on more classical $h$-adaptivity where we fix the number of basis functions per element and only refine the mesh.

We start from the observation that the estimates in \cite{HMP13} can easily be modified to give a residual based a posteriori error estimator for the $L^2$ norm.  This is done in Section \ref{sec:estimate}.  We then test these estimates on a model problem with a smooth solution.  We find that the estimator is reliable, but not efficient.  It progressively over estimates the global $L^2$  norm error.  Despite this, in the case of a smooth solution, the refinement path produces an optimal order  approximation.  

It is clear from these numerical results (and the numerical experiments in \cite{cessenat_phd}) that both the a priori and the a posteriori theory  are not optimal with respect to the mesh width.  We therefore revisit the derivation of the residual indicator.  In particular we note that from Lemma 3.10 in \cite{git07}, 
on a refined mesh, sufficiently many plane wave basis functions can approximate piecewise linear finite element functions.
This allows us to improve powers of the mesh size appearing in the a posteriori indicators.  The theory behind this observation is presented in Section \ref{sec:refest}. We then test the new indicators in Section \ref{resultsII.sec}.  The resulting residual estimators are seen to be an improvement over those in Section \ref{sec:estimate}. We then draw 
some conclusions and discuss possible extensions of this theory in Section \ref{conclusion.sec}.

\section{Notation and Preliminaries}
\label{notation.sec}
We generally adopt the notation from \cite{HMP13}. The domain $\Omega$ is assumed to be  Lipschitz smooth and to be an annular region in the sense that there are
polygons $\Omega_A$ with boundary $\Gamma_A$ and $\Omega_D$ with boundary $\Gamma_D$ with connected boundaries such that the closure $\overline{\Omega}_D$ is strictly contained in $\Omega_A$. Then
\[
\Omega=\Omega_A\setminus\overline{\Omega_D}.
\]
To prove existence of a solution to (\ref{eq:helmholtz})-(\ref{eq:dirichlet}), the following space is used:
\[
H_{\Gamma_D}(\Omega)=\left\{u\in H^1(\Omega)\;|\; u=0 \mbox{ on }\Gamma_D\right\}.
\]
The norm used is weighted with the wave-number:
\[
\Vert u\Vert_{1,\kappa,\Omega}^2=\kappa^2\Vert u\Vert_{L^2(\Omega)}^2+\Vert \nabla u\Vert_{L^2(\Omega)}^2.
\]

Then it is shown in \cite[Theorem 2.1]{HMP13} that there exists a weak solution to
the above mentioned problem.  In addition if $d_\Omega$ is the diameter of $\Omega$ then
\[
\Vert u\Vert_{1,\kappa,\Omega}\leq Cd_{\Omega}^{1/2}\Vert g_A\Vert_{L^2(\Gamma_A)}.
\]

We assume that $\Omega$ is covered by a family of meshes ${\cal T}_h$ indexed by the maximum diameter of the elements in the mesh so that $h>0$, and for any element $K\in {\cal T}_h$ we set
\[
h_K={\rm diam}(K)
\] 
where ${\rm diam}(K)$ is the diameter of the smallest circumscribed circle containing $K$.

Because we wish to derive a posteriori error indicators on refined meshes we follow \cite{HMP13} and make the next three assumptions on the mesh: 
\begin{description}
\item[\em Shape Regularity:]  For any element $K\in {\cal T}_h$ let $\rho_K$ denote the diameter of the largest inscribed circle in $K$. Then there exists a constant $\sigma$ independent of $h$ such that for all $K$ in ${\cal T}_h$, $h_K/\rho_K\leq \sigma$.

\item[\em Local quasi-uniformity:]  Suppose $K_1, K_2\in {\cal T}_h$ meet on an edge $e$.  Then there exists a constant $\tau$ independent of $h$ such that
\[
\tau^{-1}\leq \frac{h_{K_1}}{h_{K_2}}\leq \tau
\]
for all such choices of $K_1$ and $K_2$.

\item[\em Quasi-uniformity close to $\Gamma_A$:]
There exists a constant $\tau_A$ such that for all $h$ and  all $K\in{\cal T}_h$ such that $K$ shares an edge with $\Gamma_A$
\[
\frac{h}{h_K}\leq \tau_A. 
\]
\end{description}
As pointed out in \cite{HMP13} the goal is to refine the grid around the scatter where the Dirichlet boundary condition occurs.  The impedance boundary condition models an outgoing radiation condition and so uniform refinement should occur near that boundary.

We make one other major assumption because we need to use results from \cite{HMP13} that depend on it: we assume that the triangles in the grid are all affine images of a finite number of reference elements.  This may be less of a concern for an adaptive method because the elements are obtained from a refinement of an initial coarse mesh, however in our numerical results we cannot ensure that this assumption holds.

Suppose $K^{\pm}$ are a pair of elements sharing a common edge $e$ and having outward normals $\nu^{\pm}$ respectively.  We define the jumps and averages of a suitably smooth function $v^\pm$ defined on each element by
\[
\avg{v}=\frac{1}{2}(v^++v^-)|_e,\qquad \jmp{v}=(v^+\nu^++v^-\nu^-)|_e.
\]
Similarly for a piecewise defined vector function $\bfv^\pm$ we define
\[
\avg{\bfv}=\frac{1}{2}(\bfv^++\bfv^-)|_e,\qquad \jmp{\bfv}=(\bfv^+\cdot\nu^++\bfv^-\cdot\nu^-)|_e.
\]
The set of interior edges of elements in ${\cal T}_h$ is denoted ${\cal E}_I$.  Edges on the boundary $\Gamma_A$ are denoted ${\cal E}_{A}$ and on $\Gamma_D$ by ${\cal E}_D$.
In later sections we shall make frequent use of the following trace inequality, for any edge $e$ of a triangle $K$ and any $w\in H^1(K)$ there exists a constant $C$ independent of $K$ and $w$ such that
\begin{equation}
\Vert w\Vert_{L^2(e)}^2\leq C\left(\frac{1}{h_K}\Vert w\Vert_{L^2(K)}^2+h_K\Vert \nabla w\Vert_{L^2(K)}^2\right).\label{trace}
\end{equation}
In addition under the assumptions on the mesh (in particular that all elements are the affine image of a finite number of
reference elements) estimate (24) of \cite{HMP13} states that if $w\in H^{s+3/2}(K)$ for some $1/2\geq s>0$ then
\begin{equation}
\Vert\nabla w\Vert_{L^2(e)}\leq C\left(h_K^{-1}\Vert\nabla w\Vert_{L^2(K)}^2+h_K^{2s}\vert\nabla w\vert_{H^{1/2+s}(K)}^2\right),
\label{trace2}
\end{equation}
where $|\cdot|_{H^{1/2+s}(K)}$ is the $H^{1/2+s}(K)$ semi-norm.

The PWDG method used here is based on the use of plane waves propagating in different directions on each element.  Let $p_K$ denote the number of directions used
on element $K$ and this will be fixed in this paper.  We use uniformly spaced directions on the unit circle
\[
\mathbf{d}_j^K=(\cos(\theta^K_j),\sin(\theta^K_j)),\;1\leq j\leq p_K
\]
where $\theta_j^K=2\pi j/p_K$.  On an element $K$ the local solution space is
\[
V^K_{P_K}=\mbox{span}\left\{\exp(ik\mathbf{d}_j^K\cdot\mathbf{x}),\quad 1\leq j\leq p_K\right\}.
\]
Then the global solution space is
\begin{equation}
V_{h}=\left\{ u_{h}\in L^2(\Omega)\;|\; u_{h}|_K\in V^K_{p_K}\mbox{ for all }K\in {\cal T}_h\right\}.\label{eq:Vdef}
\end{equation}

\section{The PWDG method}
\label{formulation.sec}
Since our code is based on the discontinuous Galerkin approach to discretizing first order systems, we give a brief
derivation of PWDG starting by introducing a vector variable $\sigma$ such that
\[
i\kappa \sigma=\nabla u, \mbox { so } i\kappa u=\nabla\cdot\sigma.
\]
Multiplying these equations by the complex conjugate of smooth test functions $v$ and $\tau$, integrating over
an element $K$ and adding the results we obtain
\[
\int_{K}\left\{i\kappa\sigma\cdot\overline{\tau}+u\nabla\cdot\overline{\tau}+i\kappa u\overline{v}+\sigma\cdot\nabla\overline{v}\right\}\,dA=\int_{\partial K} \left\{u\overline{\tau\cdot\nu}
+ \sigma\cdot\nu \overline{v}\right\}\,ds.
\]
Rearranging, and replacing the boundary flux terms $(u,\sigma\cdot\nu)$ on the right hand side by consistent numerical fluxes $(u,\sigma\cdot\nu)$ in the normal way (we will specify the fluxes shortly), we obtain
\[
\int_{K}\left\{\sigma\cdot(\overline{-i\kappa\tau+\nabla v})+u(\overline{-i\kappa v+\nabla\cdot\tau})\right\}\,dA=\int_{\partial K} \left\{\hat{u}\overline{\tau\cdot\nu}
+ \hat{\sigma}\cdot\nu \overline{v}\right\}\,ds.
\]
Finally assuming the test functions also satisfy the first order system corresponding to
the Helmholtz equation, 
\[
i\kappa\tau=\nabla v, \mbox { so } i\kappa v=\nabla\cdot\tau.
\]
we obtain
\begin{equation}
\int_{\partial K} \left\{\hat{u}\overline{\tau\cdot\nu}
+ \hat{\sigma}\cdot\nu \overline{v}\right\}\,ds=0\label{eq:fund}
\end{equation}
on each element in the mesh.

It remains to detail the fluxes. We follow \cite{HMP13}. On an interior edge in the mesh we take the numerical fluxes to be
\begin{eqnarray*}
\hat{u}&=&\avg{u}-\frac{\beta}{{\rm{}i}\kappa}\jmp{\nabla_h u},\\
ik\widehat{\sigma}&=&\avg{\nabla_h u}-{\rm{}i}\kappa\alpha \jmp{u}.
\end{eqnarray*}
Here $\nabla_hu$ is the broken (piecewise) gradient.  On  boundary edges on $\Gamma_A$ the fluxes are
\begin{eqnarray*}
\hat{u}&=&u-\delta ((i\kappa)^{-1}\nabla_hu\cdot\nu -u-(i\kappa)^{-1}g_A),\\
i\kappa\widehat{\sigma}&=&\nabla_h u-(1-\delta)(\nabla_h u-i\kappa u\nu-g_A\nu).
\end{eqnarray*}
Finally on edges on the Dirichlet portion of the boundary the fluxes are
\begin{eqnarray*}
\hat{u}&=&0,\\
i\kappa\widehat{\sigma}&=&\nabla_h u-\alpha i\kappa u\nu.
\end{eqnarray*}

Adding (\ref{eq:fund}) over all elements in the mesh and using the above fluxes, we obtain the PWDG method of \cite{HMP13}.  In particular let
\begin{eqnarray}
&&A_h(u,v)=
\int_{{\cal E}_I}\avg{u}\jmp{\nabla_h\overline{v}}\,ds
-\int_{{\cal E}_A}\delta(\nabla_h u\cdot\nu)\overline{v}\,ds
-\int_{{\cal E}_I}\jmp{\overline{v}}\cdot\avg{\nabla_hu}\,ds\label{eq:Adef}\\&&+\int_{{\cal E}_A}(1+\delta) u(\nu\cdot\nabla_h\overline{v})\,ds-\frac{1}{i \kappa}\int_{{\cal E}_I}\beta\jmp{\nabla_h u}\jmp{\nabla_h\overline{v}}\,ds-\frac{1}{i\kappa}\int_{{\cal E}_A}\delta(\nu\cdot\nabla_hu)(\nu\cdot\nabla_h\overline{v})\,ds\nonumber\\&&
+i\kappa\int_{{\cal E}_I}\alpha\jmp{u}\cdot\jmp{\overline{v}}\,ds-i\kappa\int_{{\cal E}_A}(1-\delta)u\overline{v}\,ds-\int_{{\cal E}_D}(\nabla_hu\cdot\nu)\overline{v}\,ds+\int_{{\cal E}_D}\alpha i\kappa u\overline{v}\,ds.\nonumber
\end{eqnarray}
and
\begin{eqnarray}
\ell(v)=-\frac{1}{i\kappa}\int_{{\cal E}_A}\delta g_A(\nu\cdot\nabla_h \overline{v})\,ds+\int_{{\cal E}_A}(1-\delta)g_A\overline{v}\,ds.\label{eq:elldef}
\end{eqnarray}
An important contribution of \cite{HMP13} is that on a refined mesh the coefficients should be chosen as follows.  Let $e$ be an edge in the mesh having length $h_e$ then
\begin{eqnarray}
\alpha|_e&=&\frac{ah}{h_e},\;\beta|_e=\frac{bh}{h_e},\;\delta|_e=\frac{dh}{h_e}\leq \frac{1}{2},
\end{eqnarray}
where $a,b,d$ are positive constants. Then the discrete solution $u_{h}\in V_{h}$ satisfies
\[
A_h(u_{h},v)=\ell(v)\qquad\forall v\in V_{h}.
\]
In \cite{HMP13} it is shown that this equation has a solution regardless of the mesh size and wave number, and a global $L^2(\Omega)$ norm error estimate is proved.  Therefore we shall concentrate on a posteriori estimates for the global $L^2$ norm here.

We now recall an equivalent form of the sesquilinear form $a_h(\cdot,\cdot)$ based on using the following  ``magic lemma'':
\begin{lemma}[Lemma 6.1 from \cite{MelenkEsterhazy12}] \label{magic} For any  sufficiently smooth piecewise  defined vector field $\sigma$ and piecewise defined function
$v$
\[
\sum_{K\in T_h}\int_{\partial K}v\sigma\cdot n\,ds=
\int_{\cEI}\jmp{v}\cdot\avg{\sigma}\,ds+\int_{\cEI}\avg{v}\cdot\jmp{\sigma}\,ds+\int_{\cEA\cup\cED}v\mathbf{\nu}\cdot\sigma\,ds
\]
\end{lemma}
Using this lemma to rewrite appropriate terms in (\ref{eq:Adef}) and using the identity that
\[
0=\int_{K}(-\Delta u -k^2 u)\overline{v}\,dA=\int_K(\nabla u\cdot\nabla\overline{v}-k^2u\overline{v})\,dA+
\int_{\partial K}\frac{\partial u}{\partial \nu}\overline{v}\,ds,
\]
we obtain, for any pair of piecewise solutions of the Helmholtz equation $(u,v)$ respectively in $H^{3/2+s}(K)$, $s>0$:
\begin{eqnarray}
A_h(u,v)&=& \int_{\Omega}\left(\nabla_hu\cdot\nabla_h\overline{v}
-\kappa^2u\,\overline{v}\right)\,dA-\int_{\cEI}\left(\avg{\nabla_h u}\cdot\jmp{\overline{v}}\,ds
+\jmp{ u}\cdot\avg{\nabla_h\overline{v}}\right)\,ds\nonumber\\&&
-\frac{1}{i\kappa}\int_{\cEI}\beta\jmp{\nabla_h u}\jmp{\nabla_h\overline{v}}\,ds
+{i\kappa}\int_{\cEI}\alpha\jmp{ u}\cdot\jmp{\overline{v}}\,ds+\int_{\cEA}\delta u\nabla_h \overline{v}\cdot\nu\,ds\nonumber\\&&
-\int_{\cEA}\delta\nabla_h u\cdot\nu\overline{v}\,ds
-\frac{1}{i\kappa}\int_{\cEA}\delta(\nabla_h  u\cdot\nu)(\nabla_h \overline{v}\cdot\nu)\,ds+i\kappa\int_{\cED}\alpha u\overline{v}\,ds\nonumber\\&&
-i\kappa\int_{\cEA}(1-\delta)u\overline{v}\,ds -\int_{\cED}
\left((\nabla_h u\cdot\nu)\overline{v}+u(\nabla_h\overline{v}\cdot\nu)\right)\,ds.\label{Ahalt}
\end{eqnarray}
For later a posteriori analysis we shall also need to integrate the first term by parts again, and use the ``magic lemma''
to obtain (recalling also that $u$ satisfies the Helmholtz equation on each element):
\begin{eqnarray}
A_h(u,v)&=& \int_{\cEI}\left(\jmp{\nabla_h u}\avg{\overline{v}}
-\jmp{ u}\cdot\avg{\nabla_h\overline{v}}\right)\,ds-\frac{1}{i\kappa}\int_{\cEI}\beta\jmp{\nabla_h u}\jmp{\nabla_h\overline{v}}\,ds\nonumber\\&&
+{i\kappa}\int_{\cEI}\alpha\jmp{ u}\cdot\jmp{\overline{v}}\,ds-\int_{\cEA}\frac{\delta}{i\kappa} (i\kappa u-\nabla_h u\cdot\nu)\nabla_h \overline{v}\cdot\nu\,ds\nonumber\\&&
+\int_{\cEA}(1-\delta)(\nabla_h u\cdot\nu-i\kappa u)\overline{v}\,ds
+\int_{\cED}u(i\kappa\alpha \overline{v}-\nabla_h\overline{v}\cdot\nu)\,ds.\label{Ahap}\end{eqnarray}

\section{A posteriori estimates I}
\label{estimates.sec}
\label{sec:estimate}
 In this section we shall prove an a posteriori error estimate using residuals in the global $L^2$ norm.  This is the theoretical basis for the ESTIMATE step in the adaptive cycle of our code.

We shall need the solution of the following adjoint problem of finding $z\in H^1(\Omega)$ such that
\begin{eqnarray}
-\Delta z-\kappa^2z&=&(u-u_{h})\mbox{ in }\Omega,\label{zadj1}\\
\frac{\partial z}{\partial \nu}+i\kappa z&=&0\mbox{ on }\Gamma_A,\label{zadj2}\\
z&=&0\mbox{ on }\Gamma_D.\label{zadj3}
\end{eqnarray}
Theorem 3.2 of \cite{HMP13}  shows that a unique solution exists for the above problem and $z\in H^{3/2+s}(\Omega)$ for some $1/2\geq s>0$ (determined by the reentrant angles of the boundary). In addition 
\begin{eqnarray}
\sqrt{\Vert \nabla z\Vert^2_{L^2(\Omega)}+\kappa^2\Vert z\Vert^2_{L^2(\Omega)}}&\leq& Cd_\Omega\Vert u-u_{h}\Vert_{L^2(\Omega)},\\
\vert \nabla z\vert_{H^{1/2+s}(\Omega)}&\leq &C(1+d_\Omega \kappa)d^{1/2-s}_\Omega\Vert u-u_{h}\Vert_{L^2(\Omega)}^2,\label{eq:sdef}
\end{eqnarray}
where $d_\Omega$ is the diameter of $\Omega$.

\begin{theorem}\label{MPH_L2}
Let $u_{h}\in V_{h}$  then
\begin{eqnarray}\label{eq: l2 estimate}
\Vert u-u_{h}\Vert_{L^2(\Omega)}
&\leq& C\tau^{1/2}d_\Omega\left[(\kappa h)^{-1/2}+(d_\Omega \kappa )^{1/2}(d_\Omega^{-1}h )^s\right]
\eta_{DG}(u_{h})
\end{eqnarray}
where $s$ is the regularity exponent in (\ref{eq:sdef}) and the residual error indicator is given by
\begin{eqnarray}
\eta_{DG}(u_{h})^2&=&
\kappa ^{-1}\Vert \beta^{1/2}\jmp{\nabla_h u_{h}}\Vert_{L^2({\cal E}_I)}^2
+\kappa \Vert \alpha^{1/2}\jmp{u_{h}}\Vert_{L^2({\cal E}_I)}^2\\
&&+\kappa ^{-1}\Vert \delta^{1/2}(g-\nabla u_{h}\cdot\nu+i\kappa u_{h})\Vert_{L^2(\Gamma_A)}^2+
\kappa \Vert\alpha^{1/2} u_{h}\Vert_{L^2(\Gamma_D)}^2.\nonumber
\end{eqnarray}

\end{theorem}
{{\em Remark:} The function $u_{h}$ above does not have to be the PWDG solution.  It could also be a least squares solution. But in our examples it will be computed by PWDG.

\begin{proof}
Let $w=u-u_{h}$ .  Then $w$ is a piecewise solution of  $\Delta w + \kappa^2w =0$ on each $K\in {\cal{T}}_h$ such that
$w\in H^{3/2+s}(K)$ for each element in the mesh.   Multiplying (\ref{zadj1}) of the adjoint equation  by $w$ and integrating by parts on each element $K$, we get
\begin{eqnarray*}
(w,w)_{0,\Omega} &=& \sum\limits_{K\in {\cal{T}}_h}\int\limits_{K} w(-\Delta\overline{z}-\kappa^2\overline{z})\;dA \\
&=& \sum\limits_{K\in {\cal{T}}_h}\int\limits_{K} (\nabla w\cdot\nabla\overline{z}-\kappa^2 w\overline{z})\;dA-\sum\limits_{K\in {\cal{T}}_h}\int\limits_{\partial K} w \frac{\partial \overline{z}}{\partial \nu}\;ds \\
&=& \sum\limits_{K\in {\cal{T}}_h}\int\limits_{K}-(\Delta w+\kappa^2 w)\overline{z}\;dA+\sum\limits_{K\in {\cal{T}}_h}\int\limits_{\partial K} \left(\frac{\partial w}{\partial \nu}\overline{z}- w\frac{\partial \overline{z}}{\partial \nu}\right)\;ds.
\end{eqnarray*}
Here $(\cdot,\cdot)_{0,\Omega}$ is the $L^2(\Omega)$ inner product.
Since $\Delta w+\kappa^2w=0$, the volume term drops, using in addition the boundary conditions (\ref{zadj2})-(\ref{zadj3}) we have 
\begin{eqnarray*}
(w,w)_{0,\Omega} &=& \sum\limits_{K\in {\cal{T}}_h}\int\limits_{\partial K} \left(\frac{\partial w}{\partial \nu}\overline{z}- w\frac{\partial \overline{z}}{\partial \nu}\right)\;ds \\
&=& \int\limits_{\cal{E}_I} \left(\jmp{\nabla_h w}\overline{z}-\jmp{w}\cdot\nabla_h\overline{z}\right)\;ds+\int\limits_{\cal{E}_A}
\left( \nabla_h w\cdot \nu \overline{z}-\frac{\partial \overline{z}}{\partial \nu}w\right)\;ds-\int\limits_{\cal{E}_D}w\nabla_h\overline{z}\cdot\nu\;ds \\
&=& \int\limits_{\cal{E}_I} \left(\jmp{\nabla_h w}\overline{z}-\jmp{w}\cdot\nabla_h\overline{z}\right)\;ds+\int\limits_{\cal{E}_A}
\left( \nabla_h w\cdot \nu -i\kappa w\right)\overline{z}\;ds-\int\limits_{\cal{E}_D}w\nabla_h\overline{z}\cdot\nu\;ds.
\end{eqnarray*}
Multiplying and dividing by factors of $\kappa,$ $\alpha$, $\beta$ and $\delta$, we obtain by the Cauchy-Schwarz inequality,
\begin{eqnarray*}
|(w,w)|_{0,\Omega} &\leq & \sum\limits_{e\in {\cal{E}_I}} (\kappa^{-1/2}\Vert \beta^{1/2}\jmp{\nabla_h w}\Vert_{L^2(e)}\; \kappa^{1/2}\Vert \beta^{-1/2}z\Vert_{L^2(e)} \\& & +\; \kappa^{1/2}\Vert \alpha^{1/2}\jmp{w}\Vert_{L^2(e)}\; \kappa^{-1/2}\Vert \alpha^{-1/2}\nabla_h z\cdot\nu\Vert_{L^2(e)})  \\ & & 
+ \sum\limits_{e\in {\cal{E}_A}} \kappa^{-1/2}\Vert \delta^{1/2}(\nabla_h w\cdot\nu -i\kappa w)\Vert_{L^2(e)}\; \kappa^{1/2}\Vert \delta^{-1/2}z\Vert_{L^2(e)} \\ & & 
+ \sum\limits_{e\in {\cal{E}_D}} \kappa^{1/2}\Vert \alpha^{1/2} w\Vert_{L^2(e)}\;\kappa^{-1/2}\Vert \alpha^{-1/2}\nabla_h z\cdot\nu\Vert_{L^2(e)}.
\end{eqnarray*}
Hence, expanding $w$ and noting that $\jmp{u}=0$ and $\jmp{\nabla_h u}=0$ on interior edges, and taking into account the boundary conditions for $u$, we have
\begin{eqnarray}
\Vert u-u_{h},\phi\Vert^2_{0,\Omega}&\leq & \eta_{DG}(u_{h}) {\cal{G}}(z)^{1/2}\label{esta}
\end{eqnarray}
where 
\begin{eqnarray*}
{\cal{G}}(z)&:=&\sum\limits_{e\in {\cal{E}_I}} \left(\kappa \Vert \beta^{-1/2} z \Vert_{L^2(e)}^2 + \kappa^{-1}\Vert \alpha^{-1/2}\nabla_h z\cdot\nu\Vert^2_{L^2(e)}\right)\\
& & +\sum\limits_{e\in {\cal{E}_A}}\kappa \Vert \delta^{-1/2}z\Vert^2_{L^2(e)} + \sum\limits_{e\in {\cal{E}_D}} \kappa^{-1}\Vert \alpha^{-1/2}\nabla_h z\cdot \nu\Vert^2_{L^2(e)}.
\end{eqnarray*}
Using the definitions of the penalty terms $\alpha,\beta,\delta$, the factor ${\cal{G}}(z)$ is bounded by 
$${\cal{G}}(z)\leq C\tau d_{\Omega}^2(\kappa^{-1}h^{-1}+d_{\Omega}^{1-2s}\kappa h^{2s})\Vert u-u_{h} \Vert^2_{0,\Omega}.$$
For details  see \cite[Lemma 4.4]{HMP13}.
Using this result in (\ref{esta}) proves the theorem.\hfill\end{proof} 

We now test the error indicators derived above to drive $h$-adaptivity (we keep the number of
directions per element fixed and equal on all elements).
Our first test uses a smooth solution on an L-shaped domain.  In this case uniform refinement is likely to be optimal, and we expect the adaptive method to result in an approximately uniform mesh.
 All computations are done in MATLAB and we shall discuss the algorithm later in Section \ref{resultsII.sec}.

We consider an $L$-shaped domain $\Omega=(-1,1)^2\backslash([0,1]\times[-1,0]).$ We choose Dirichlet boundary conditions such that the exact solution of 
$(\ref{eq:helmholtz})$ is given by 
\begin{equation}
u(\textbf{x}) = J_{\xi}(\kappa r)\sin(\xi\theta)\label{reg_exact}
\end{equation}
where $\textbf{x}=r(\cos\theta,\sin\theta)$,
for $\xi = 2$ (later we will also choose $\xi = 2/3$ corresponding to a singular solution) and $\kappa=12$.  Here $J_{\xi}$ denotes the Bessel function of the first kind and order $\xi$.  The solution is shown in Fig.~\ref{L_smooth_sol}.  Note that although we have not implemented the impedance boundary condition, the theory in this section can also be proved with 
just the Dirichlet boundary condition provided $\kappa^2$ is not an interior Dirichlet eigenvalue for the domain.  In the Dirichlet case the dependence of the overall coefficient on $\kappa$ cannot be estimated.  But the overall constant is not used in the marking strategy.

\begin{figure}
\begin{center}
\includegraphics[scale=0.5]{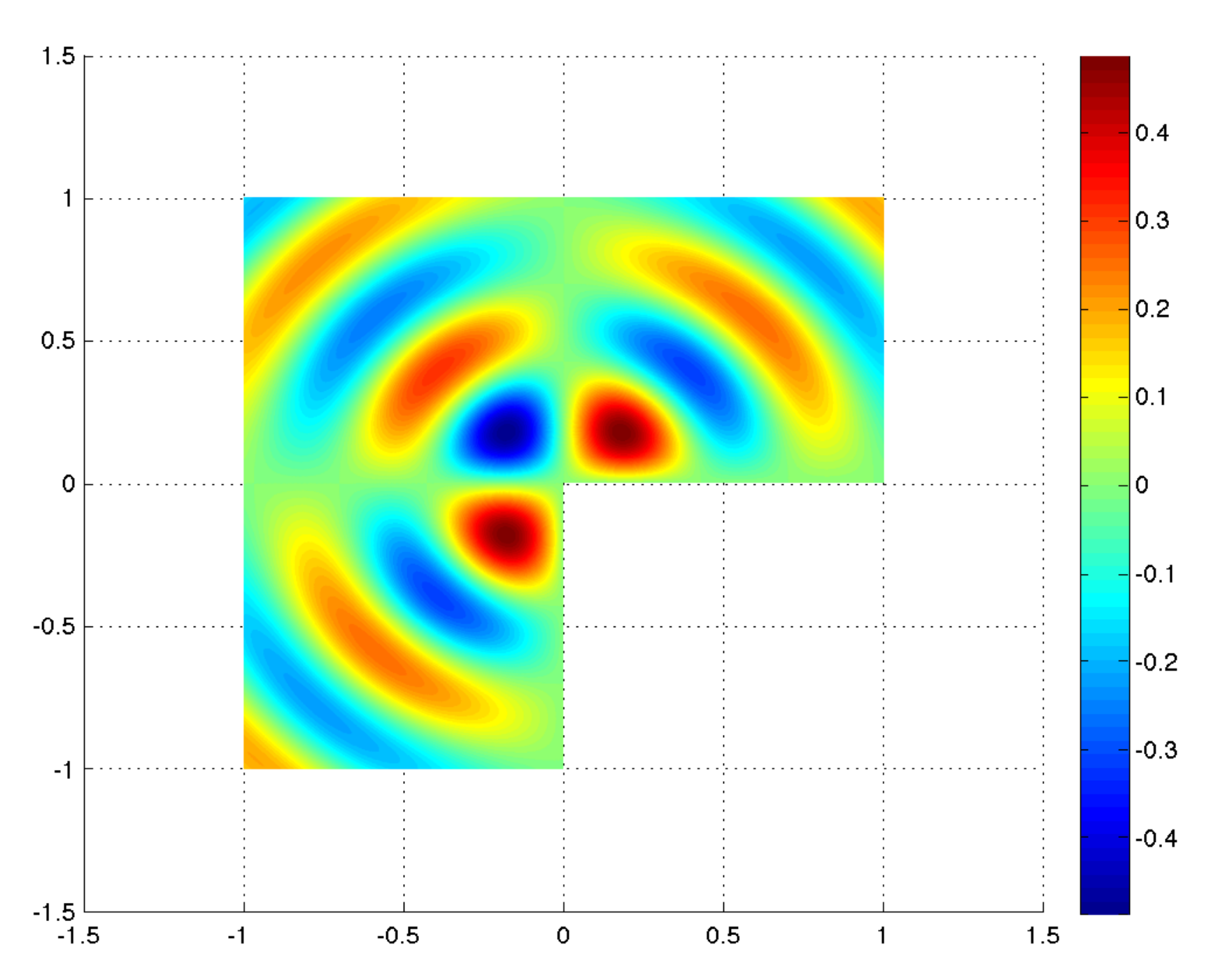}
\end{center}
\caption{The computed solution after 12 iterations when $\xi=2$ and $k=12$ using $p_K=7$ plane waves per element.
 This is indistinguishable graphically from the exact solution.}
\label{L_smooth_sol}
\end{figure}

 The initial mesh and the refined mesh after 12 adaptive steps are shown in Fig.~\ref{init_mesh_L}.  We see that the adaptive scheme has correctly chosen to refine almost uniformly in the domain since there is no singularity at the reentrant corner.

\begin{figure}
\centering
\subfigure[Initial]{%
\includegraphics[scale=0.4]{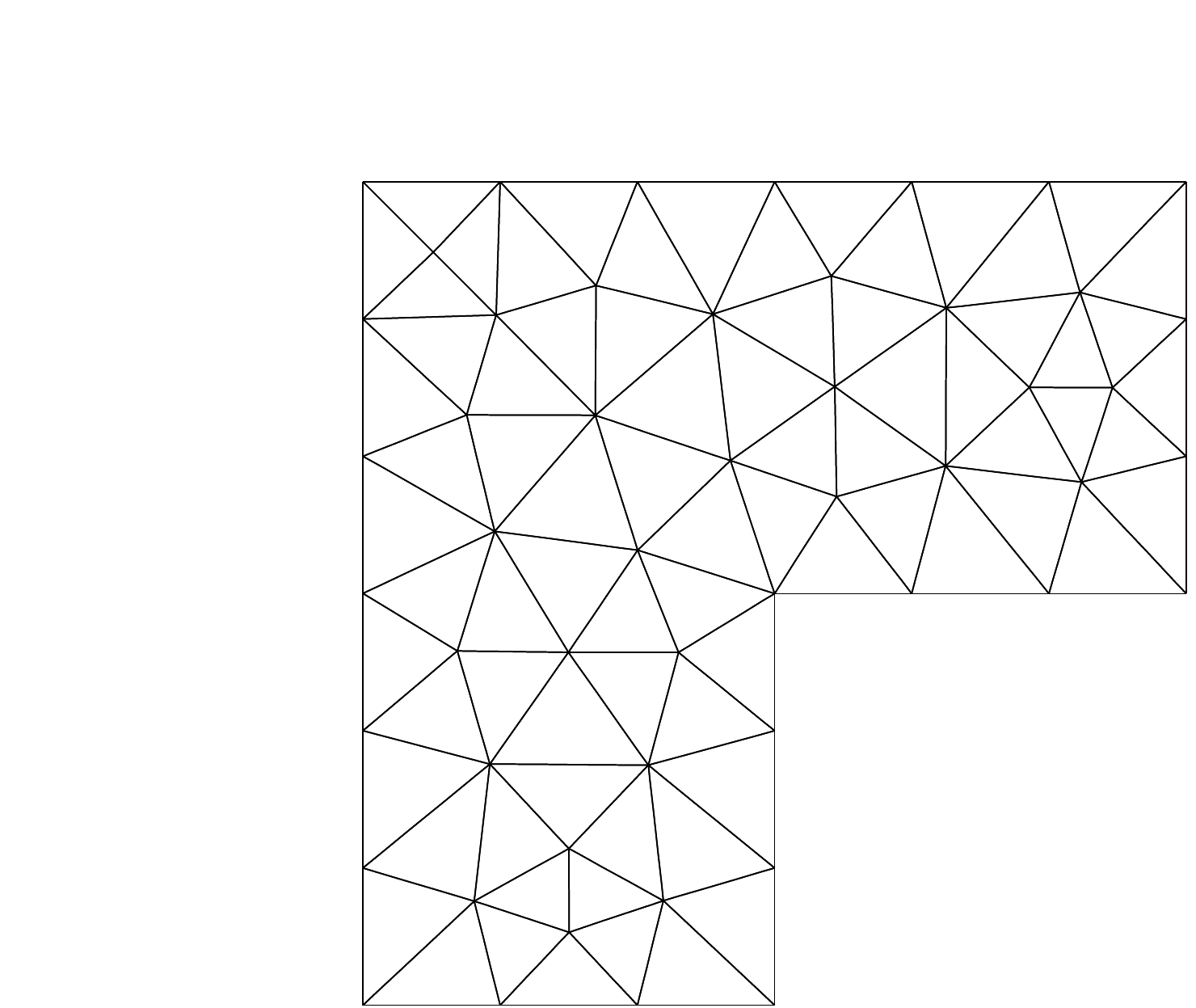}}
\quad
\subfigure[After 12 iterations]{%
\includegraphics[scale=0.4]{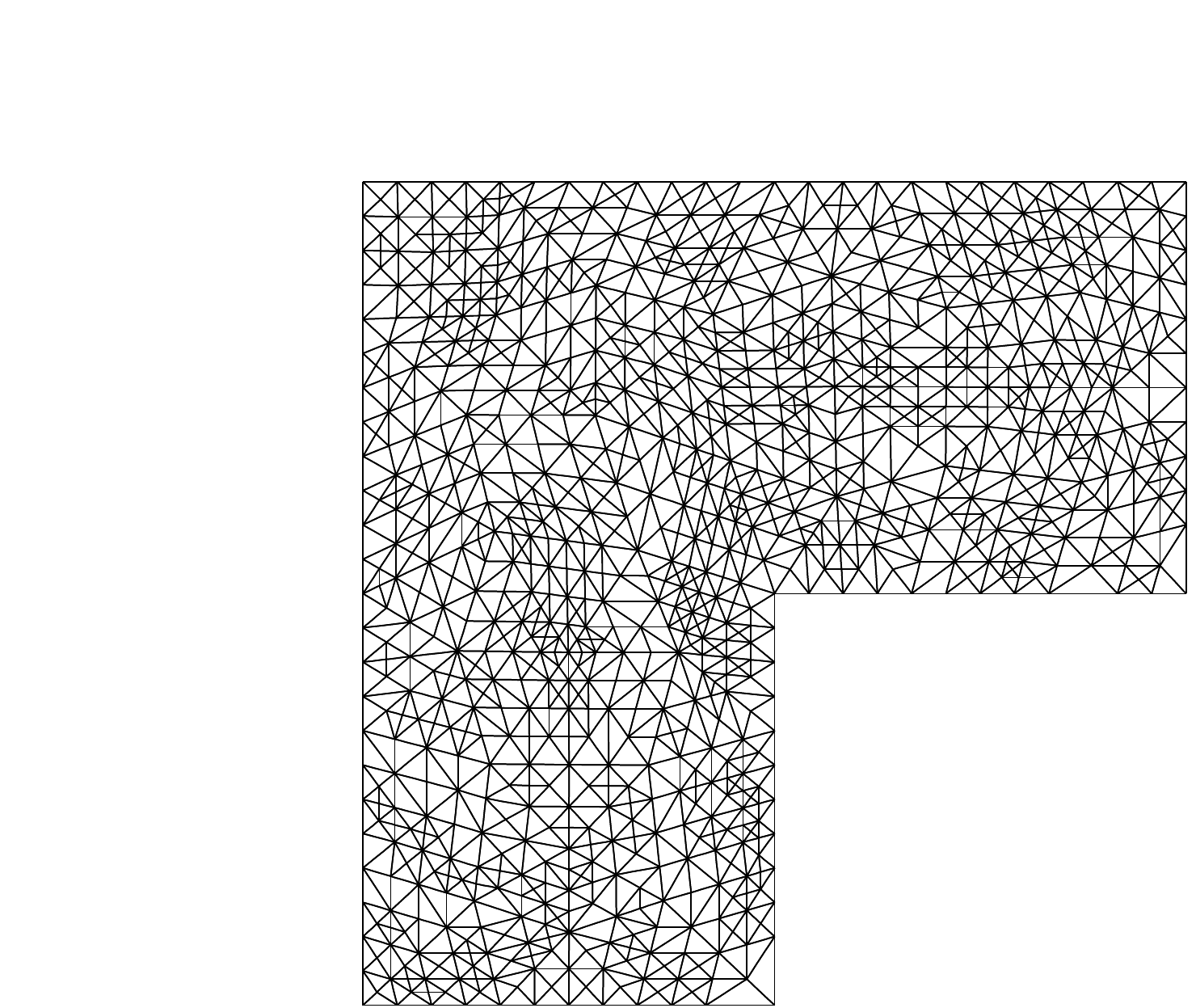}}
\caption{The left panel shows the initial mesh and the right panel
shows the adaptively computed mesh after 12 iterations when $\xi=2$ and $k=12$ using $p_K=7$ plane waves per element.}
\label{init_mesh_L}
\end{figure}

In Figure \ref{N5_s_MPH} we show detailed error results starting from the mesh in Fig. ~\ref{init_mesh_L} using the indicator
in Theorem \ref{MPH_L2} with $p_K=5$ plane waves per element.  The code uses the Doerfler marking strategy  with a bulk parameter $\theta=0.3$
(see the discussion in Section \ref{resultsII.sec}).
In these figures we show  the relative error in $L^2$ norm and the the indicator $\eta_{DG}$.  We scale the indicator so that the indicator and actual relative error are equal at the first step.  For reliability we then want the estimated
error to lie above the true error, and for efficiency we want the gap between the two curves to be small. Of course until the mesh is refined
sufficiently both efficiency and reliability may no be observed. In the right panel of each figure we show the ratio of the exact relative error to the error indicator and term this the ``efficiency ratio''.  The efficiency decreases markedly as the algorithm progresses.  
\begin{figure}
\begin{center}\begin{tabular}{ccc}
\resizebox{0.4\textwidth}{!}{\includegraphics{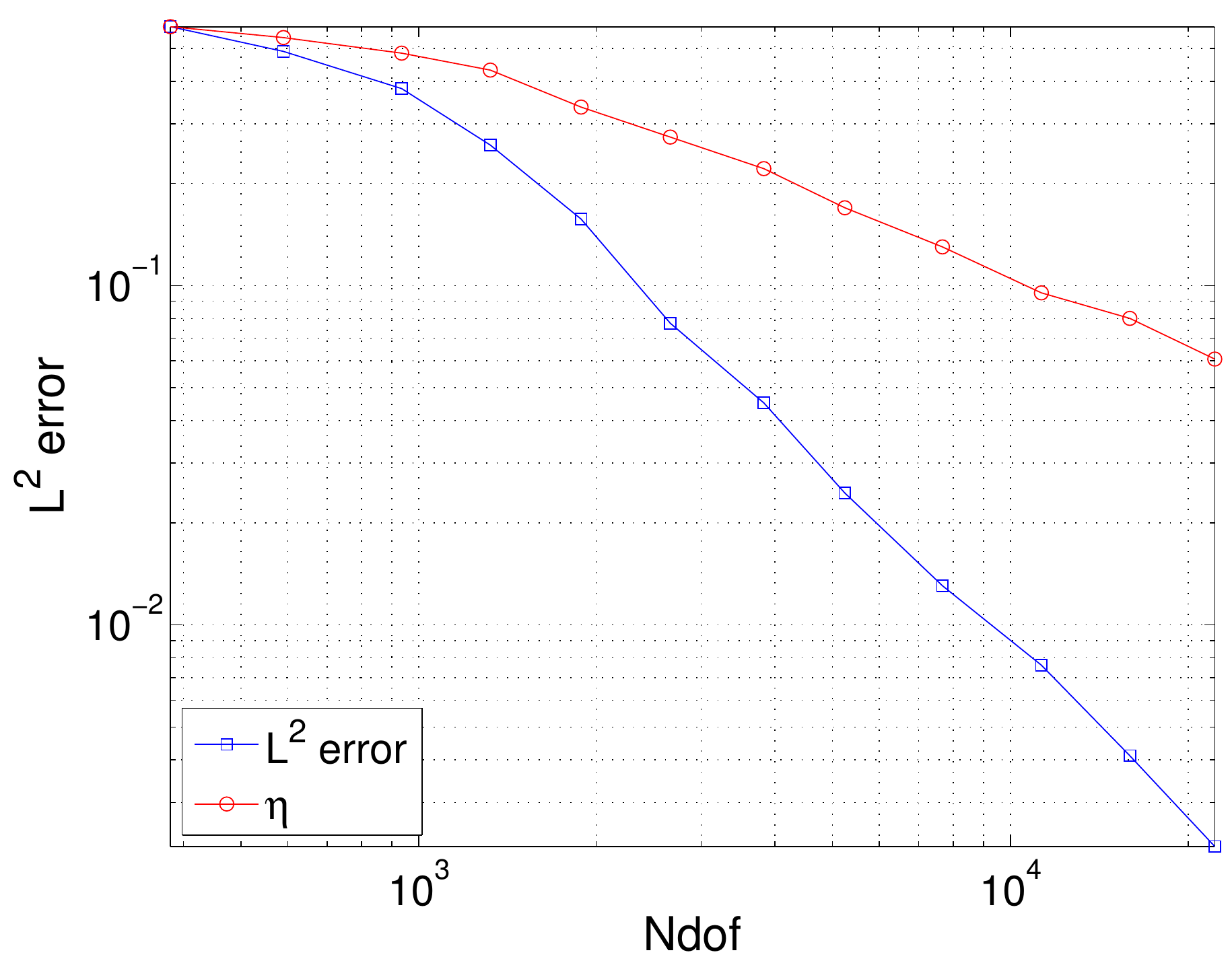}}&
\resizebox{0.4\textwidth}{!}{\includegraphics{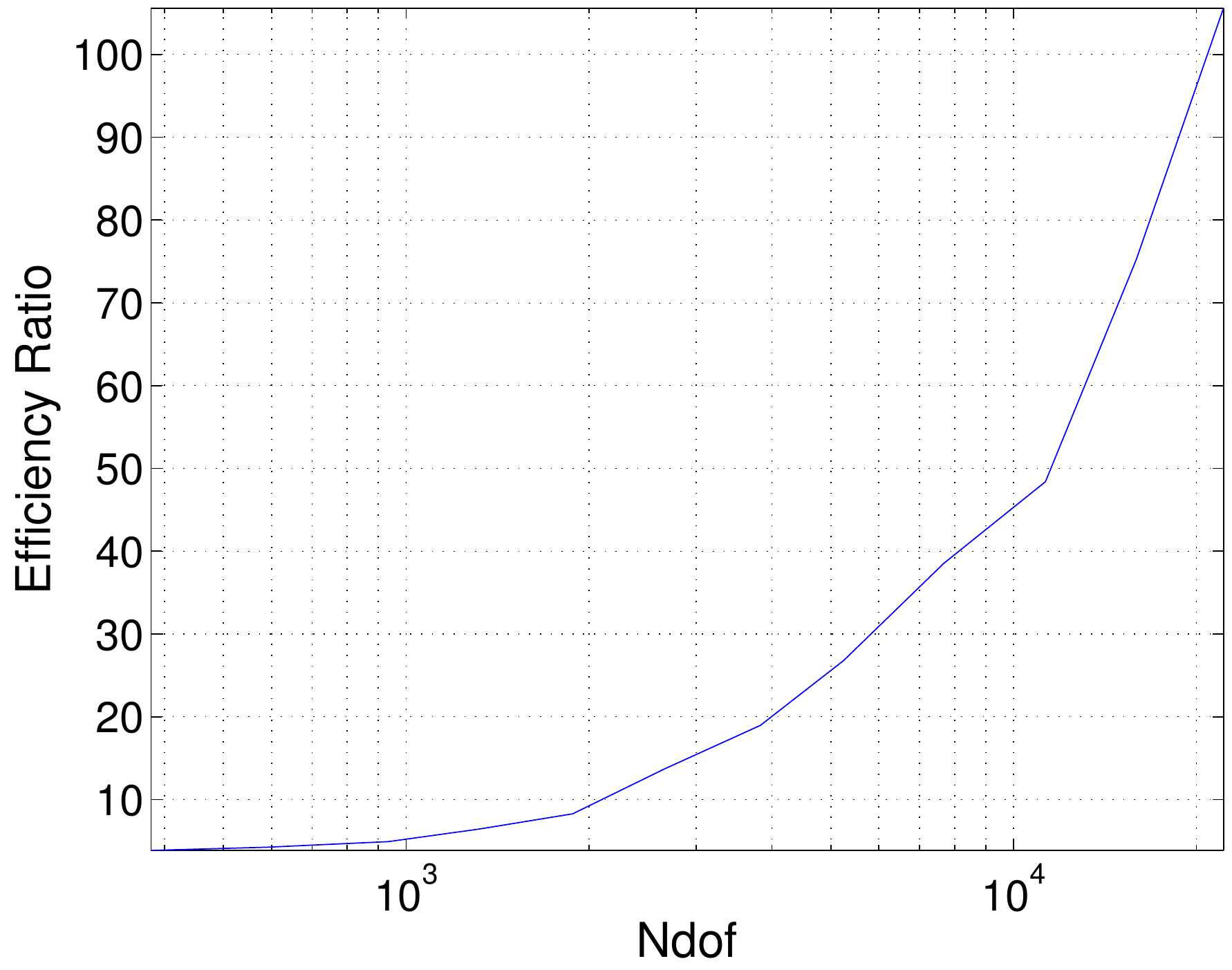}}
\end{tabular}
\end{center}
\caption{Adaptive refinement using $p_K=5$ waves per element and the indicator from Theorem \ref{MPH_L2}. Left panel: relative $L^2$ norm and indicator. Right panel: efficiency in the $L^2$ norm.  Right panel: relative $H^1$ norm behavior.  Although the indicator is reliable, it tends to overestimate the error so is not efficient.}
\label{N5_s_MPH}
\end{figure}

Results for $p^K=7$ waves per element are shown in Fig.~\ref{N7_s_MPH}.  Again mesh refinement does improve the solution error, but the efficiency of the indicator deteriorates rapidly as the mesh is refined.
\begin{figure}
\begin{center}\begin{tabular}{ccc}
\resizebox{0.4\textwidth}{!}{\includegraphics{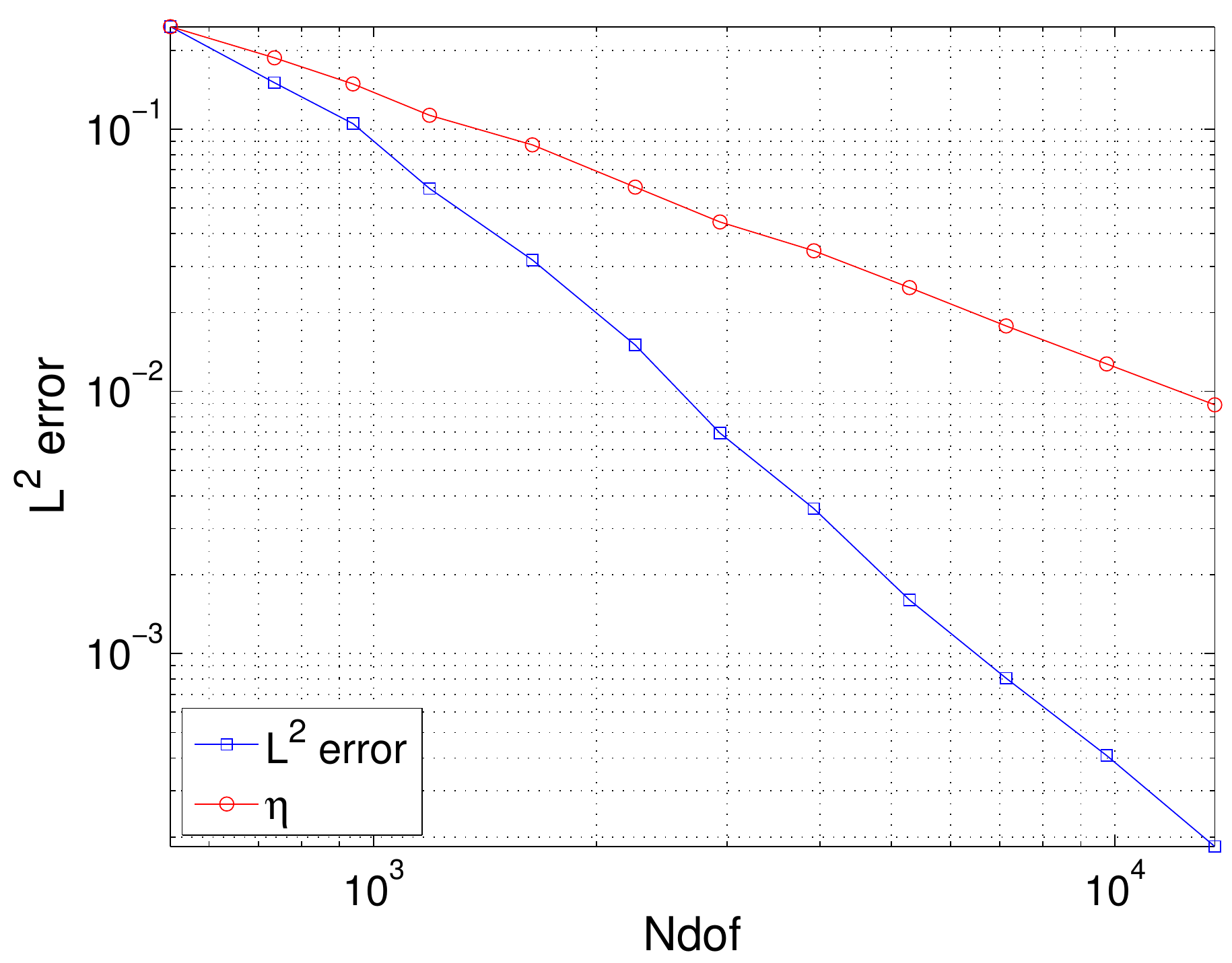}}&
\resizebox{0.4\textwidth}{!}{\includegraphics{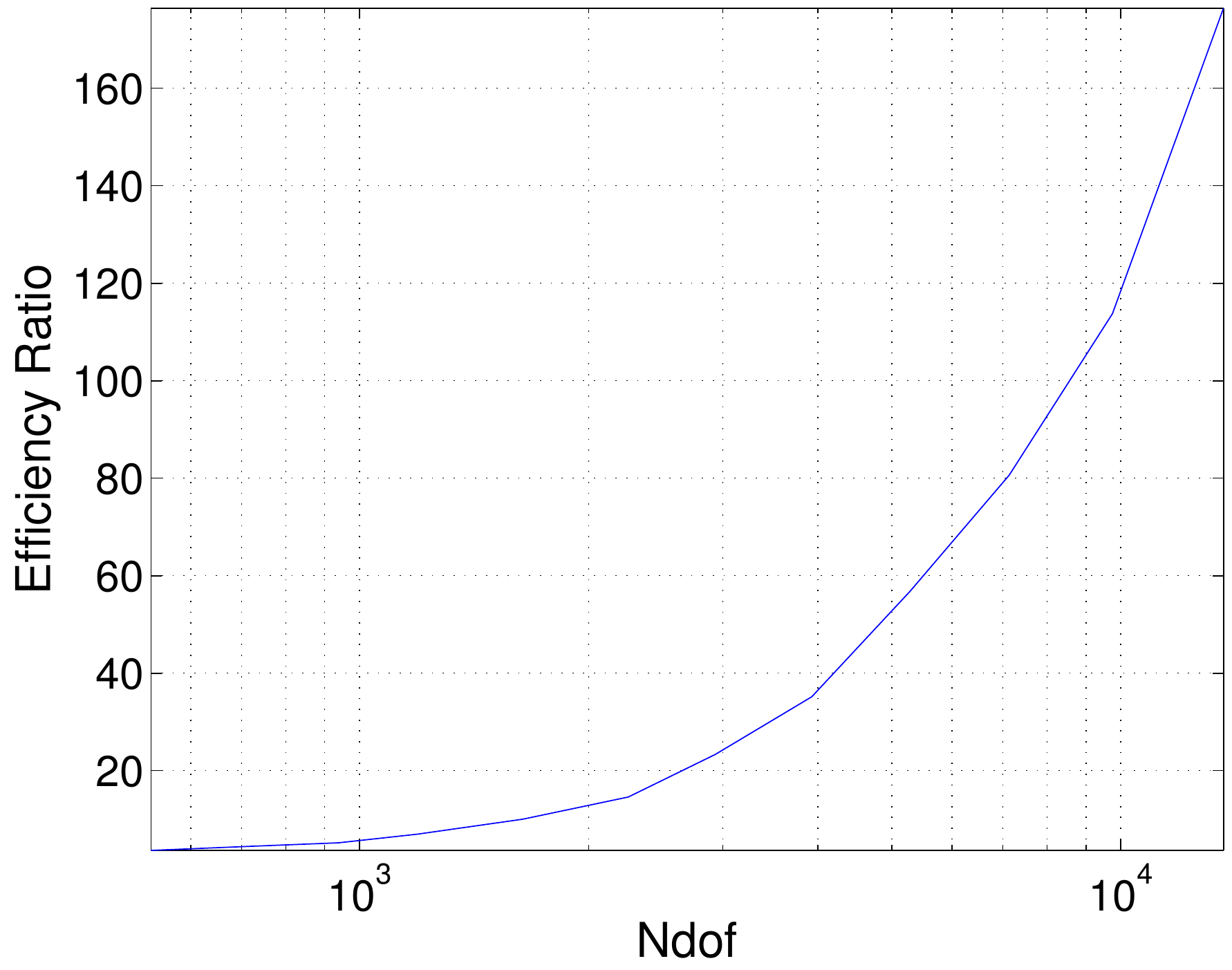}}
\end{tabular}
\end{center}
\caption{Adaptive refinement using $p_K=7$ waves per element and the indicator from Theorem \ref{MPH_L2}. Left panel: relative $L^2$ norm behavior. Right panel: efficiency in the $L^2$ The behavior of the indicator is similar to that for $p_K=5$ in Fig.~\ref{N5_s_MPH}.}
\label{N7_s_MPH}
\end{figure}

\section{A posteriori estimates II}
\label{sec:refest}

The results at the end of Section \ref{estimates.sec}  show that the basic error indicator in Theorem \ref{MPH_L2}, while reliable, is not efficient.  We therefore 
need to re-examine $h$-convergence theory to determine if a different weighting for the residual can be derived. 
In Section \ref{sec:estimate} we used special weights $\alpha$ and $\beta$ designed to allow the estimation of
${\cal{}G}(z)$ in terms of inverse powers of the global mesh size.  Because
of the upcoming results in this section, we no longer need inverse powers of the global mesh size in the estimate, and we now make the  choice that the parameters $\alpha$, $\beta$ and $\delta$
are positive constants independent of the mesh size, and that $\delta<1$.  Note that the choice $\alpha=\beta=\delta=1/2$ gives the classical UWVF \cite{buf07}.

We want an a posteriori error estimate for $\Vert u-u_{h}\Vert_{L^2(\Omega)}$ and will again use the solution
$z$ of the adjoint problem (\ref{zadj1})-(\ref{zadj3}).
By the adjoint consistency of the PWDG method (or direct calculation) we see that $z$ is sufficiently smooth to satisfy 
\[
A_h(w,z)=\int_{\Omega}w(\overline{u-u_{h}})\,dA,
\]
for all sufficiently smooth piecewise solutions $w$ of the Helmholtz equation ($w\in H^{3/2+s}(K)$ for some $s>0$ and each element suffices). 

Since $z\in H^{3/2+s}(\Omega)$, $s>0$, we can interpolate $z$ by a standard piecewise linear finite element function
denoted $z_h^c$.  We shall need to approximate $z_h^c$ by a function $z_{h,pw}$.  That this is possible follows from the
proof of Lemma 3.10 in \cite{git07} and is given in Lemma 6.3 in \cite{HMPS12}.  We give a slightly modified version:

\begin{lemma}\label{oldassumption1}
Suppose that on an element $K$  we are using $P_K\geq 4$ plane waves denoted $\{\psi_j^K\}_{j=1}^{P_K}$.  Then there are constants $\{\alpha^K_{i,j}\}$ (depending on $\kappa$) for $0\leq i\leq 2$ and $1\leq j\leq P_K$ such that
if $\mu_{pw}^i=\sum_{j=1}^{P_K}\alpha^K_{i,j}\psi_j^K$ and for all $x=(x_1,x_2)\in K$
\begin{eqnarray*}
|1-\mu^0_{pw}|=O(\kappa^2|x|^2), &\quad& |\nabla \mu_{pw}^0|=O(\kappa^2|x|)\\
|x_j-\mu^j_{pw}|=O(\kappa^2|x|^3), && |\nabla (x_j- \mu_{pw}^j)|=O(\kappa^2|x|^2),\;j=1,2,\\
|\nabla\nabla\mu^0_{pw}|=O(\kappa^2), && |\nabla\nabla \mu_{pw}^j|=O(\kappa^2|x|),\;j=1,2.
\end{eqnarray*}
\end{lemma}

\begin{remark}
This lemma is motivated by the following observation.  Suppose we are in one dimension and on the interval $[-h/2,h/2]$.  Let the basis functions be  $\psi_1(x)=\exp(i\kappa x)$ and $\psi_2(x)=\exp(-i\kappa x)$.  Then
\begin{eqnarray*}
\mu^0(x)&=&\frac{\psi_1(x)+\psi_2(x)}{2}=\cos(\kappa x)=1-O(\kappa^2x^2),\\
\mu^1(x)&=&\frac{\psi_1(x)-\psi_2(x)}{2i\kappa}=\frac{\sin(\kappa x)}{\kappa}=x-O(\kappa^2x^3),
\end{eqnarray*}
give a good approximation to linear functions for small $h$. Other estimates  follow accordingly.

If we select $p_K=3$ waves per element
\[
\psi_j(x,y)=\exp(i\kappa(\cos(\theta_j)x_1+\sin(\theta_j)x_2)),\quad j=1,2,3.
\]
where $\theta_j=(2\pi/3)(j-1)$, then we can compute coefficients $\alpha_{i,j}$ such that
\begin{eqnarray*}
\mu^0_{pw}&=&1+O(|x|^2\kappa^2),\\
\mu^j_{pw}&=&x_j+O(|x|^2\kappa),
\end{eqnarray*}
provided $-\sin(\theta_2)+\sin(\theta_3)-\cos(\theta_2)\,\sin(\theta3)+\sin(\theta_2)\,\cos(\theta_3)\not=0$.  But equality only occurs if $\theta_2=0$ or $\theta_2=\theta_3$, so this condition is satisfied.  These results are  not sufficient for 
the lemma, but could be used to derive an alternative indicator in this case.

If we choose $p_K=4$ we have
\[
\psi_1=\exp(i\kappa x_1),\;\psi_2(x)=\exp(i\kappa x_2),\;\psi_3(x)=\exp(-i\kappa x_1),\;\psi_4(x)=\exp(-i\kappa x_2).
\]
Then Lemma  \ref{oldassumption1} is satisfied because the approximation problem reduces to the one dimensional case. 

When $p_K=5$ with equally spaced directions a symbolic algebra package (Maple) again verifies the required asymptotics. Indeed this is the lowest order case considered in \cite{git07,HMPS12} where a general proof is given for $p_K\geq 5$.
\end{remark}

Now suppose we are on a triangle $K$ and $z_h^c=\sum_{j=1}^3z(\mathbf{a}^{K}_j)\lambda^K_j$ where $\lambda^K_j$ is the
$j$th barycentric coordinate function and $\mathbf{a}^K_j$ is the $j$th vertex of the triangle.  We can assume that the centroid is at the origin by translation.  Then, $\lambda_j^K=a_j^K+b_j^Kx_1+c_j^Kx_2$ and $a_j^K=O(1)$, $b_j^K=O(1/h_K)$ and
$c_j^K=O(1/h_K)$.  Replacing 1, $x_1$ and $x_2$ by the above plane wave approximations $\mu_{pw}^j$, $j=0,1,2$, and denoting this approximation
by $\lambda_{pw,j}^K$ we have:
\begin{lemma} For $p_K\geq 4$ we have the following estimates for all $x\in K$,
\[
|\lambda_j^K-\lambda_{pw,j}^K|+h_K|\nabla (\lambda_j^K-\lambda_{pw,j}^K)|+h^2_K|\nabla\nabla (\lambda_j^K-\lambda_{pw,j}^K)|\leq C(h_K^2\kappa^2)
\]
\end{lemma}
\begin{remark} This lemma is essentially used in the proof of Lemma 3.10 in \cite{git07}.
\end{remark}
\begin{proof} To estimate $\lambda^K_j-\lambda^K_{pw,j}$ on $K$  we note that
\begin{eqnarray*}
|\lambda^K_j-\lambda^K_{pw,j}|&=&|a_j^K(1-\mu_{pw}^0)+b_j^K(x_1-\mu_{pw}^1)+c_j^K(x_2-\mu_{pw}^2)|\\
&\leq &C(k^2|x|^2+(1/h_K)(k^2h_K^3))\leq C \kappa^2h_K^2.
\end{eqnarray*}
The proof of the other estimates proceeds similarly.
\end{proof}

Using the plane wave approximation to the barycentric coordinate functions element by element, we can then 
construct an approximate interpolant $z_{h,pw}\in V_h$.
We  need to estimate $z_h^c-z_{h,pw}^c$ and $\nabla_h(z_h^c-z_{h,pw})$ on edges in the mesh.  This is done in the next lemma
\begin{lemma}\label{zdef} Suppose $e$ is an edge between two elements $K_1$ and $K_2$.
Under  the standing assumptions on the mesh, there is a constant $C$ independent of $e$, $z$, $K_j$, $h_{K_j}$, $j=1,2$ and $\kappa$ such that
\begin{eqnarray*}
\Vert \avg{z_h^c-z_{h,pw}}\Vert^2_{L^2(e)}&\leq& C\sum_{j=1}^2h_{K_j}^5\kappa^4\Vert z\Vert_{L^{\infty}(K_j)}^2,\\
\Vert \avg{\nabla_h(z_h^c-z_{h,pw})}\Vert^2_{L^2(e)}&\leq& C\sum_{j=1}^2h_{K_j}^3\kappa^4\Vert z\Vert_{L^{\infty}(K_j)}^2.
\end{eqnarray*}
\end{lemma}

\begin{proof} Using the standard trace estimate
\[
\Vert \avg{z_h^c-z_{h,pw}}\Vert^2_{L^2(e)}\leq C\sum_{j=1}^2\left[
\frac{1}{h_{K_j}}\Vert z_h^c-z_{h,pw}\Vert_{L^2(K_j)}^2+
h_{K_j}\Vert \nabla(z_h^c-z_{h,pw})\Vert_{L^2(K_j)}^2\right].
\]
%
Using the estimates for the basis functions in the previous lemma, on each triangle $K_j$,
\begin{eqnarray*}
\int_{K_j}|z_h^c-z_{h,pw}|^2\,ds&=&\int_{K_j}\left|\sum_{\ell=1}^3z(\mathbf{a}^{K_j}_\ell)(\lambda^K_\ell-\lambda^K_{h,\ell})\right|^2\,ds
\leq Ch_{K_j}^6\kappa^4\Vert z\Vert_{L^\infty(K_j)}^2.
\end{eqnarray*}
In the same way 
\begin{eqnarray*}
\int_{K_j}|\nabla(z_h^c-z_{h,pw})|^2\,ds&=&\int_{K_j}\left|\sum_{\ell=1}^3z(\mathbf{a}_\ell)\nabla(\lambda^{K_j}_\ell-\lambda^{K_j}_{h,\ell})\right|^2\,ds
\leq Ch_{K_j}^4\kappa^4\Vert z\Vert_{L^\infty(K_j)}^2.
\end{eqnarray*}
So
\[
\Vert \avg{z_h^c-z_{h,pw}}\Vert^2_{L^2(e)}\leq C\sum_{j=1}^2h_{K_j}^5\kappa^4\Vert z\Vert_{L^{\infty}(K_j)}.
\]

Using the standard trace estimate again (noting that the basis functions are piecewise smooth)
\[
\Vert \avg{\nabla_h(z_h^c-z_{h,pw})}\Vert^2_{L^2(e)}\leq C\sum_{j=1}^2\left[
\frac{1}{h_{K_j}}\Vert \nabla(z_h^c-z_{h,pw})\Vert_{L^2(K_j)}^2+
h_{K_j}\Vert \nabla\nabla(z_h^c-z_{h,pw})\Vert_{L^2(K_j)}^2\right].
\]
%
Using the estimates for the basis functions in the previous lemma and noting that since $z_h^c$ is linear, $\nabla\nabla z_h^c=0$,
\begin{eqnarray*}
\int_{K_j}|\nabla\nabla(z_h^c-z_{h,pw})|^2\,dA&=&\int_{K_j}\left|\sum_{\ell=1}^3z(\mathbf{a}^{K_j}_\ell)(\nabla\nabla\lambda^{K_j}_{h,\ell})\right|^2
\leq Ch_{K_j}^2\kappa^4\Vert z\Vert_{L^\infty(K_j)}^2.
\end{eqnarray*}
So $\Vert \avg{\nabla_h(z_h^c-z_{h,pw})}\Vert^2_{L^2(e)}\leq C\sum_{j=1}^2h_{K_j}^3\kappa^4\Vert z\Vert_{L^{\infty}(K_j)}.$
This completes the proof.
\end{proof}

Now, since  $z_{h,pw}\in V_{h}$, by Galerkin orthogonality,
\[
\int_{\Omega}(u-u_{h})(\overline{u-u_{h}})\,dA=A_h(u-u_{h},z)=A_h(u-u_{h},z-z_{pw,h}).
\]
We first add and subtract the finite element piecewise linear interpolant on the mesh denoted $z^c_h$.  This is not in the plane wave subspace $V_{h}$ so no terms simplify:
\[
A_h(u-u_{h},z-z_{h})=A_h(u-u_{h},z-z^c_{h})+A_h(u-u_{h},z_h^c-z_{h,pw})
\]
We can now analyze the two terms on the right hand side above.
Using (\ref{Ahap}), the first term can be written
\begin{eqnarray*}
\lefteqn{A_h(u-u_{h},z-z_h^c)}\nonumber\\
&=&\int_{\cEI}\jmp{\nabla_h(u-u_{h})}\cdot\avg{\overline{z-z_h^c}}\,ds-\int_{\cEI}\jmp{ (u-u_{h})}\cdot\avg{\nabla_h(\overline{z-z_h^c})}\,ds\nonumber\\&&
-\frac{1}{i\kappa}\int_{\cEI}\beta\jmp{\nabla_h (u-u_{h})}\jmp{\nabla_h(\overline{z-z_h^c})}\,ds
+{i\kappa }\int_{\cEI}\alpha\jmp{ (u-u_{h})}\cdot\jmp{(\overline{z-z_h^c})}\,ds\nonumber\\&&
+\int_{\cEA}(1-\delta)\left[\frac{\partial (u-u_{h})}{\partial \nu}-i\kappa(u-u_{h})\right](\overline{z-z_h^c})\,ds\\&&-\frac{\delta}{i\kappa}
\int_{\cEA}\left[\frac{\partial (u-u_{h})}{\partial \nu}-i\kappa(u-u_{h})\right] \nabla_h(\overline{z-z_h^c})\cdot\nu\,ds\\
&&+\int_{\cED}(u-u_{h})(i\kappa\alpha(\overline{z-z_h^c})-\nabla_h(\overline{z-z_h^c})\cdot\nu)\,ds.
\end{eqnarray*}
Note that $z=z_h^c=0$ on $\cED$ and $\jmp{z-z_h^c}=0$ on $\cEI$.  In addition $u=0$ on $\cED$, and $u$ and its normal derivative are continuous across interior edges. Finally $u$ also satisfies the Dirichlet and impedance boundary conditions. So the above expression simplifies as follows:
\begin{eqnarray}
A_h(u-u_{h},z-z_h^c)&=&
-\int_{\cEI}\jmp{\nabla_hu_{h}}\cdot \avg{\overline{z-z_h^c}}\,ds+\int_{\cEI}\jmp{ u_{h}}\cdot\avg{\nabla_h(\overline{z-z_h^c})}ds\nonumber\\&&
+\frac{1}{i\kappa}\int_{\cEI}\beta\jmp{\nabla_h u_{h}}\jmp{\nabla_h(\overline{z-z_h^c})}\,ds
\nonumber\\&&
+\int_{\cEA}(1-\delta)\left[g_A-\frac{\partial u_{h}}{\partial\nu}+i\kappa u_{h}\right](\overline{z-z_h^c})\,ds\nonumber
\\&&-\frac{\delta}{i\kappa}
\int_{\cEA}\left[g_A-\frac{\partial u_{h}}{\partial \nu}+i\kappa u_{h}\right] \nabla_h(\overline{z-z_h^c})\cdot\nu\,ds\nonumber\\
&&+\int_{\cED}u_{h}\nabla_h(\overline{z-z_h^c})\cdot\nu\,ds.\label{t1}
\end{eqnarray}
Terms involving $z-z^c_{h}$ (non-differentiated) can be estimated via the standard trace estimate.  First
\begin{eqnarray*}
\lefteqn{\left|-\int_{\cEI}\jmp{\nabla_h u_{h}}\cdot\avg{\overline{z-z^c_{h}}}\,ds
+\int_{\cEA}(1-\delta)\left[g_A-\frac{\partial u_{h}}{\partial \nu}+i\kappa u_{h}\right](\overline{z-z^c_{h}})\,ds\right|}\\
&\leq&\sum_{e\in\cEI}\left[\Vert\beta^{-1/2}\avg{{z-z^c_{h}}}\Vert_{L^2(e)}\Vert\beta^{1/2}\jmp{\nabla_hu_{h}}\Vert_{L^2(e)}
\right]\\&&
+\sum_{e\in\cEA}\Vert(1-\delta)^{1/2}\left[g_A-\frac{\partial u_{h}}{\partial n}+i\kappa u_{h}\right]\Vert_{L^2(e)}\Vert(1-\delta)^{1/2}( {z-z^c_{h}})\Vert_{L^2(e)}.\\
\end{eqnarray*}
Using the usual  trace inequality (\ref{trace}), let $e$ be an edge in the mesh shared by elements $K_1$ and $K_2$ then
\begin{eqnarray*}
\Vert\beta^{-1/2}\avg{z-z^c_{h}}\Vert_{L^2(e)}&\leq& C\sum_{j=1}^2\left[\frac{1}{h^{1/2}_{K_j}}\Vert
z-z_h^c\Vert_{L^2(K_j)}+h^{1/2}_{K_j}\Vert \nabla (z-z_h^c)\Vert_{L^2(K_j)}\right]\\
&\leq &
C\sum_{j=1}^2h^{1+s}_{K_j}\vert
z\vert_{H^{3/2+s}(K_j)}.
\end{eqnarray*}
where we have also used an error estimate for the interpolant.  

The same estimate holds for the jump in $z-z_h^c$. Using the Cauchy-Schwarz inequality we arrive at
\begin{eqnarray*}
\lefteqn{\left|-\int_{\cEI}\jmp{\nabla_h u_{h}}\cdot\avg{\overline{z-z^c_{h}}}\,ds
+\int_{\cEA}(1-\delta)\left[g_A-\frac{\partial u_{h}}{\partial\nu}+i\kappa u_{h}\right](\overline{z-z^c_{h}})\,ds\right|}\\
&\leq&\left[\Vert\beta^{1/2}h_e^{1+s}\jmp{\nabla_hu_{h}}\Vert_{L^2(\cEI)}
+\Vert(1-\delta)^{1/2}h_e^{1+s}\left[g_A-\frac{\partial u_{h}}{\partial n}+iku_{h}\right]\Vert_{L^2(\cEA)}\right]\vert z\vert_{H^{3/2+s}(\Omega)}.
\end{eqnarray*}

Now we must perform the same estimate for terms in (\ref{t1}) involving derivatives of $z-z_h^c$.
\begin{eqnarray*}
&&\left|\int_{\cEI}\jmp{ u_{h}}\avg{\nabla_h(\overline{z-z^c_{h}})}\,ds-\frac{1}{i\kappa}\int_{\cEI}\beta\jmp{\nabla_h u_{h}}\jmp{\nabla_h(\overline{z-z^c_{h}})}\,ds\right.\\&&\left.-\frac{\delta}{i\kappa}
\int_{\cEA}\left[g_A-\frac{\partial u_{h}}{\partial \nu}+iku_{h}\right]\nabla_h (\overline{z-z^c_{h}})\cdot\nu\,ds
+\int_{\cED}u_{h}\nabla_h(\overline{z-z_h^c})\cdot\nu\,ds\right|\\
&\leq&\sum_{e\in \cEI}\Vert \alpha^{1/2}\jmp{ u_{h}}\Vert_{L^2(e)}\Vert\alpha^{-1/2}\avg{\nabla_h(z-z^c_{h})}\Vert_{L^2(e)}\\&&+\sum_{e\in \cEI}\frac{1}{\kappa}\Vert\beta^{1/2}\jmp{\nabla_h u_{h}}\Vert_{L^2(e)}\Vert\beta^{1/2}\jmp{\nabla_h({z-z^c_{h}})}\Vert_{L^2(e)}\\&&+\sum_{e\in \cEA}\frac{1}{\kappa}
\Vert\delta^{1/2}\left[g_A-\frac{\partial u_{h}}{\partial \nu}+i\kappa u_{h}\right]\Vert_{L^2(e)}\Vert\delta^{1/2}\frac{\partial ({z-z^c_{h}})}{\partial \nu}\Vert_{L^2(e)}\\
&&+\sum_{e\in \cED}\Vert \alpha^{1/2} u_{h}\Vert_{L^2(e)}\Vert \alpha^{-1/2} \nabla_h(z-z_h^c)\cdot\nu\Vert_{L^2(e)}.
\end{eqnarray*}
We proceed as for the previous estimates. On an edge $e$ between $K_1$ and $K_2$ we have, using the trace estimate (\ref{trace2}):
\[
\Vert\alpha^{-1/2}\avg{\nabla_h(z-z^c_{h})}\Vert_{L^2(e)}
\leq C\sum_{j=1}^2\left[\frac{1}{h_{K_j}^{1/2}}\Vert\nabla(z-z_h^c)\Vert_{L^2(K_j)}
+h_{K_j}^{s}\vert \nabla(z-z_h^c)\vert_{H^{1/2+s}(K_j)}\right].
\]
Since $z_h^c$ is piecewise linear
$\vert \nabla(z-z_h^c)\vert_{H^{1/2+s}(K_j)}=\vert \nabla z \vert_{H^{1/2+s}(K_j)}$.  Using usual
estimates for the interpolant:
\[
\Vert\alpha^{-1/2}\avg{\nabla_h(z-z^c_{h})}\Vert_{L^2(e)}
\leq C\sum_{j=1}^2h_{K_j}^{s}\vert z\vert_{H^{3/2+s}(K_j)}.
\]
Other average and jump terms can be estimated in the same way.  We arrive at
\begin{eqnarray*}
&&\left|\int_{\cEI}\jmp{ u_{h}}\avg{\nabla_h(\overline{z-z^c_{h}})}\,ds+\frac{1}{i\kappa}\int_{\cEI}\beta\jmp{\nabla_h u_{h}}\jmp{\nabla_h(\overline{z-z^c_{h}})}\,ds\right.\\&&\left.-\frac{\delta}{i\kappa}
\int_{\cEA}\left[g_A-\frac{\partial u_{h}}{\partial \nu}+i\kappa u_{h}\right]\nabla_h (\overline{z-z^c_{h}})\cdot\nu\,ds
+\int_{\cED}u_{h}\nabla_h(\overline{z-z_h^c})\cdot\nu\,ds\right|\\
&\leq&C\left[\Vert \alpha^{1/2}h_e^{s}\jmp{ u_{h}}\Vert_{L^2(\cEI)}+\frac{1}{\kappa}\Vert\beta^{1/2}h_e^{s}\jmp{\nabla_h
u_{h}}\Vert_{L^2(\cEI)}\right.\\&&\left.+\frac{1}{\kappa}
\Vert\delta^{1/2}h_e^{s}\left[g_A-\frac{\partial u_{h}}{\partial \nu}+i\kappa u_{h}\right]\Vert_{L^2(\cEA)}+\Vert h_e^s \alpha^{1/2} u_{h}\Vert_{L^2(\cED)}\right]
\vert z\vert_{H^{3/2+s}(\Omega)}.
\end{eqnarray*}
\begin{lemma}  For $h$ small enough, under the conditions on the mesh stated in Section~\ref{notation.sec}, there exists a constant $C$ such that
\begin{eqnarray*}
\lefteqn{|A_h(u-u_{h},z-z^c_{h})|\leq C\left[\Vert\beta^{1/2}h_e^{1+s}\jmp{\nabla_hu_{h}}\Vert_{L^2(\cEI)}\right.}
\\&&+\Vert(1-\delta)^{1/2}h_e^{1+s}\left[g_A-\frac{\partial u_{h}}{\partial \nu}+i\kappa u_{h}\right]\Vert_{L^2(\cEA)}\\
&&+\left.\Vert \alpha^{1/2}h_e^{s}\jmp{ u_{h}}\Vert_{L^2(\cEI)}+\frac{1}{\kappa}\Vert\beta^{1/2}h_e^{s}\jmp{\nabla_h
u_{h}}\Vert_{L^2(\cEI)}\right.\\&&\left.+\frac{1}{\kappa}
\Vert\delta^{1/2}h_e^{s}\left[g_A-\frac{\partial u_{h}}{\partial \nu}+i\kappa u_{h}\right]\Vert_{L^2(\cEA)}+
\Vert h_e^{s}\alpha^{1/2} u_{h}\Vert_{L^2(\cED)}\right]
\vert z\vert_{H^{3/2+s}(\Omega)}.
\end{eqnarray*}
Here $C$ is independent of the mesh and the solution.
\end{lemma}

It remains to estimate $A_h(u-u_{h},z_h^c-z_{h,pw})$. Recall that  $z_{h,pw}$ is defined element by element according to Lemma \ref{zdef} and
\begin{eqnarray*}
A_h(u-u_{h},z_h^c-z_{h,pw})
&=&-\int_{\cEI}\jmp{\nabla_hu_{h}}\cdot \avg{\overline{z^c_h-z_{h,pw}}}\,ds+\int_{\cEI}\jmp{ u_{h}}\avg{\nabla_h(\overline{z^c_h-z_{h,pw}})}\,ds\\&&
+\frac{1}{i\kappa}\int_{\cEI}\beta\jmp{\nabla_h u_{h}}\jmp{\nabla_h(\overline{z_h^c-z_{h,pw}})}\,ds
-{i\kappa}\int_{\cEI}\alpha\jmp{ u_{h}}\jmp{(\overline{z^c_h-z_{h,pw}})}\,ds\\&&
+\int_{\cEA}(1-\delta)\left[g_A-\frac{\partial u_{h}}{\partial \nu}+i\kappa u_{h}\right](\overline{z_h^c-z_{h,pw}})\,ds\\&&-\frac{\delta}{i\kappa}
\int_{\cEA}\left[g_A-\frac{\partial u_{h}}{\partial \nu}+i\kappa u_{h}\right]\nabla_h (\overline{z_h^c-z_{h,pw}})\cdot\nu\,ds\\
&&-\int_{\cED}u_{h}(i\kappa\alpha(\overline{z_h^c-z_{h,pw}})-\nabla_h(\overline{z_h^c-z_{h,pw}})\cdot\nu)\,ds.
\end{eqnarray*}
As before considering an edge $e$ between elements $K_1$ and $K_2$ and using the fact that $\beta$ is constant:
\begin{eqnarray*}
\left|\int_{e}\avg{\overline{z^c_h-z_{h,pw}}}\cdot\jmp{\nabla_h(u-u_{h})}\,ds\right|&\leq &\Vert \beta^{1/2} h^{3/2}_e\jmp{\nabla_h
u_{h}}\Vert_{L^2(e)}\Vert\beta^{-1/2} h_e^{-3/2}\avg{z^c_h-z_{h,pw}}\Vert_{L^2(e)}\\
&\leq &C\Vert \beta^{1/2} h_e^{3/2} \jmp{\nabla_h
u_{h}}\Vert_{L^2(e)}\sqrt{\sum_{j=1}^2h_e^{-3}h_{K_j}^5\kappa^4\Vert z\Vert_{L^{\infty}(K_j)}^2}\\
&\leq & C\kappa^2\Vert \beta^{1/2} h_e^{3/2} \jmp{\nabla_h
u_{h}}\Vert_{L^2(e)}\sqrt{\sum_{j=1}^2h_{K_j}^2\Vert z\Vert_{L^{\infty}(K_j)}^2},
\end{eqnarray*}
Now adding over all edges in $\cEI$
\begin{eqnarray*}
\left|\int_{\cEI}\avg{\overline{z^c_h-z_{h,pw}}}\cdot\jmp{\nabla_h(u-u_{h})}\,ds\right|&\leq & C\kappa^2
\Vert \beta^{1/2} h_e^{3/2} \jmp{\nabla_h
u_{h}}\Vert_{L^2(\cEI)}\sqrt{\sum_{K\in T_h}h_{K_j}^2\Vert z\Vert_{L^{\infty}(K)}^2}  \\&\leq&
C\kappa^2
\Vert \beta^{1/2} h_e^{3/2} \jmp{\nabla_h
u_{h}}\Vert_{L^2(\cEI)}\Vert z\Vert_{L^{\infty}(\Omega)}\\
&\leq & C\kappa^2
\Vert \beta^{1/2} h_e^{3/2} \jmp{\nabla_h
u_{h}}\Vert_{L^2(\cEI)}\Vert z\Vert_{H^{s+3/2}(\Omega)}
\end{eqnarray*}
where we have used the Sobolev embedding theorem to estimate $\Vert z\Vert_{L^{\infty}(\Omega)}$.
Similarly
\begin{eqnarray*}
\left|\int_{\cEI}\jmp{u_{h}}\avg{\nabla_h(\overline{z^c_h-z_{h,pw}})}\,ds\right|&\leq &
\Vert\alpha^{1/2}h_e^{1/2}\jmp{ u_{h}}\Vert_{L^2(e)}
\Vert h_e^{-1/2}\avg{ \nabla_h(z_h^c-z_{h,pw})}\Vert_{L^2(e)}\\
&\leq&C\kappa^2\Vert h_e^{1/2}\alpha^{1/2}\jmp{ u_{h}}\Vert_{L^2(e)}\sqrt{\sum_{j=1}^2h_{K_j}^2\Vert z\Vert_{L^{\infty}(K_j)}^2}
\end{eqnarray*}
Proceeding as above we can estimates each of the terms in the expansion of $A_h$.

\begin{lemma}
Under the assumptions on the mesh in Section~\ref{notation.sec} there is a constant
$C$ independent of $h$, $u$ and $u_{h}$ such that
\begin{eqnarray*}
|A_h(u-u_{h},z_h^c-z_{h,pw})|&\leq&C\left[\Vert\beta^{1/2}h_e^{3/2}\jmp{\nabla_hu_{h}}\Vert_{L^2(\cEI)}+
\Vert \alpha^{1/2}h_e^{3/2}\jmp{u_{h}}\Vert_{L^2(\cEI)}\right.\\&&
+\Vert(1-\delta)^{1/2}h_e^{3/2}\left[g_A-\frac{\partial u_{h}}{\partial \nu}+i\kappa u_{h}\right]\Vert_{L^2(\cEA)}\\
&&+\left.\Vert \alpha^{1/2}h_e^{1/2}\jmp{ u_{h}}\Vert_{L^2(\cEI)}+\frac{1}{\kappa}\Vert\beta^{1/2}h_e^{1/2}\jmp{\nabla_h
u_{h}}\Vert_{L^2(\cEI)}\right.\\&&\left.+\frac{1}{k}
\Vert\delta^{1/2}h_e^{1/2}\left[g_A-\frac{\partial u_{h}}{\partial \nu}+i\kappa u_{h}\right]\Vert_{L^2(e)}\right]
\Vert z\Vert_{H^{3/2+s}(\Omega)}.
\end{eqnarray*}
\end{lemma}
Since $s\leq 1/2$  and using the estimates for $\Vert z\Vert_{H^{3/2+s}(\Omega)}$ from Section \ref{estimates.sec} we obtain:
\begin{theorem}
\label{apost2.thm} 
Under  the assumptions on the mesh in Section~\ref{notation.sec}, for any sufficiently fine mesh there is a constant
$C$ independent of $h$, $u$ and $u_{h}$ such that\begin{eqnarray*}
\Vert  u -u_{h}\Vert_{L^2(\Omega)}&\leq&C\left[\Vert \alpha^{1/2}h_e^{s}\jmp{ u_{h}}\Vert_{L^2(\cEI)}+\frac{1}{\kappa}\Vert\beta^{1/2}h_e^{s}\jmp{\nabla_h
u_{h}}\Vert_{L^2(\cEI)}\right.\\&&\left.+\frac{1}{\kappa}
\Vert\delta^{1/2}h_e^{s}\left[g_A-\frac{\partial u_{h}}{\partial \nu}+i\kappa u_{h}\right]\Vert_{L^2(\cEA)}
+\Vert\alpha^{1/2}h_e^{s}u_{h}\Vert_{L^2(\cED)}\right]
\end{eqnarray*}
\end{theorem}
\begin{remark}
The right hand side is now a new a posteriori error indicator for PWDG. Note there is no longer an overall factor of $h^{-1/2}$
compared to the estimate in Theorem \ref{MPH_L2}.  We have not traced the dependence of the constant $C$ on $\kappa$ and  $d_\Omega$ but this could be carried out (however the marking strategy is independent of the overall constant). In practice we find the choice $s=0$ gives the a reliable but pessimistic indicator.
\end{remark}

\section{Numerical Results}
\label{resultsII.sec}

We now test the new residual estimators derived in the previous section using the UWVF choice of parameters $\alpha=\beta=\delta=1/2$. 
In the following numerical tests we iteratively apply the classical refinement sequence
\newline\begin{center}{SOLVE \; - \; ESTIMATE \; - \; MARK \; - \; REFINE}\end{center}

\vspace{8pt}\noindent In the ESTIMATE phase of the following experiments we 
rank the effective contributions to the righthand side of the
a posteriori bound given in Theorem \ref{apost2.thm} from the element $K$ 
using a proxy for the residual formula
\begin{eqnarray*}
\eta_{K} & = & \Vert \alpha^{1/2}h_e^{s}\jmp{ u_{h}}\Vert_{L^2(\partial K)}^2
+\frac{1}{\kappa^2}\Vert\beta^{1/2}h_e^{s}\jmp{\nabla_h u_{h}}\Vert_{L^2(K)}^2\\
& & + \frac{1}{\kappa^2} \Vert\delta^{1/2}h_e^{s}\left[g_A-\frac{\partial u_{h}}{\partial \nu}+i\kappa u_{h}\right]\Vert_{L^2(\partial K)}^2
+\Vert\alpha^{1/2}h_e^{s}u_{h}\Vert_{L^2(K)}^2.
\end{eqnarray*}
Following D\"{o}rfler \cite{dorfler1996convergent} the elements 
responsible for the top $\theta$ fraction of $\eta := \sum_{K} \eta_{K}$ are
marked for refinement in the MARK phase. In the REFINE phase we use a recursive longest edge bisection
\cite{mitchell1989comparison}  to produce a new mesh with guaranteed
lower bounds for the smallest element angles. The recursive longest edge bisection algorithm is 
chosen because it propagates the refinement beyond the elements marked in the MARK phase to achieve this goal.

We start with several results for the regular Bessel function solution considered in Section~\ref{sec:estimate} and defined by equation (\ref{reg_exact}).  Since we are on the L-shaped domain we choose $s=1/6$.
These results can be compared to the results in Figs.~\ref{N5_s_MPH} and \ref{N7_s_MPH}. Although the efficiency
shown in the right hand column for each choice of $p_K$ still deteriorates for the $L^2$ norm as the mesh is refined, the rate of rise is less compared to the previous indicator.  In addition the efficiency of the indicator improves for larger $p_K$.  

\begin{figure}
\begin{center}
\begin{tabular}{ccc}
\resizebox{0.4\textwidth}{!}{\includegraphics{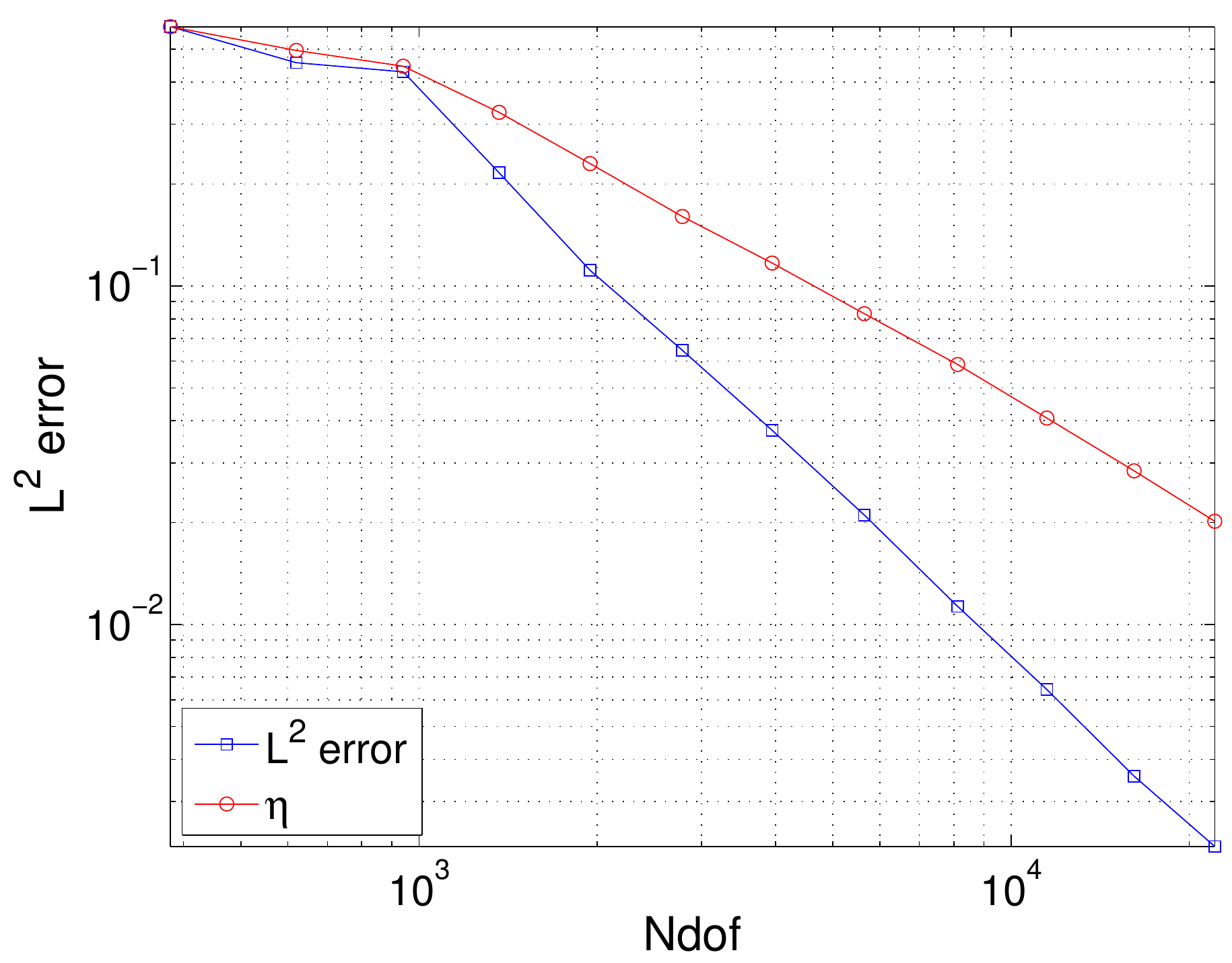}}&
\resizebox{0.4\textwidth}{!}{\includegraphics{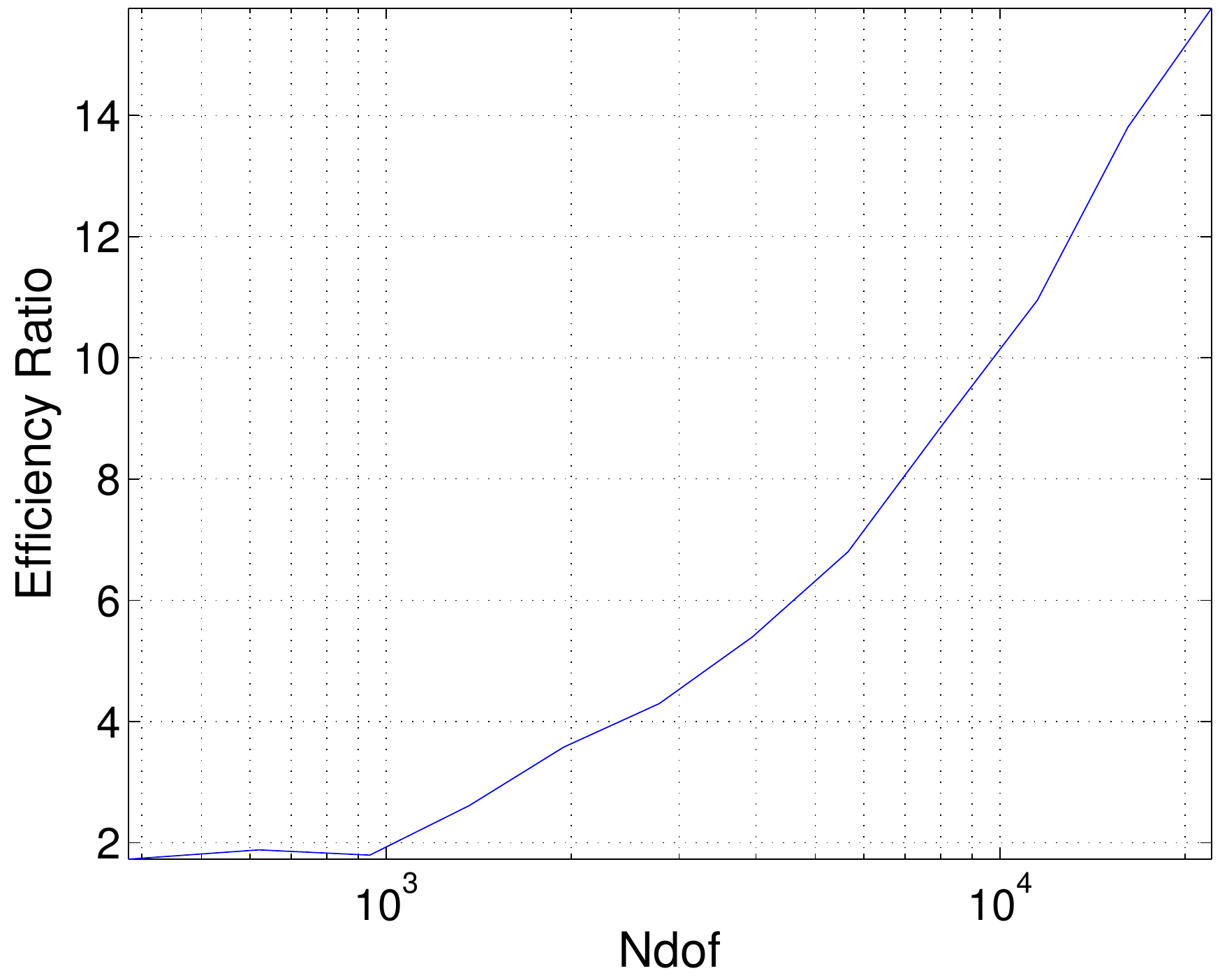}}\\
\resizebox{0.4\textwidth}{!}{\includegraphics{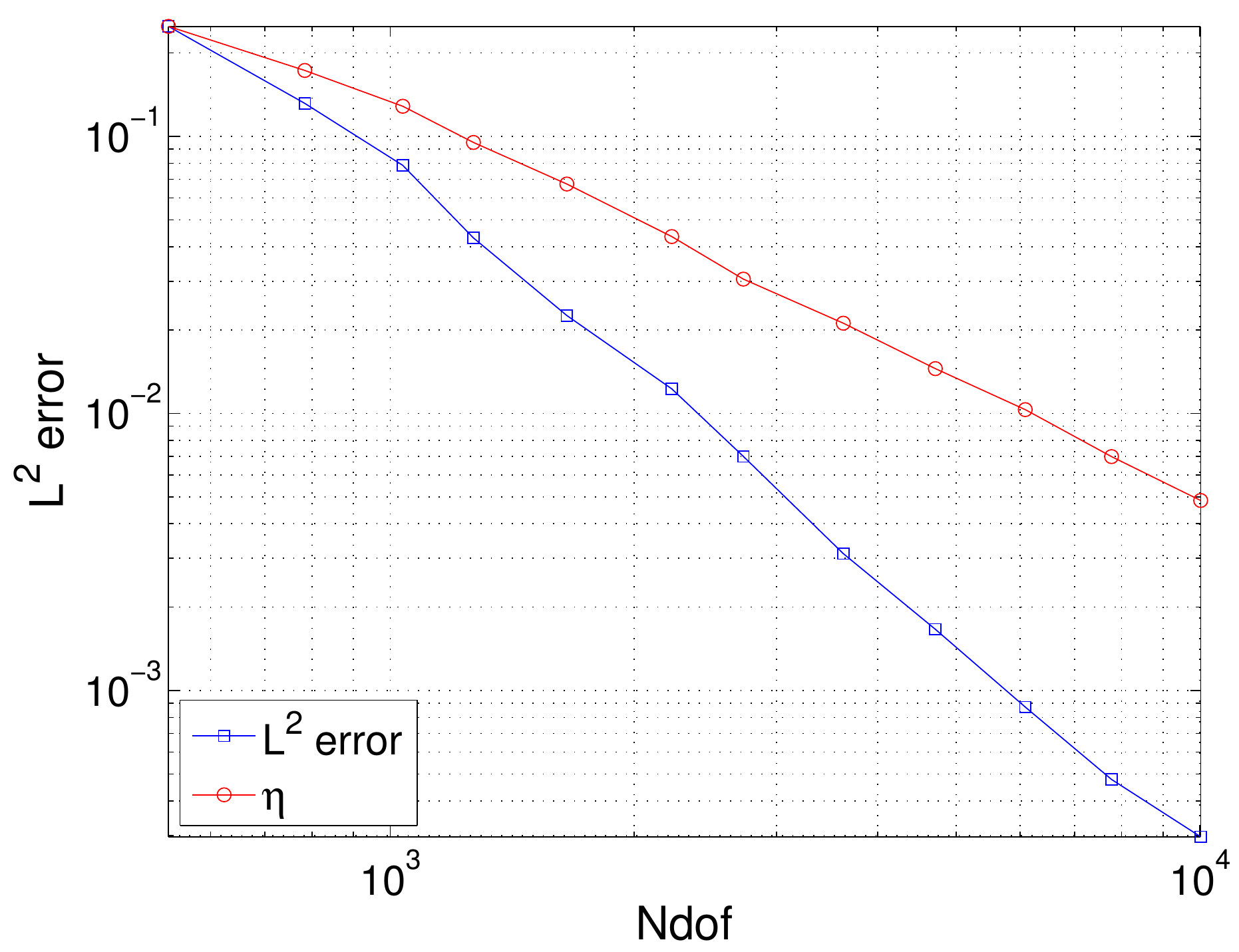}}&
\resizebox{0.4\textwidth}{!}{\includegraphics{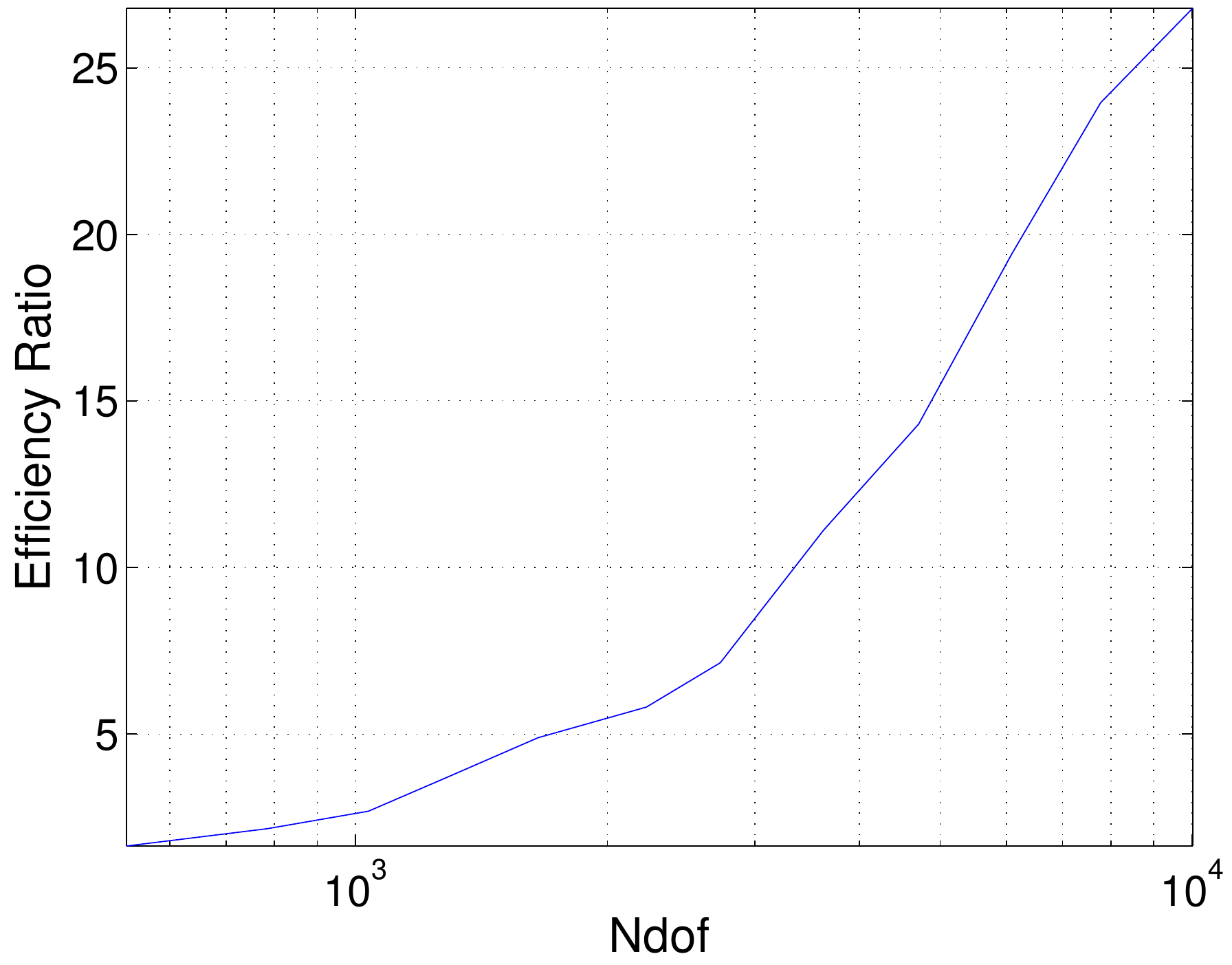}}\\
\resizebox{0.4\textwidth}{!}{\includegraphics{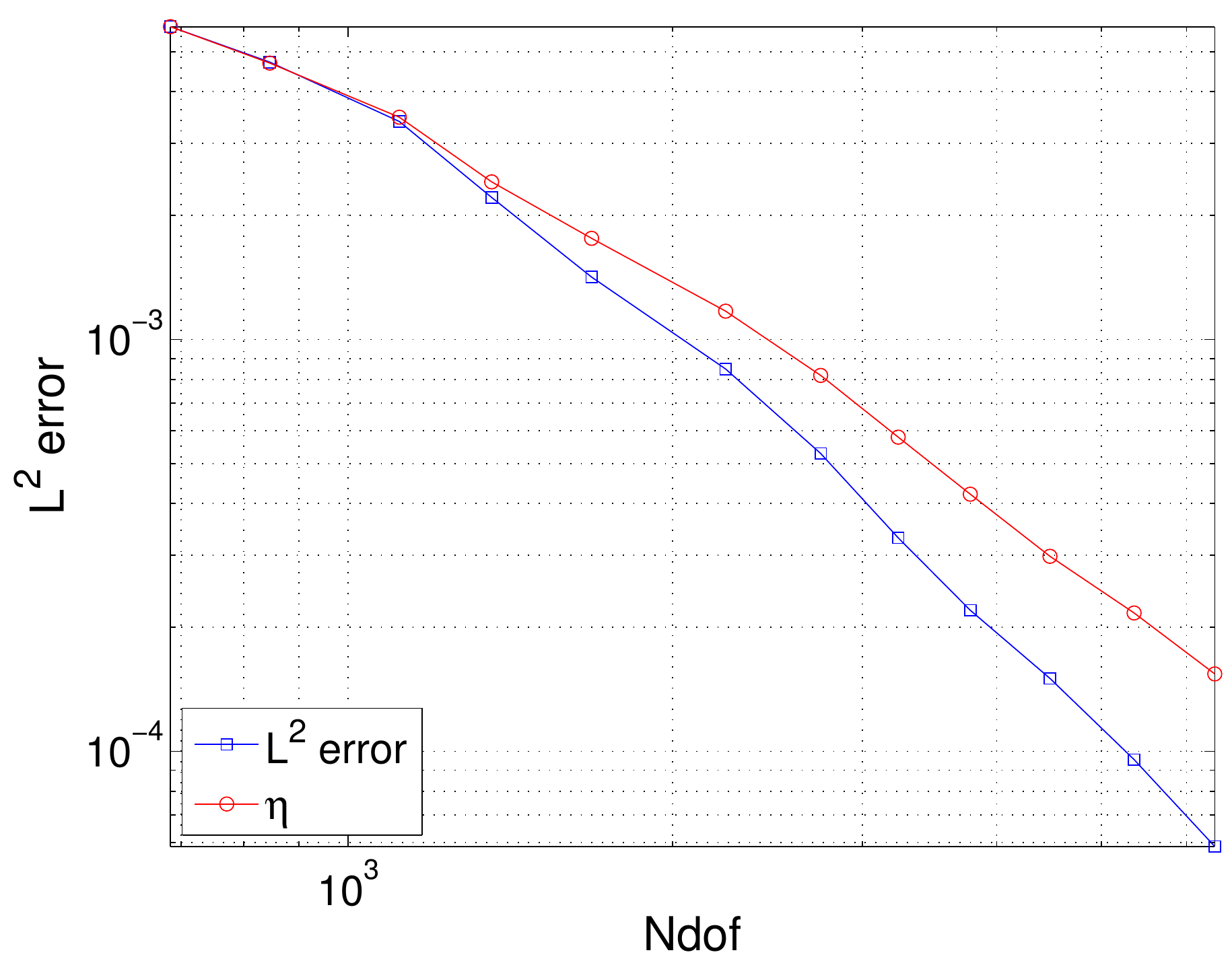}}&
\resizebox{0.4\textwidth}{!}{\includegraphics{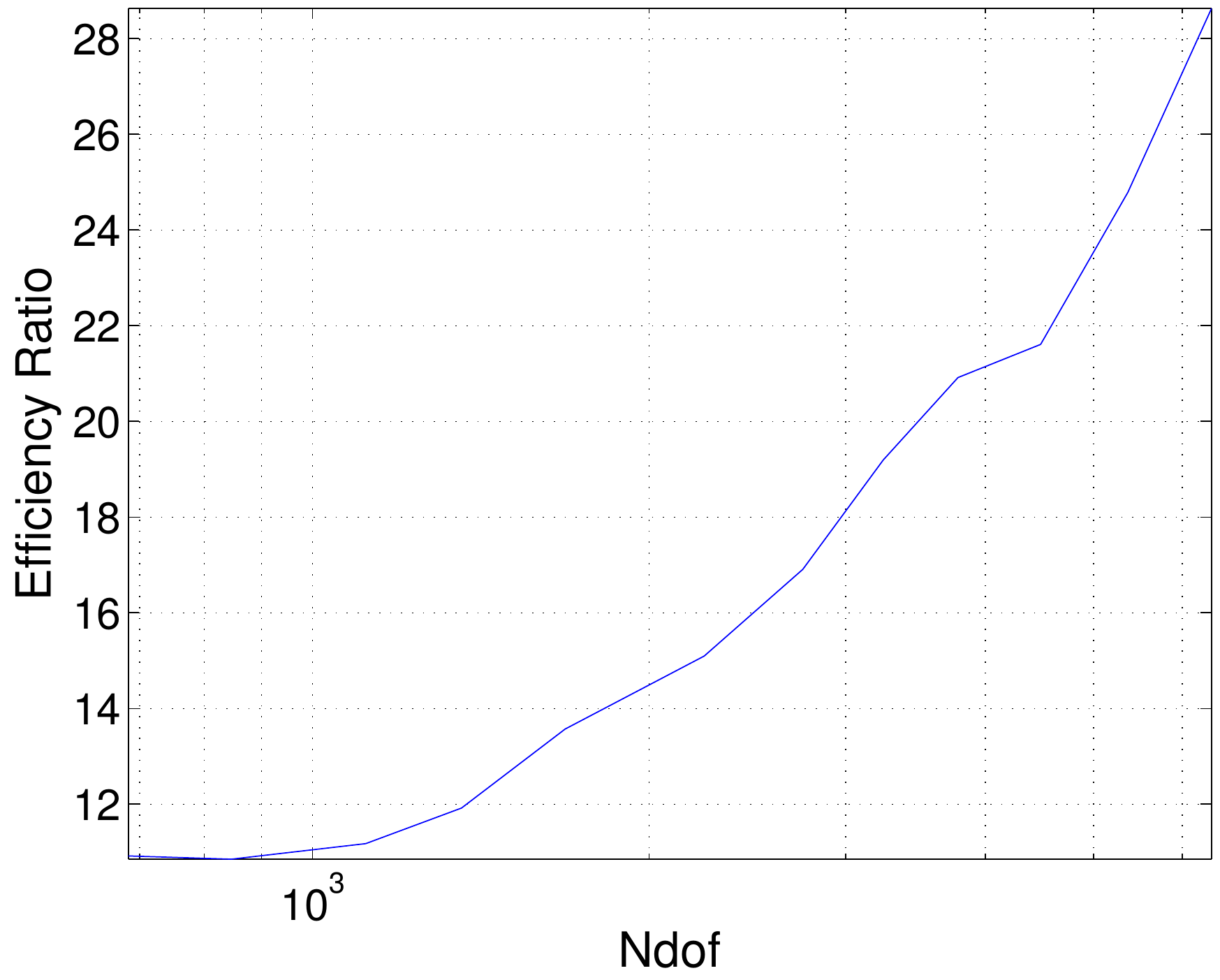}}
\end{tabular}
\caption{Results for the smooth Bessel function solution on the L-shaped domain using $s=1/6$.  The top row is for  $p_K=5$, the middle for $p_K=7$ and the bottom for $p_K=9$.  The left column shows the indicator (normalized to the actual error
at the start) and relative $L^2$ error as a function of the number of degrees of freedom.  The right column measures the efficiency of the indicator and shows
the ratio of the true error in the $L^2$norm to the residual.  Ideally this curve should be flat (at least for a well resolved solution).  }
\label{L_smooth}
\end{center}
\end{figure}

We now consider a physically relevant solution with an appropriate singularity at the reentrant corner.  
We choose  the exact solution of 
$(\ref{eq:helmholtz})$ given by 
$$u(\textbf{x}) = J_{\xi}(kr)\sin(\xi\theta)$$
for $\xi = 2/3$.  In this case, near $r=0$, $u\approx C r^{2/3}\sin(2\theta/3)$ so $u\in H^{5/3-\epsilon}(\Omega)$ for any $\epsilon>0$ and we again take $s=1/6$ in the estimators.  The boundary conditions (only Dirichlet in our numerical experiments) are determined from this exact solution.

The  computed solution and the corresponding final mesh after 12 refinement steps is shown in
Fig. ~\ref{fig: singular_bessel_soln} (starting from the mesh in Fig.~\ref{init_mesh_L}).  Clearly the algorithm has refined the mesh near the reentrant corner as expected. 
	
\begin{figure}
\centering
\subfigure[Computed Solution]{%
\includegraphics[scale=0.35]{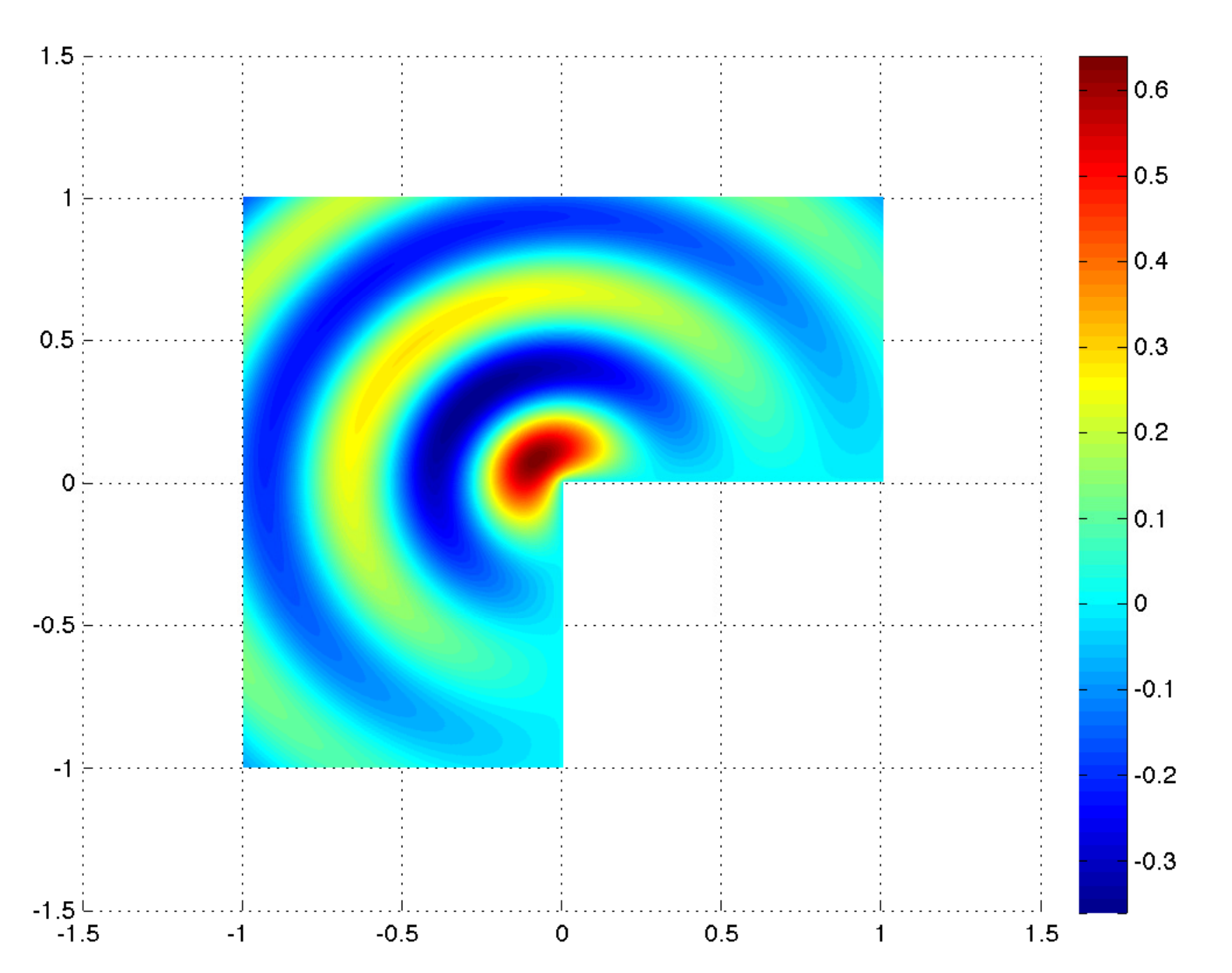}}
\quad 
\subfigure[Refined Mesh]{%
\includegraphics[scale=0.35]{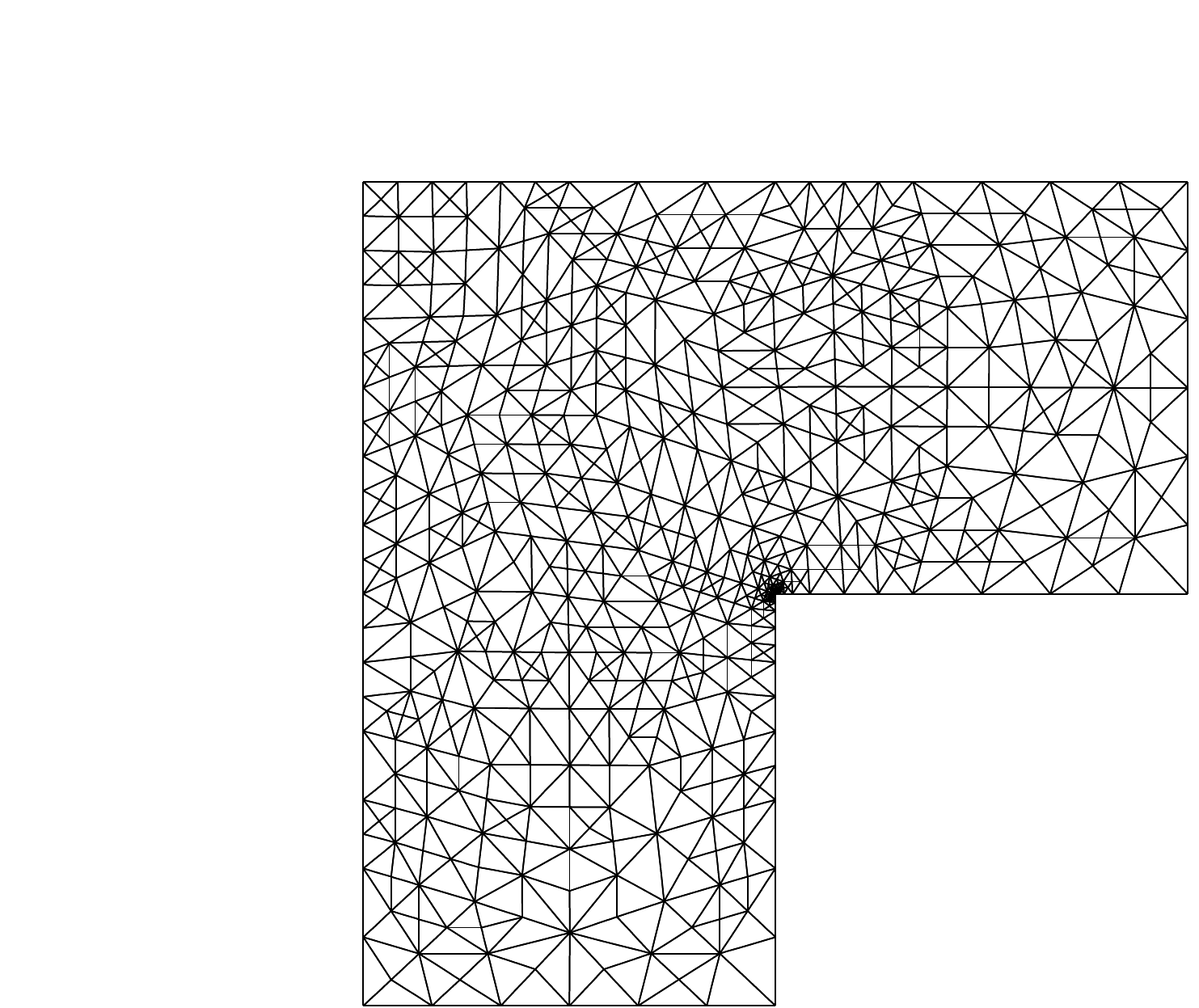}}

\caption{The numerical solution and final mesh  after 12 iterations when $\xi=2/3$ (singular solution) and $\kappa=12$ using $p_K=7$ plane waves per element.  At the resolution of the graphics, the exact and computed solution are indistinguishable.}
\label{fig: singular_bessel_soln}
\end{figure}

Results for $p_K=3$ and $p_K=4$ are shown in Fig.~\ref{peq3}.  In this case we start with a mesh obtained by two steps
of uniform refinement of the mesh in Fig.~\ref{init_mesh_L}.  This is because for low $p_K$ the original initial mesh is too coarse to produce any approximation of the solution.  If we start with the mesh in Fig.~\ref{init_mesh_L} the algorithm does correctly refine uniformly but many adaptive steps are needed before accuracy starts to improve.  The results show that our indicator works even when $p_K=3$ even though piecewise linear polynomials cannot be well
approximated in the sense of Lemma 5.1.
 
\begin{figure}
\begin{center}
\begin{tabular}{ccc}
\resizebox{0.4\textwidth}{!}{\includegraphics{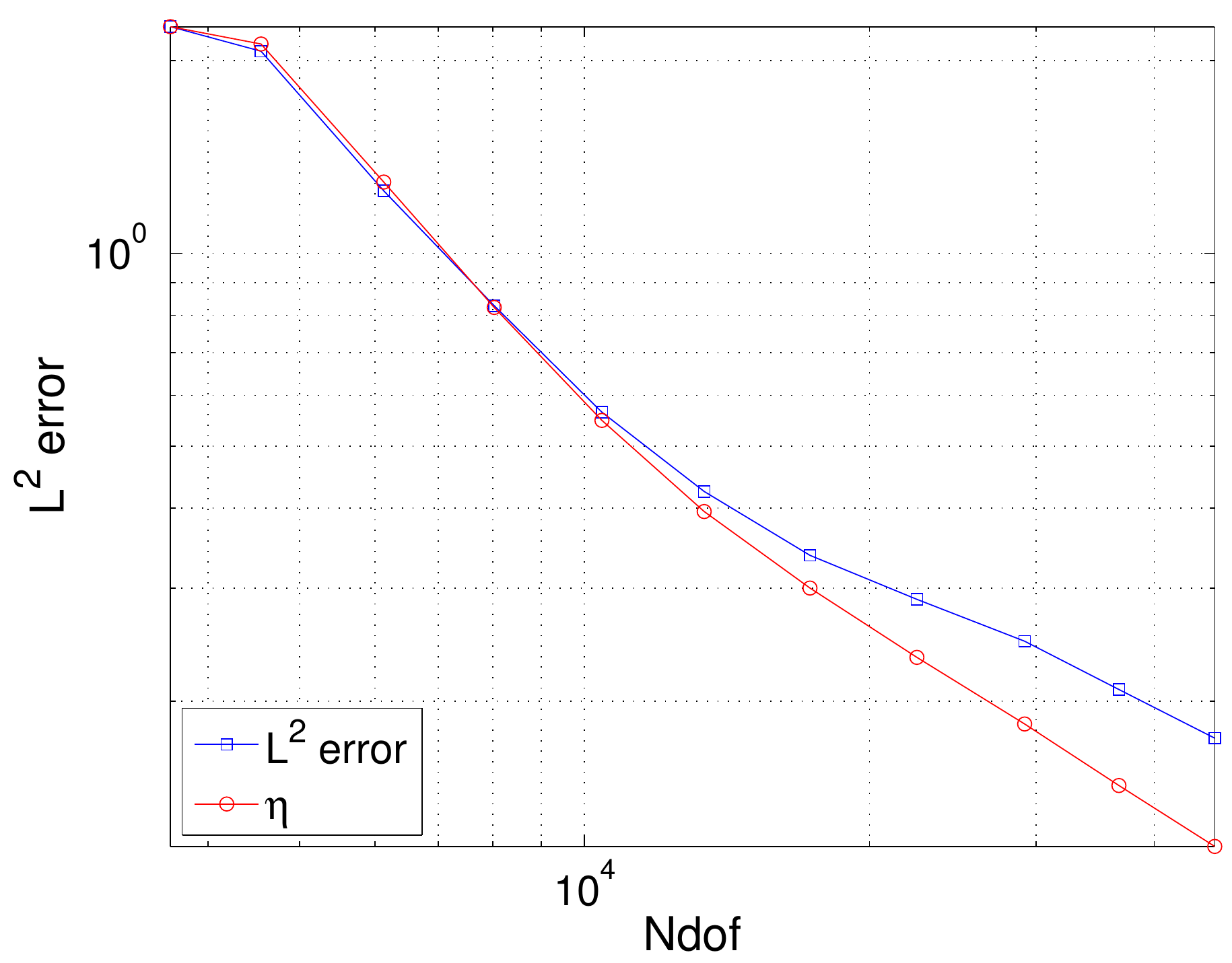}}&
\resizebox{0.4\textwidth}{!}{\includegraphics{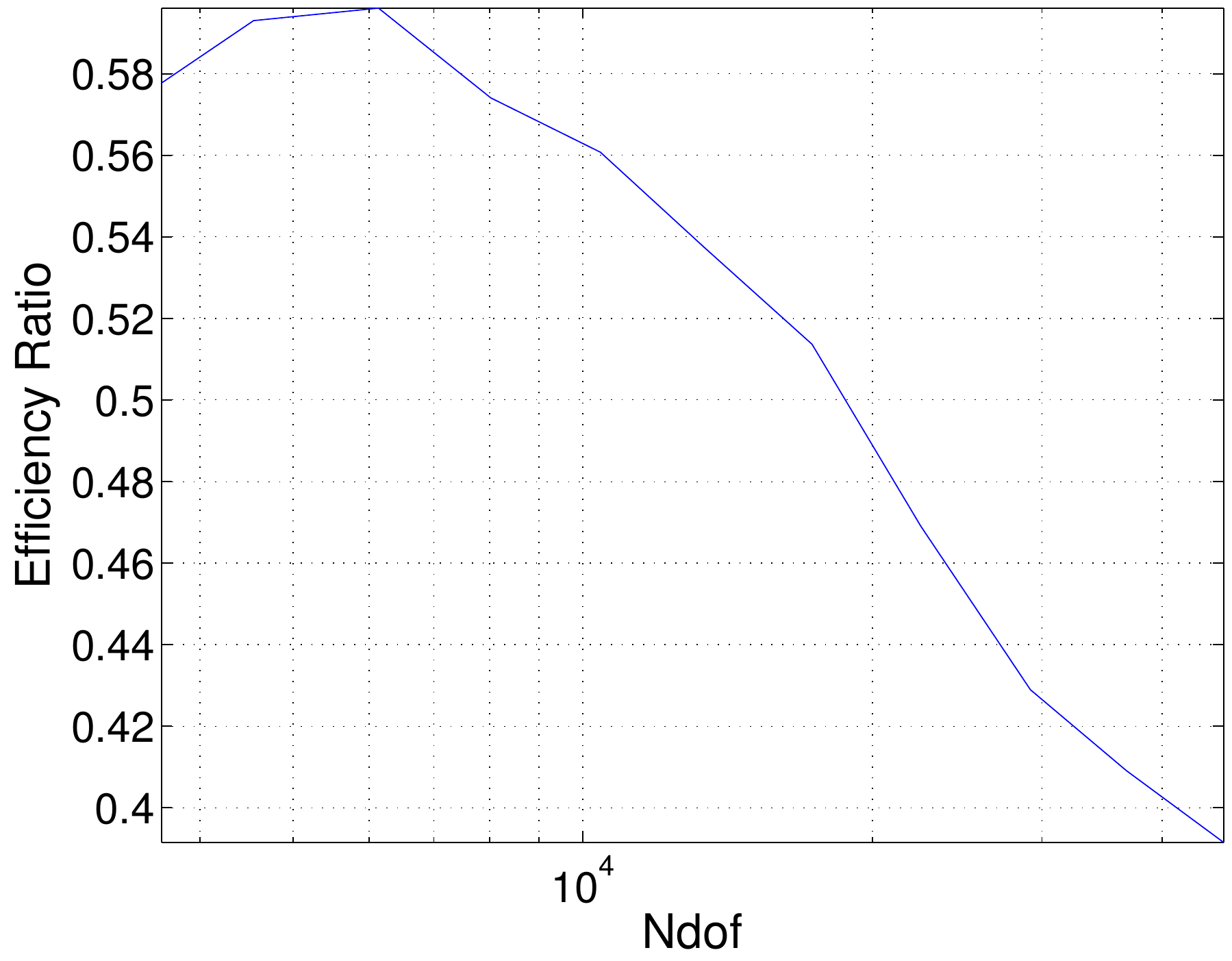}}\\
\resizebox{0.4\textwidth}{!}{\includegraphics{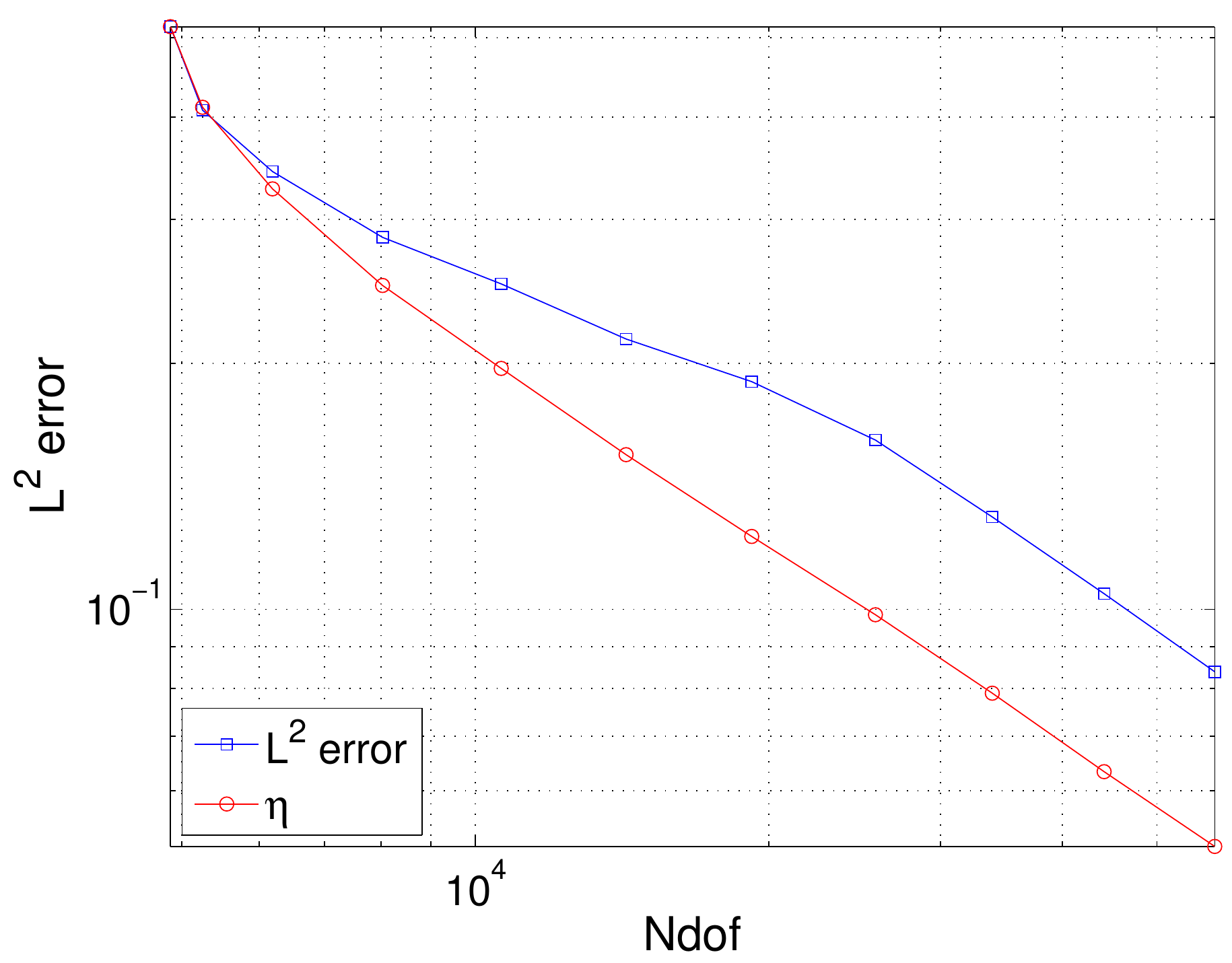}}&
\resizebox{0.4\textwidth}{!}{\includegraphics{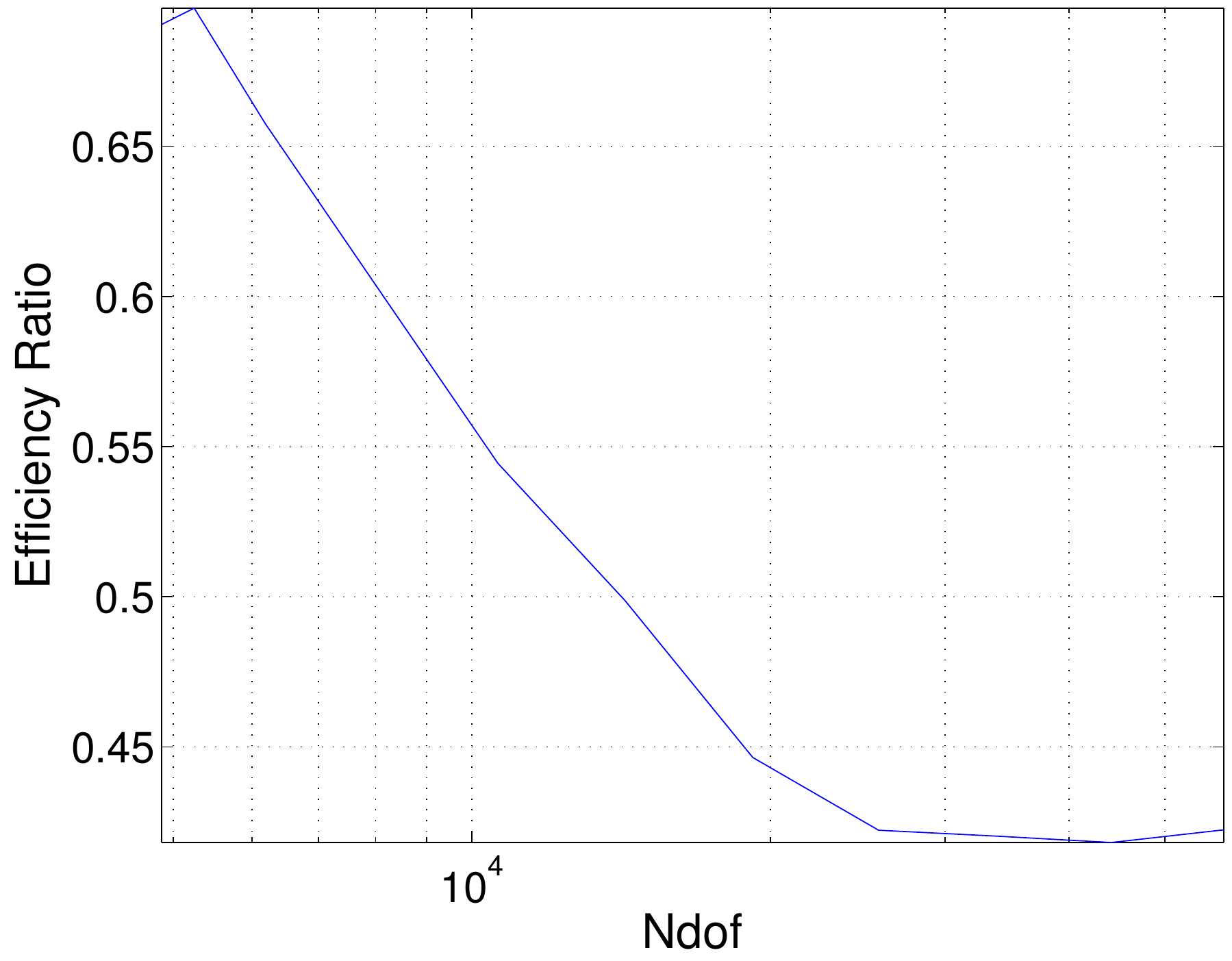}}
\end{tabular}
\end{center}
\caption{Results for the singular solution (Bessel function with $\xi=2/3$) using $p_K=3$ (top row) and $p_K=4$ (bottom row) starting from two levels of refinement of the initial grid in  Fig.~\ref{init_mesh_L}.  This figure has the same columns as Fig.~\ref{L_smooth}.  As expected there us little difference between the error attained by the two methods (the a priori error estimates are the same order for $p_K=3$ and $p_K=4$), but the residual estimator behaves better in the case when $p_K=4$ in that the efficiency curve flattens out.}
\label{peq3}
\end{figure}

Results for $p_K=5,7,9$ are shown in Fig.~\ref{L_smooth} starting with the mesh in Fig.~\ref{init_mesh_L} and using $s=1/6$ in our estimator.    Convergence is slower than for the smooth solution, but the efficiency of the indicators is improved although it does deteriorate as the mesh is refined.
\begin{figure}
\begin{center}
\begin{tabular}{ccc}
\resizebox{0.4\textwidth}{!}{\includegraphics{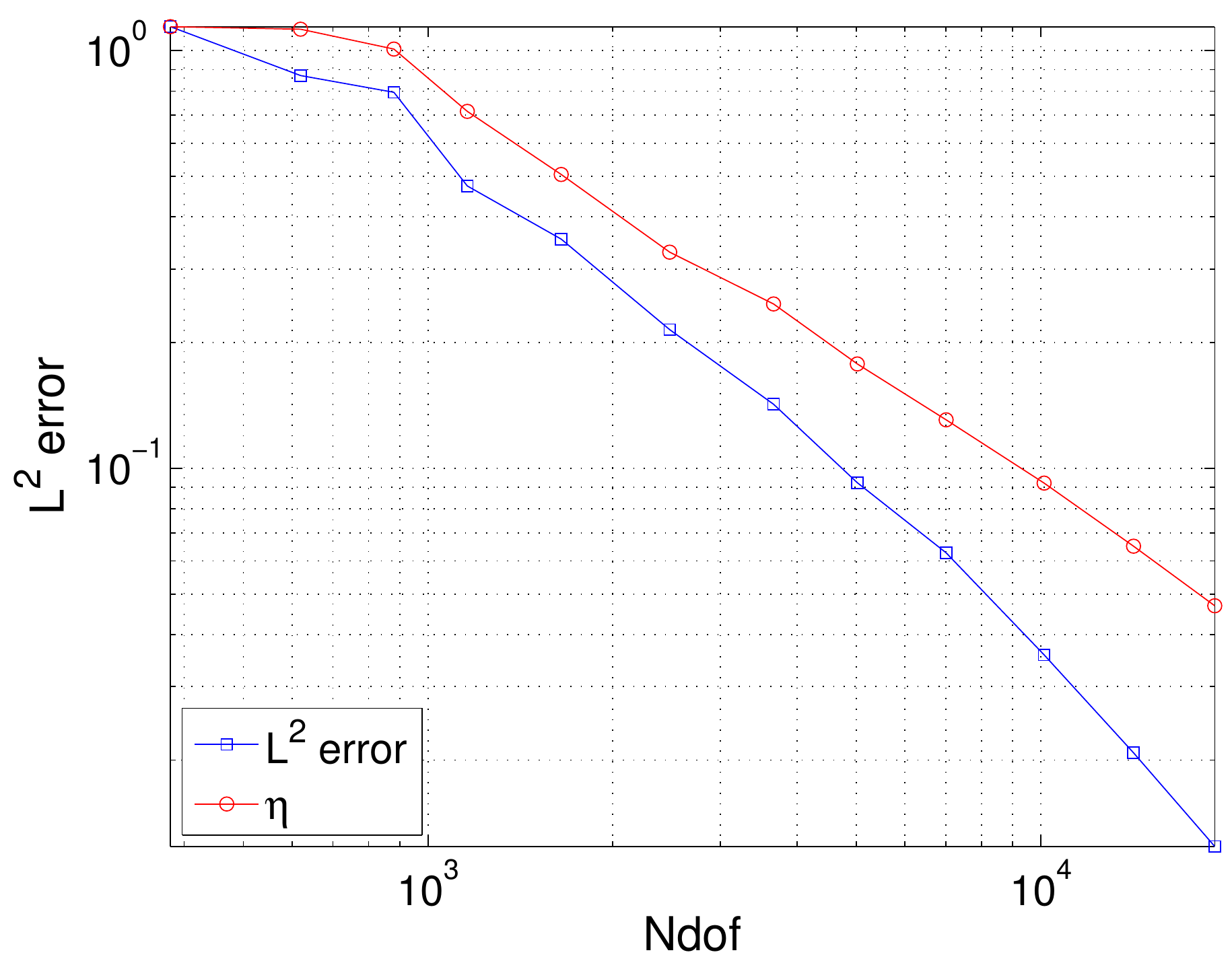}}&
\resizebox{0.4\textwidth}{!}{\includegraphics{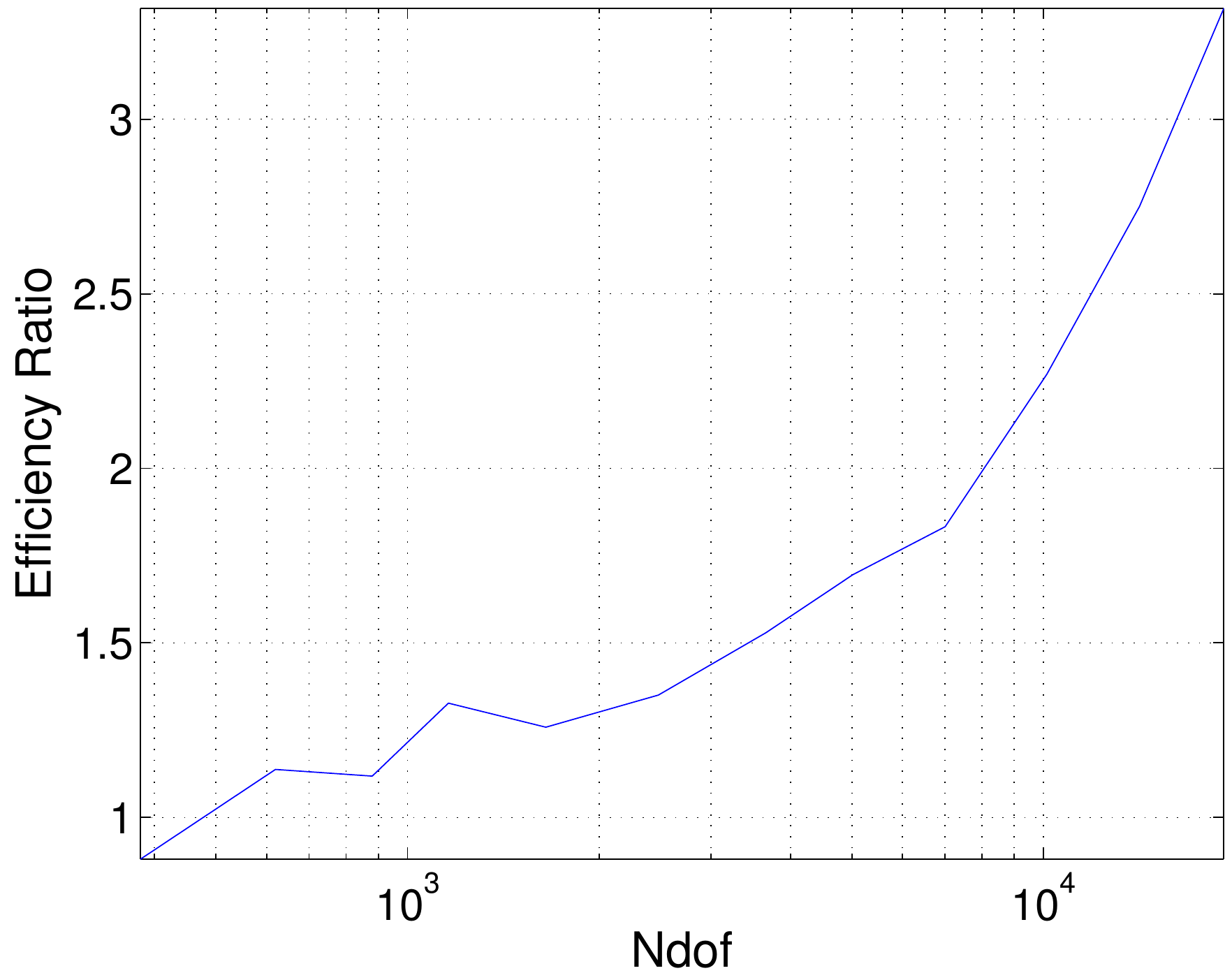}}\\
\resizebox{0.4\textwidth}{!}{\includegraphics{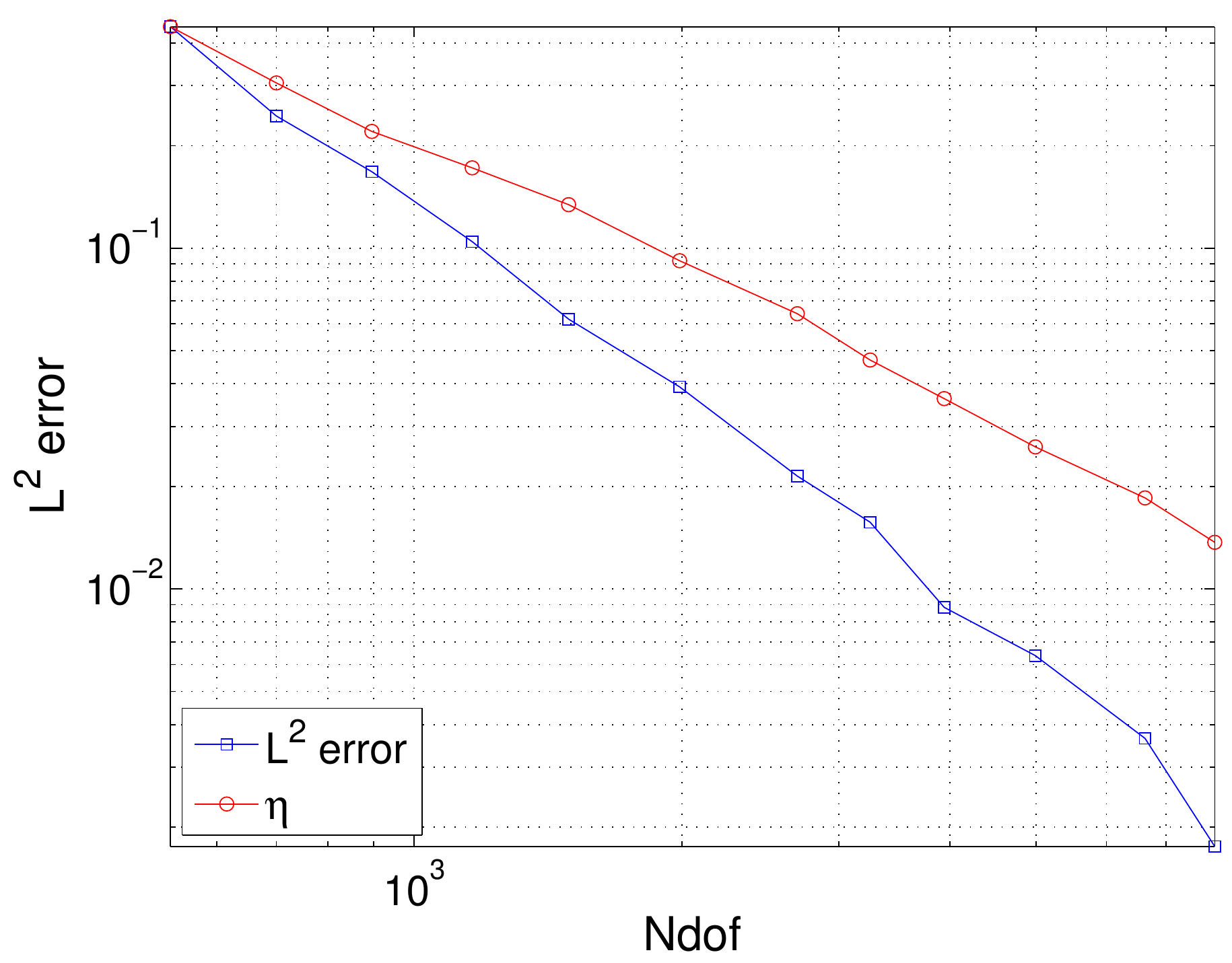}}&
\resizebox{0.4\textwidth}{!}{\includegraphics{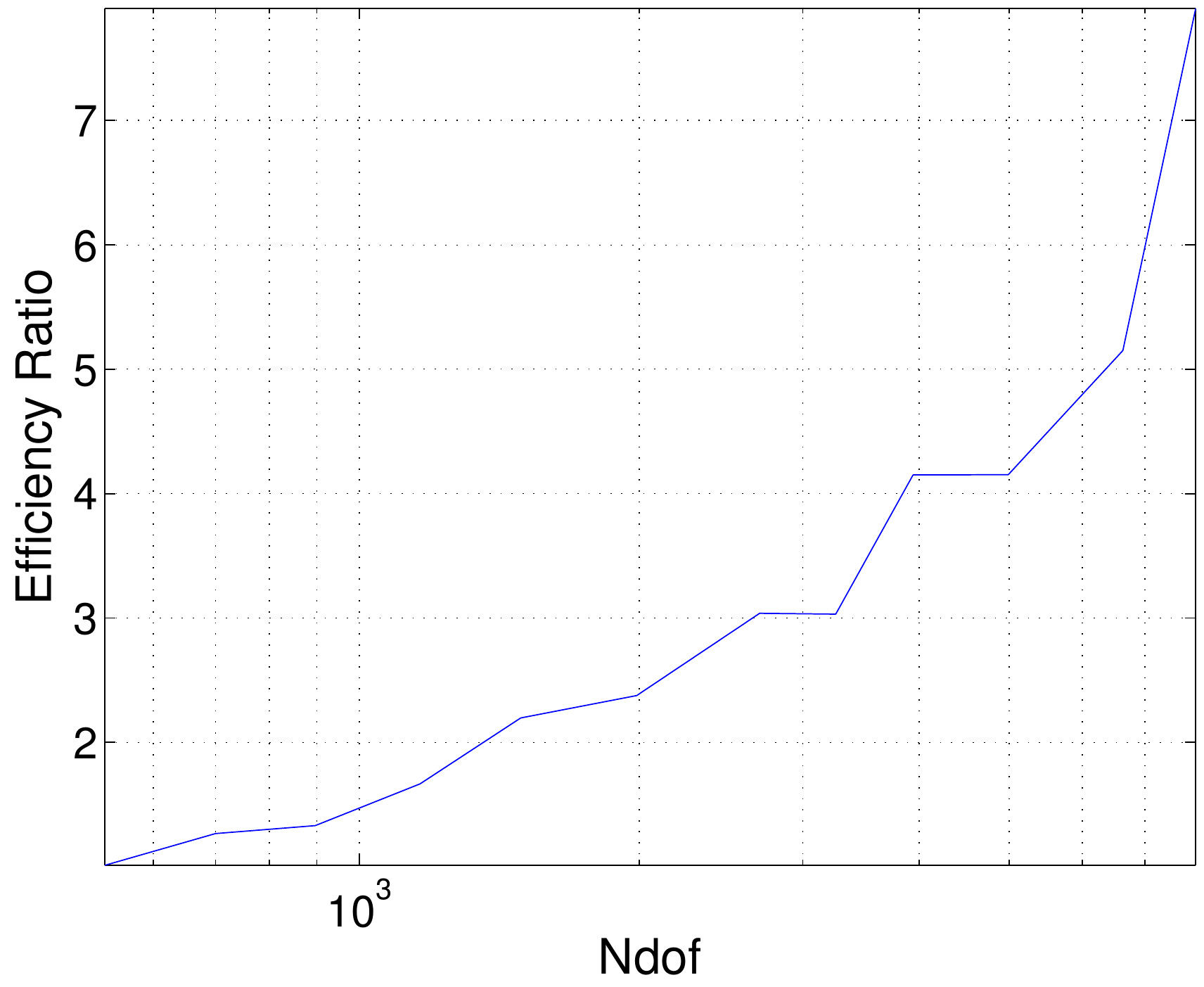}}\\
\resizebox{0.4\textwidth}{!}{\includegraphics{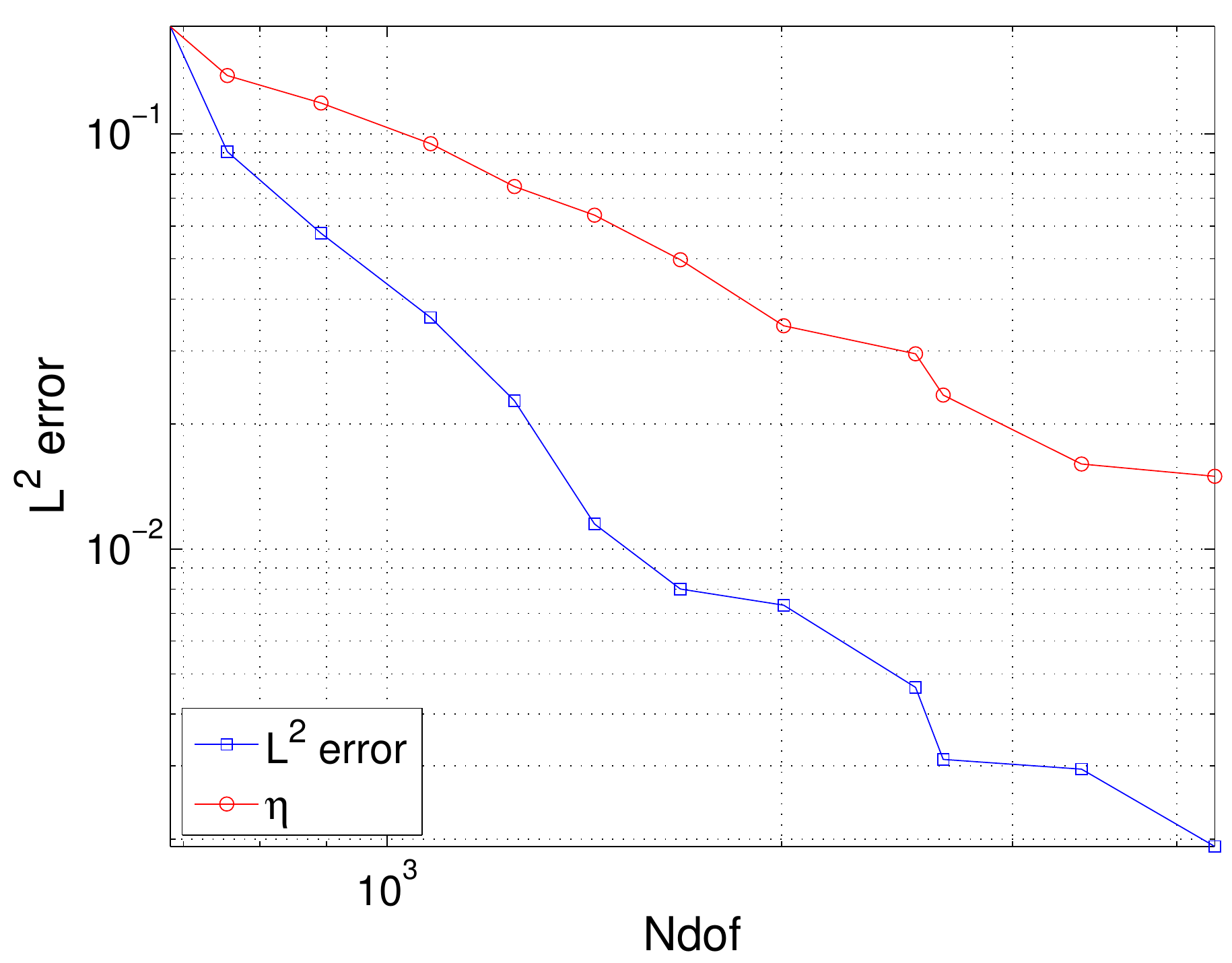}}&
\resizebox{0.4\textwidth}{!}{\includegraphics{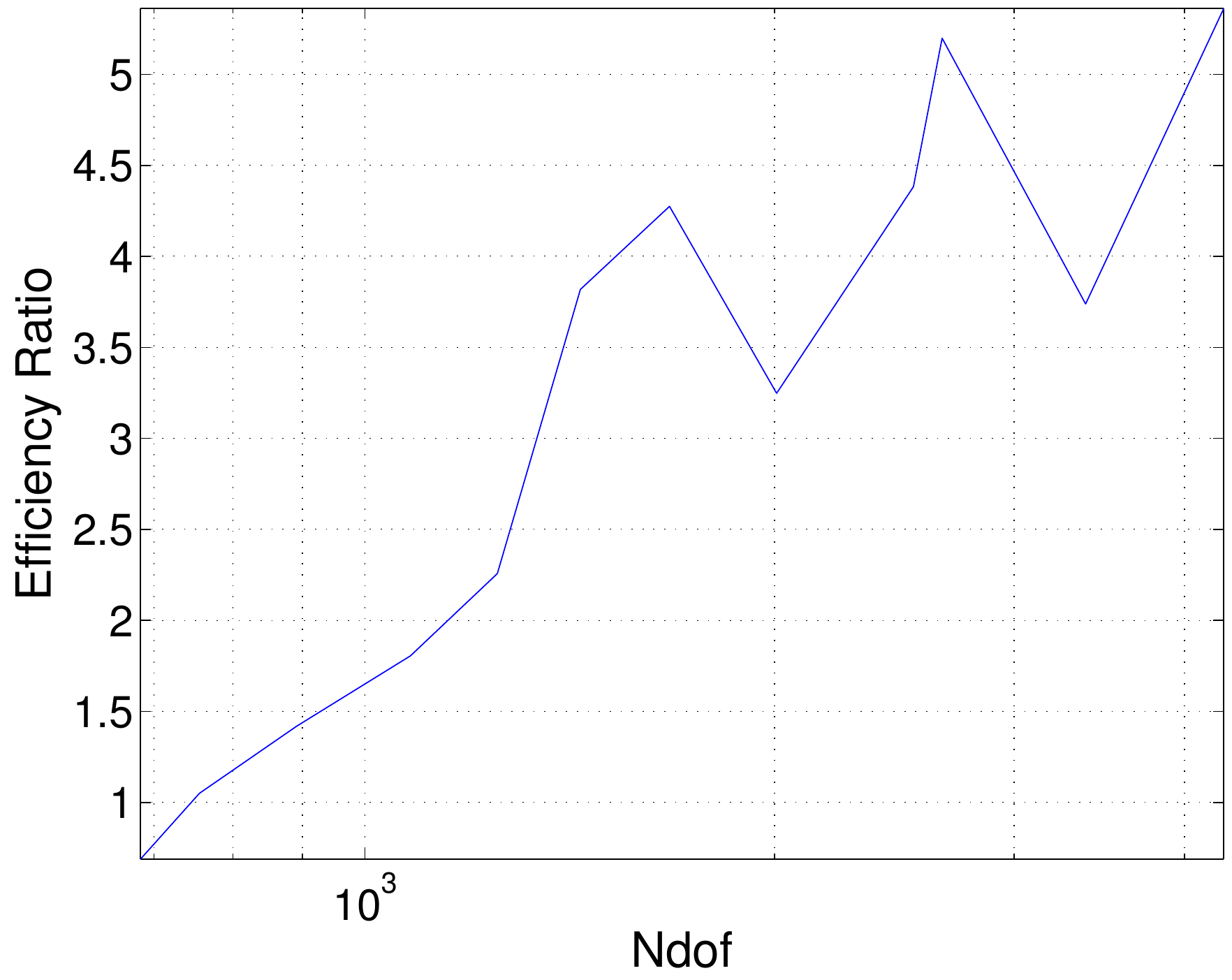}}
\end{tabular}
\end{center}
\caption{Results for the singular solution (Bessel function with $\xi=2/3$) using $p_K=5$ (top row), $p_K=7$ (middle row) and $p_K=9$ (bottom row). We start from the initial grid in  Fig.~\ref{init_mesh_L}.This figure has the same layout as Fig.~\ref{L_smooth}. }.
\label{L_nonsmooth}
\end{figure}

For the Helmholtz equation, besides standard elliptic corner singularities mentioned above, adaptivity may also help deal with boundary layers that can arise at interfaces between regions with different refractive indices. 
We now consider adaptivity for the transmission and reflection of a plane wave across a fluid-fluid interface on a square domain $\Omega:=(-1,1)^2$ with two different refractive indices. The interface is located at $y=0$. The problem now is to find $u\in H^1(\Omega)$ such that
\begin{equation}
\Delta u+k^2\epsilon_ru=0\mbox{ in }\Omega\label{star}
\end{equation}
subject to appropriate boundary conditions where 
\[
\epsilon_r(x,y)=\left\{\begin{array}{cc} n_1^2&\mbox{ if }y>0,\\
 n_2^2&\mbox{ if }y<0.\end{array}\right.
\]
We choose $ n_1=1$ and $n_2=4$.  Then it is easy to show that for any angle $0\leq \theta_i<\pi/2$
and $\mathbf{d}=(\cos(\theta_i),\sin(\theta_i))$ the following is 
a solution of (\ref{star})
\[
u(x,y)=\left\{\begin{array}{cc}T\exp(i(K_1x+K_2y))&\mbox{ if }y>0,\\
\exp(i\kappa n_1(d_1x+d_2y))+R\exp(i\kappa n_1(d_1x-d_2y))&\mbox{ if }y<0.\end{array}\right.
\]
where $K_1=\kappa n_1 d_1$ and $K_2=\kappa\sqrt(n_2^2-n_1^2d_1^2)$ and
\begin{eqnarray*}
R&=&-(K_2-\kappa n_1 d_2)/(K_2+\kappa n_1 d_2),\\
T&=&1+R.
\end{eqnarray*}
If $n_2^2-n_1^2d_1^2<0$ (i.e. if $n_2>n_1$ and $d_1$ is large enough) then $K_2$ is imaginary (we choose a positive imaginary part) and the solution for $y>0$ decays exponentially into the upper half plane (physically this is said
to be total internal reflection since the wave above the interface is vanishingly low amplitude far from the interface).  If $d_1$ is small enough (i.e. close to normal incidence) the wave is refracted at the interface  and a 
traveling wave is seen above and below the interface.  Thus there is a critical angle $\theta_i=\theta_{crit}$ such that
for $\theta_i>\theta_{crit}$ the wave is refracted, and for $\theta_i<\theta_{crit}$ we have internal reflection.  This is shown in Fig.~\ref{fig: total_internal_soln}.  The case of internal reflection is challenging for a plane wave based method since
evanescent (or exponentially decaying) waves are not in the basis.  We therefore investigate if our residual estimators can appropriately refine the mesh in this case (this not a problem covered by our theory).

\begin{figure}[ht]
\centering
\subfigure[$\theta_{inc}=29^\circ$,\;$\theta_{i}<\theta_{crit}$]{%
\includegraphics[scale=0.4]{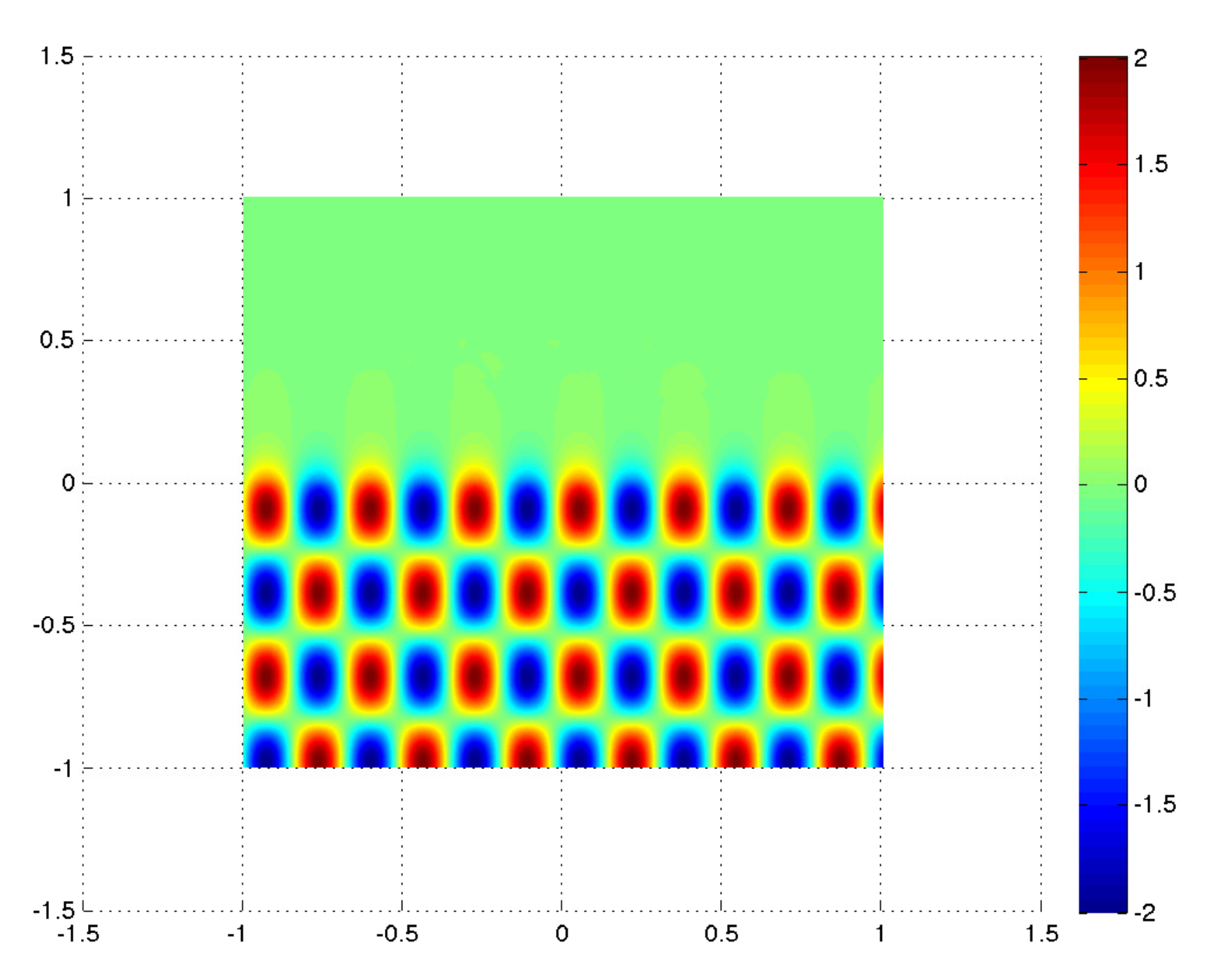}}
\quad
\subfigure[$\theta_{inc}=69^\circ$,\;$\theta_{i}>\theta_{crit}$]{%
\includegraphics[scale=0.4]{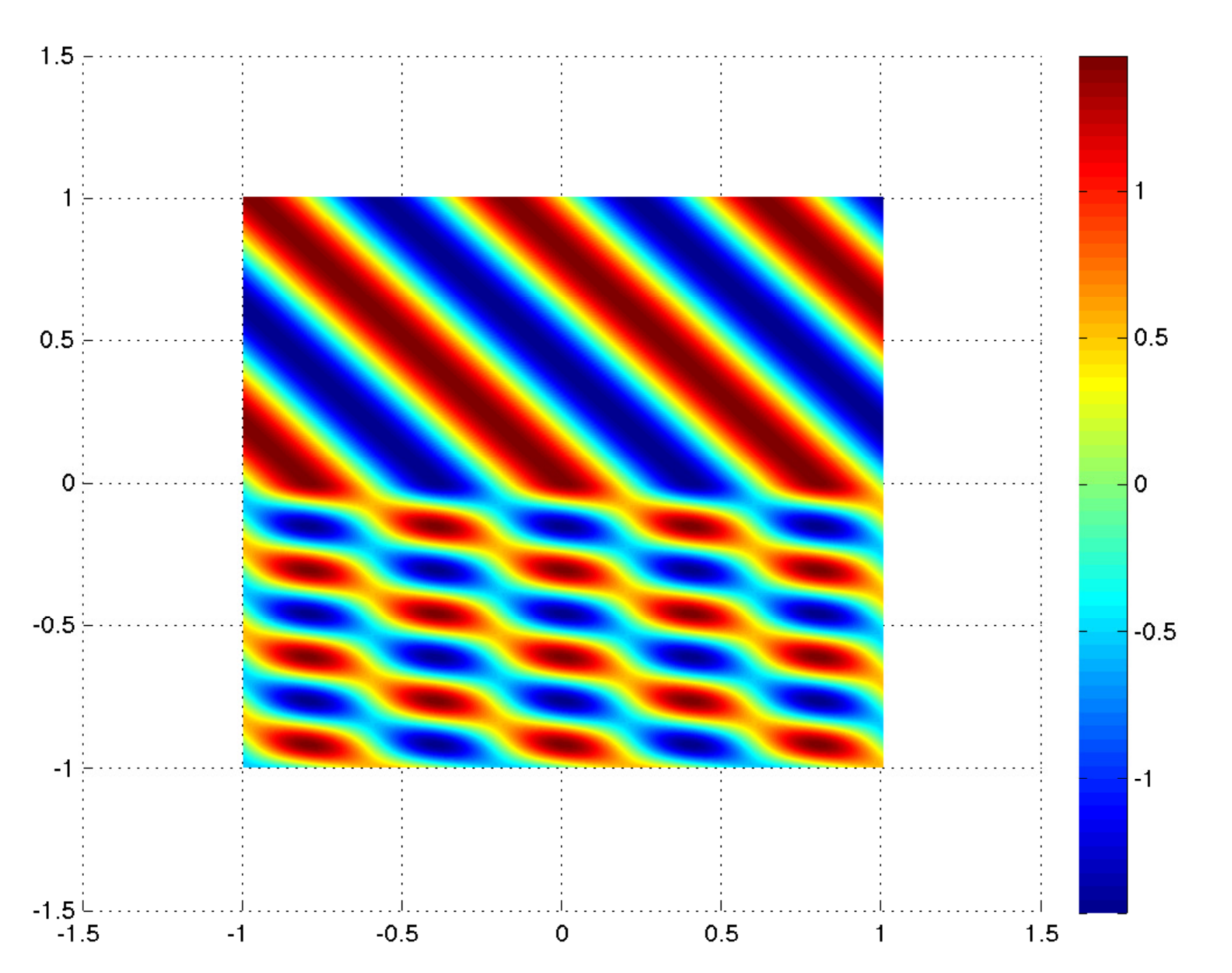}}
\caption{Numerical solutions after 12 iterations when $k=11$ and $n_1=2$,$n_2=1$, $p_K=7$ plane waves per element. When $\theta_{i}<\theta_{crit}$ the wave decays exponentially into the upper half of the plane as shown for $\theta_i=29^\circ$ (left panel).  When $\theta_i=69^\circ$ the wave is transmitted into the upper half of the square (right panel). }\label{fig: total_internal_soln}
\end{figure}

In particular we use Dirichlet boundary conditions derived from the exact solution (assuming $\kappa$ is not an interior 
eigenvalue for the domain) and choose the wavenumber is $\kappa=11$.    In view of the fact that the domain is convex  with a smooth interior interface we choose $s=1/2$ in the estimator.

\begin{figure}
\centering
\subfigure[Initial Mesh]{
\includegraphics[scale=0.4]{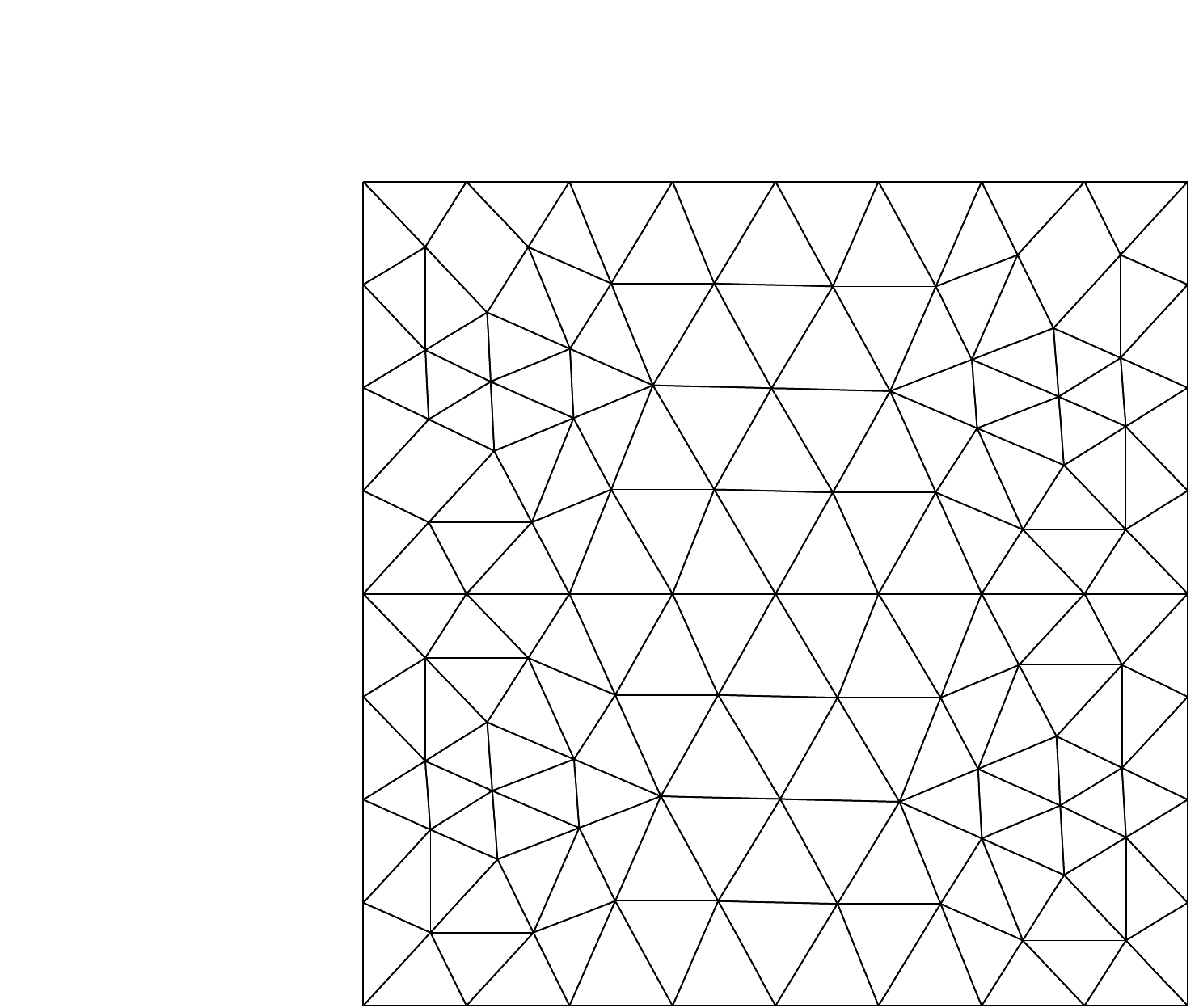}}
\subfigure[ $\theta_{inc}=69^\circ$]{
\includegraphics[scale=0.4]{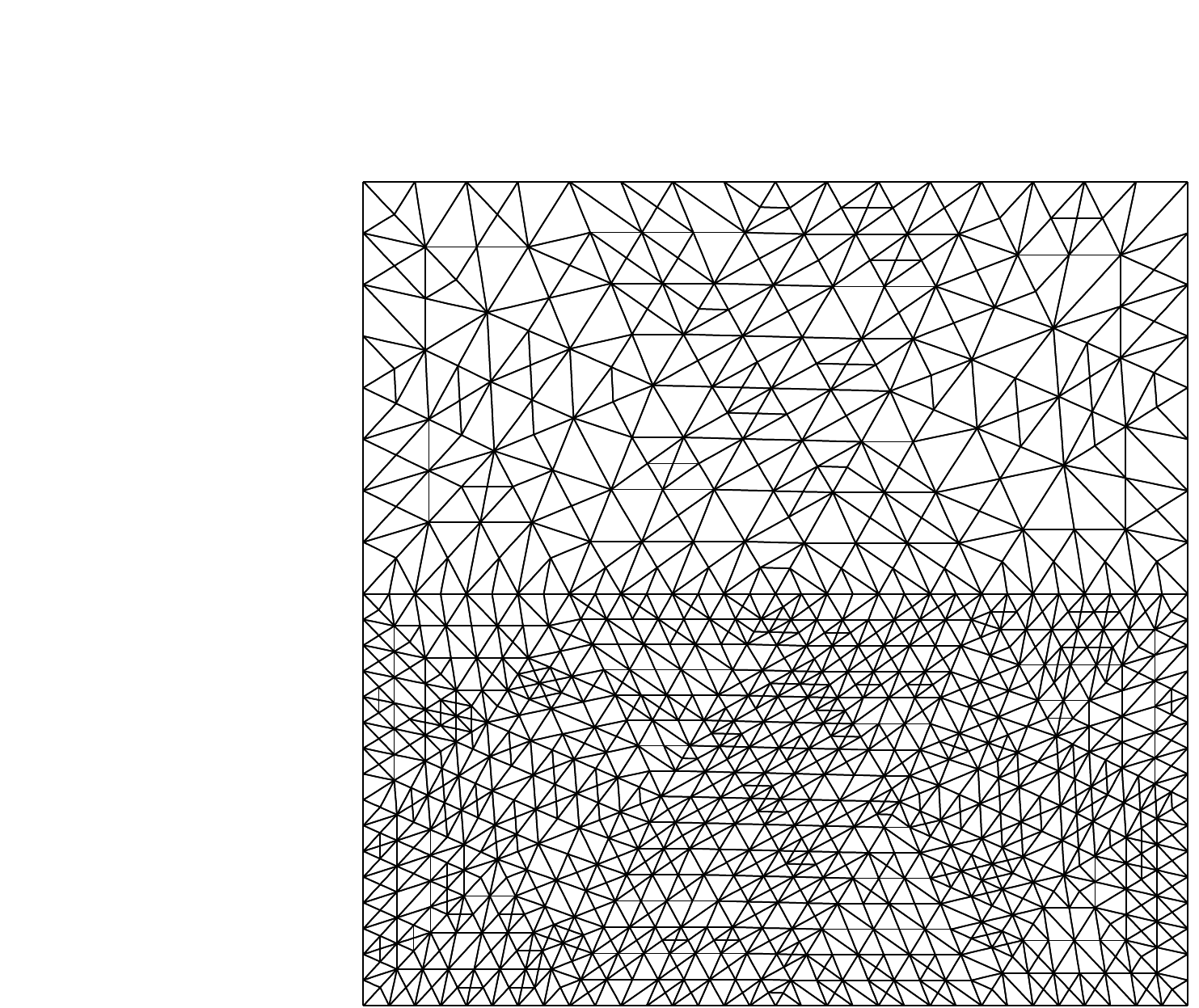}}
\subfigure[ $\theta_{inc}=29^\circ$]{
\includegraphics[scale=0.4]{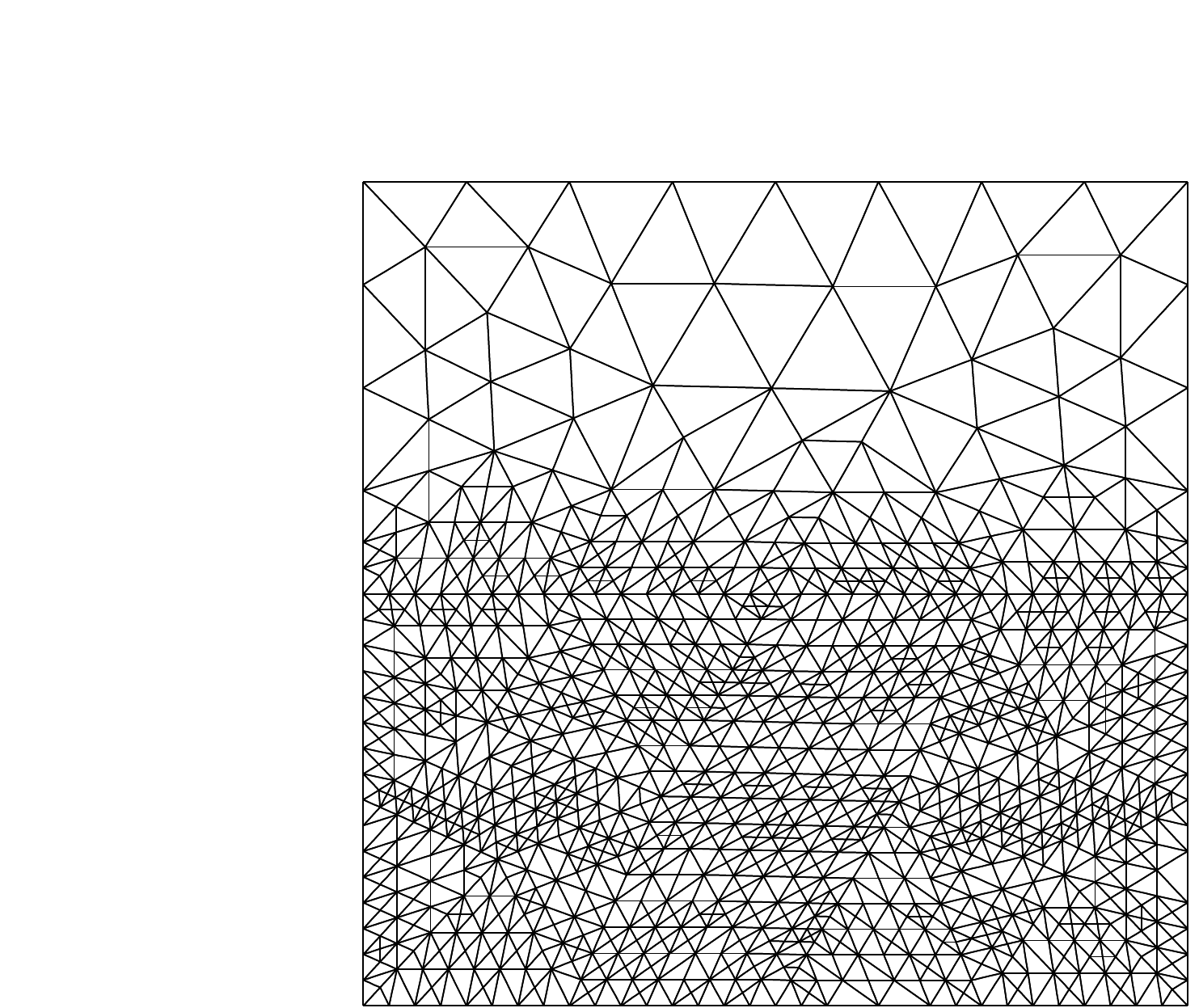}}
\caption{Initial mesh and the meshes after 12 adaptive iterations for transmission ($\theta_i=69^\circ$) and internal reflection
($\theta_i=29^\circ$).  Here $p_K=7$.}
\label{ir_mesh}
\end{figure}

Representative meshes produced by our algorithm are shown in Fig.~\ref{ir_mesh}. Starting with the initial mesh in panel a), we generate the mesh in panel b) when $\theta_i=69^\circ$.  The algorithm correctly refines the lower half square more, and there is an abrupt transition to the less refined upper half. In panel c) we show the mesh
when $\theta=29^\circ$. In this case the algorithm correctly does not refine well above the interface, but at the interface
$y=0$ some refinement occurs even for $y>0$ in order to resolve the exponentially decaying solutions there. We shall only consider the
case $\theta_i=29^\circ$  (internal reflection) from now on.

Detailed error plots when $p_K=5,7,9$ are shown in Figure ~\ref{IR}.  The results are broadly similar to our previous results.  The error is decreased by the refinement strategy, but efficiency generally deteriorates as the mesh is refined.  Again the error indicator for the higher order method, $p_K=9$, is best.
\begin{figure}
\begin{center}
\begin{tabular}{ccc}
\resizebox{0.4\textwidth}{!}{\includegraphics{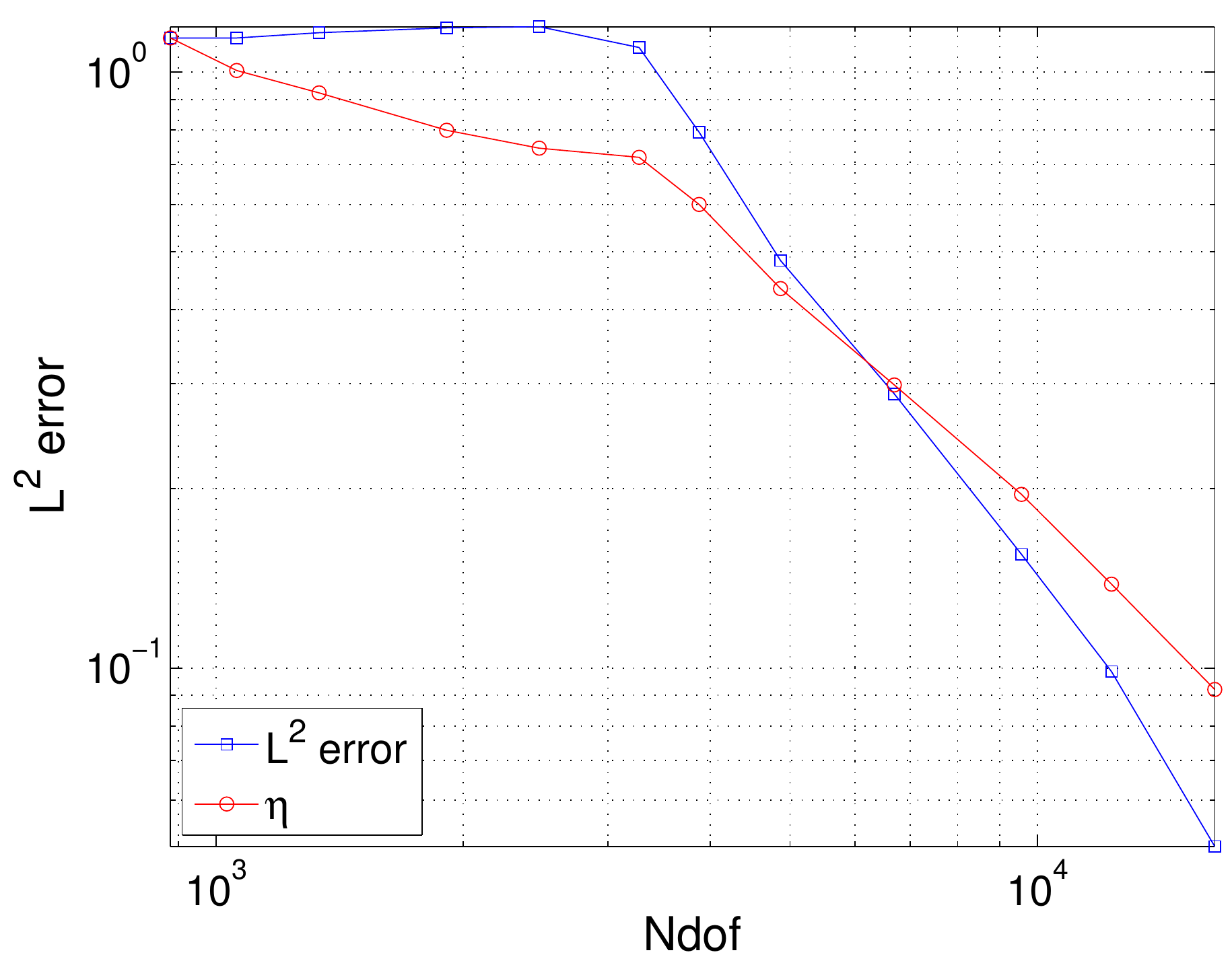}}&
\resizebox{0.4\textwidth}{!}{\includegraphics{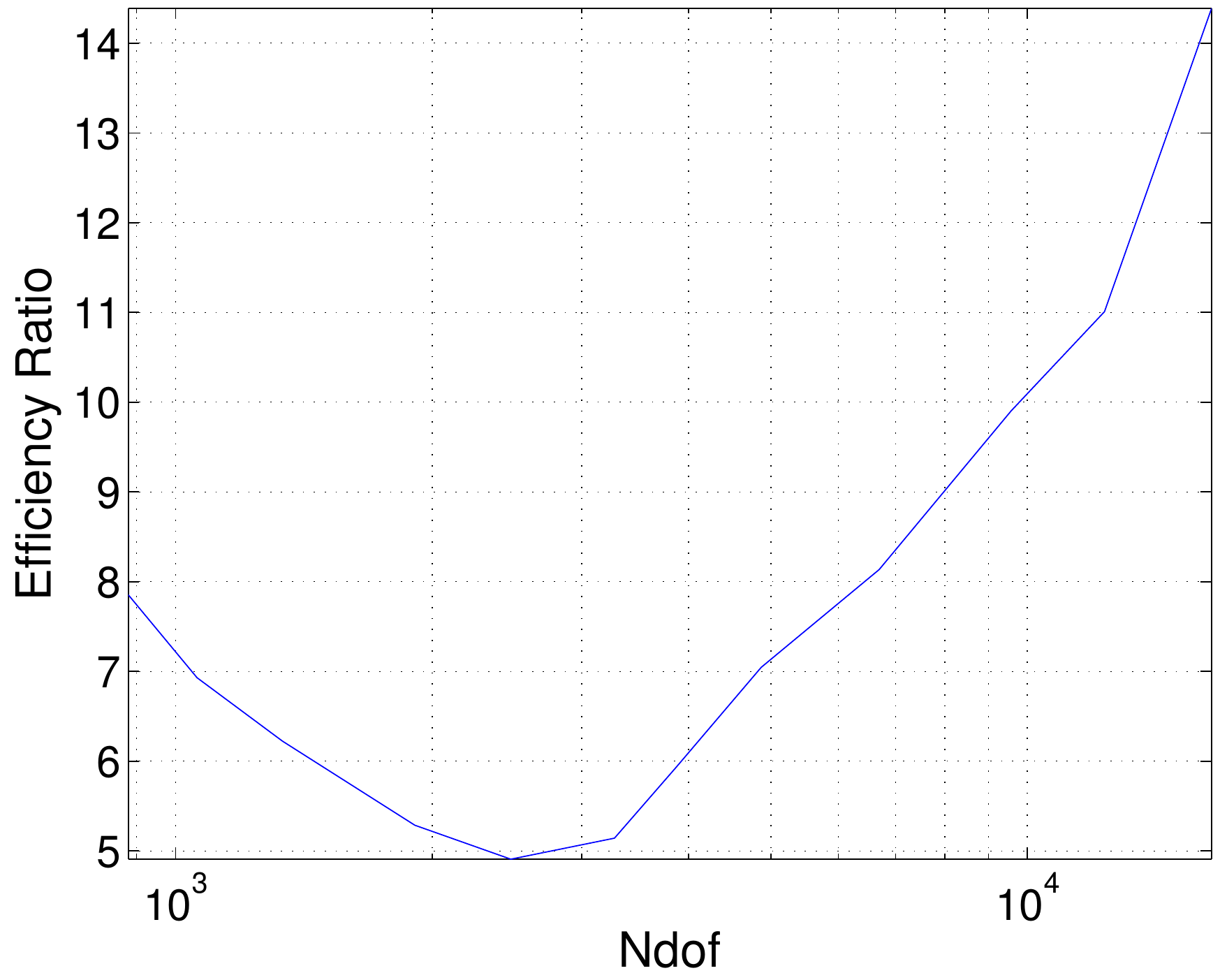}}\\
\resizebox{0.4\textwidth}{!}{\includegraphics{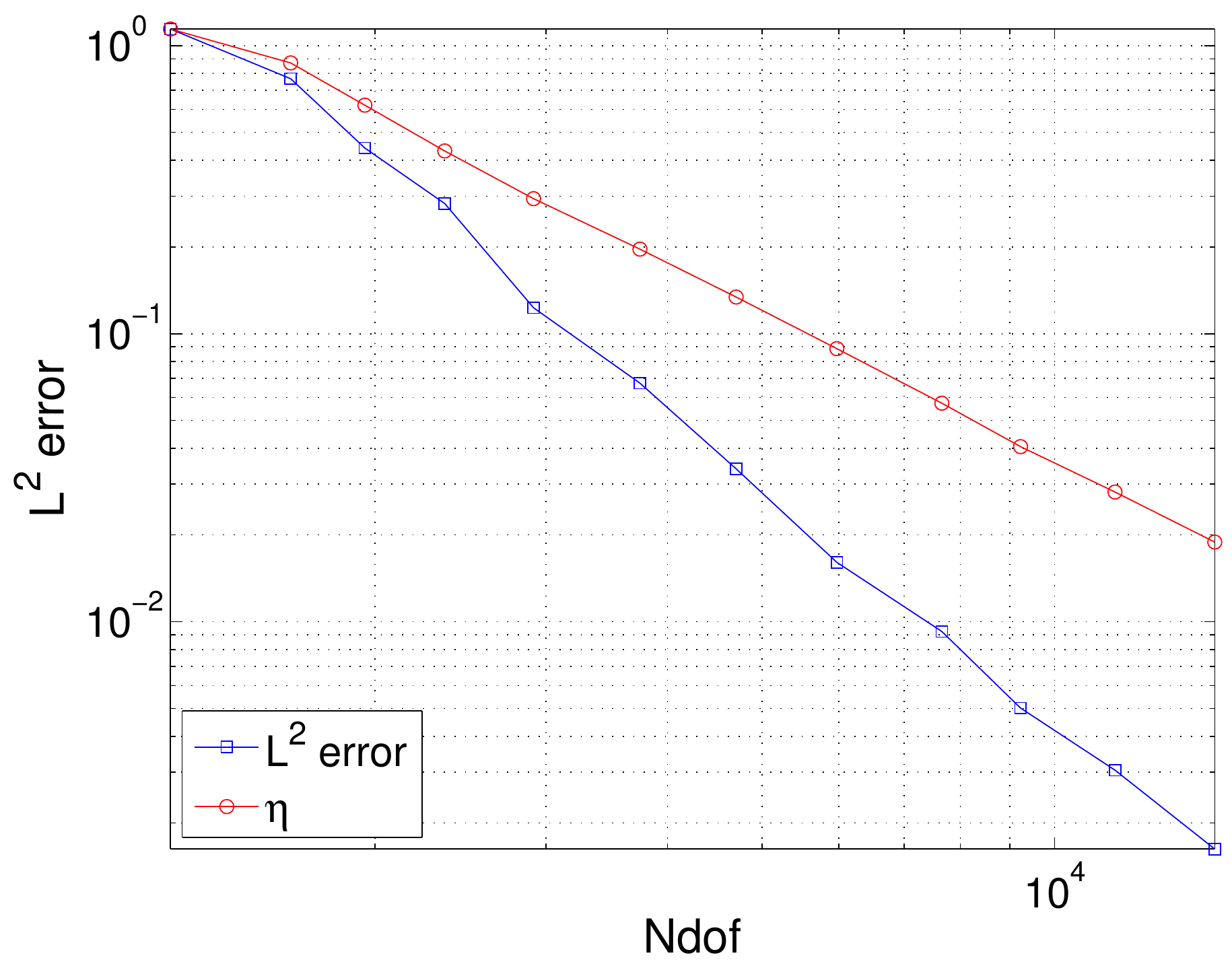}}&
\resizebox{0.4\textwidth}{!}{\includegraphics{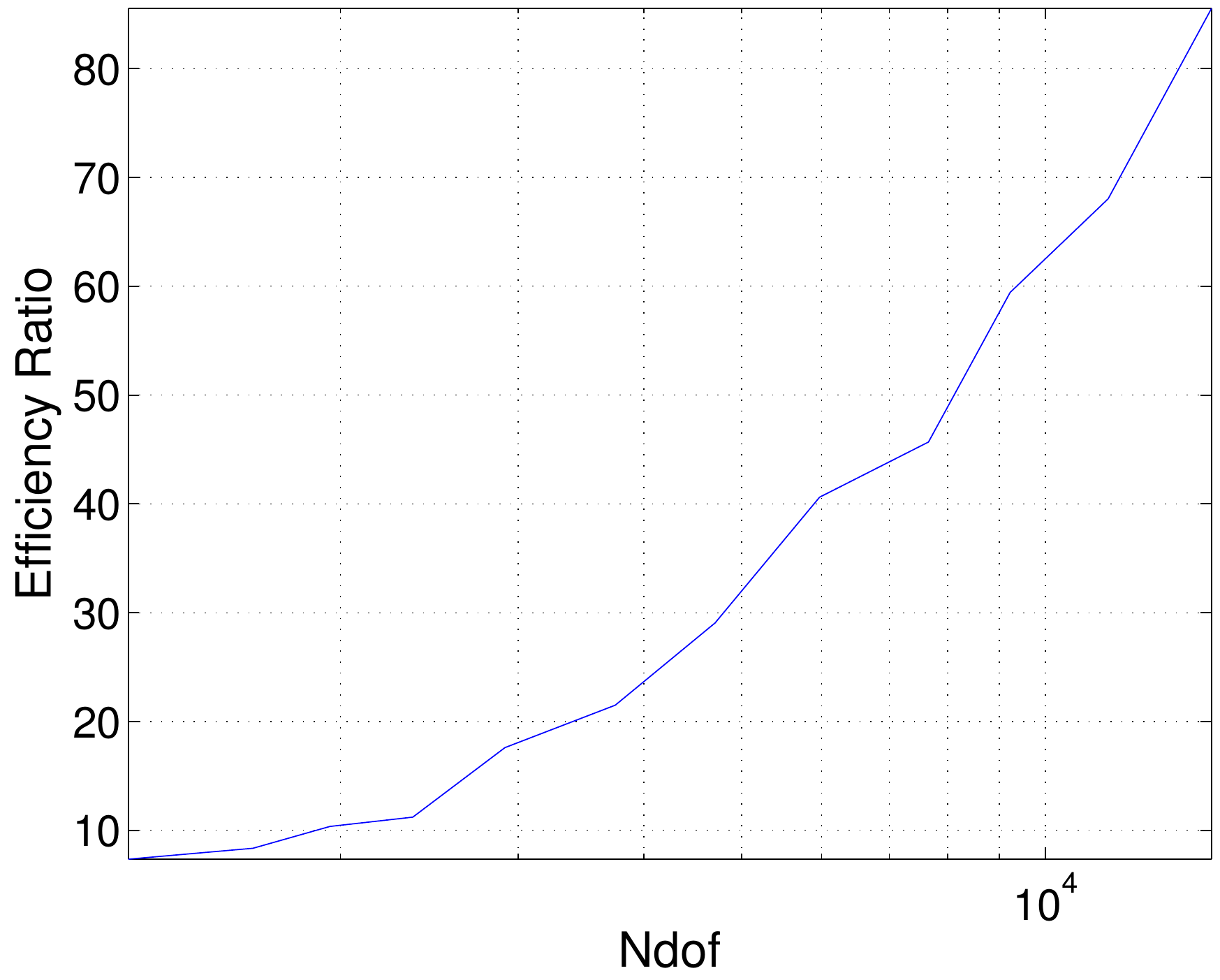}}\\
\resizebox{0.4\textwidth}{!}{\includegraphics{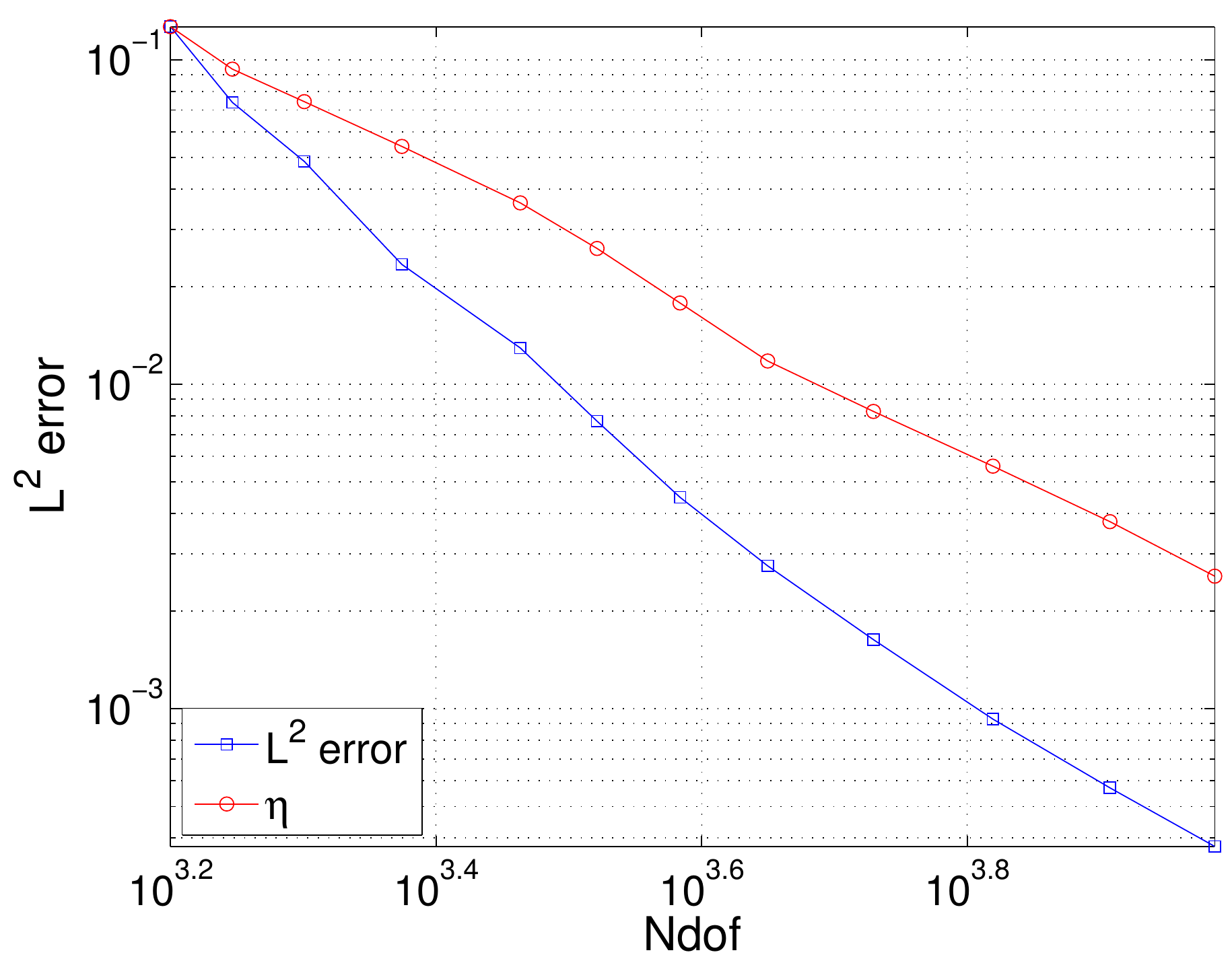}}&
\resizebox{0.4\textwidth}{!}{\includegraphics{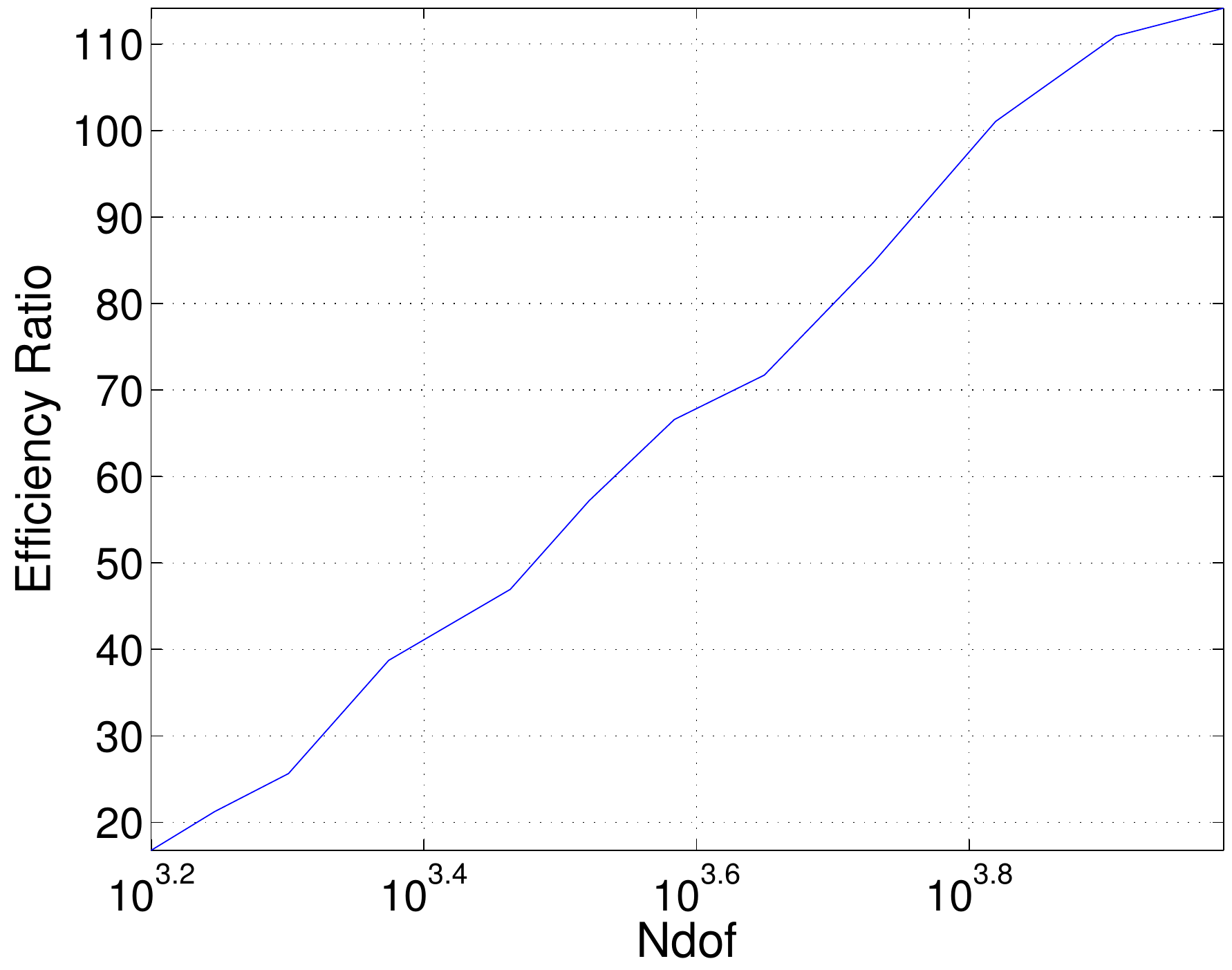}}\\
\end{tabular}
\caption{Results for total internal reflection when $p_K=5$ (top row), $p_K=7$ (middle row) and $p_K=9$ (bottom row).
Here we choose $s=1/2$.  This figure has the same layout as Fig.~\ref{L_smooth}. }
\label{IR}
\end{center}
\end{figure}

For our final results we return to the L-shaped domain and $p_K=9$.  We have seen that the efficiency of the indicator deteriorates as the mesh is refined when we take $s=1/6$ in the residual indicators.  We have also seen that the maximum
choice of $s$ is $s=1/2$ and we now test the indicator for $s=1/2$ for the smooth and singular Bessel function solutions. Results are shown in Fig.~\ref{peq9h}.  The efficiency in the $L^2$ norm is improved but still deteriorates as the mesh is refined.

\begin{figure}
\begin{center}
\begin{tabular}{ccc}
\resizebox{0.4\textwidth}{!}{\includegraphics{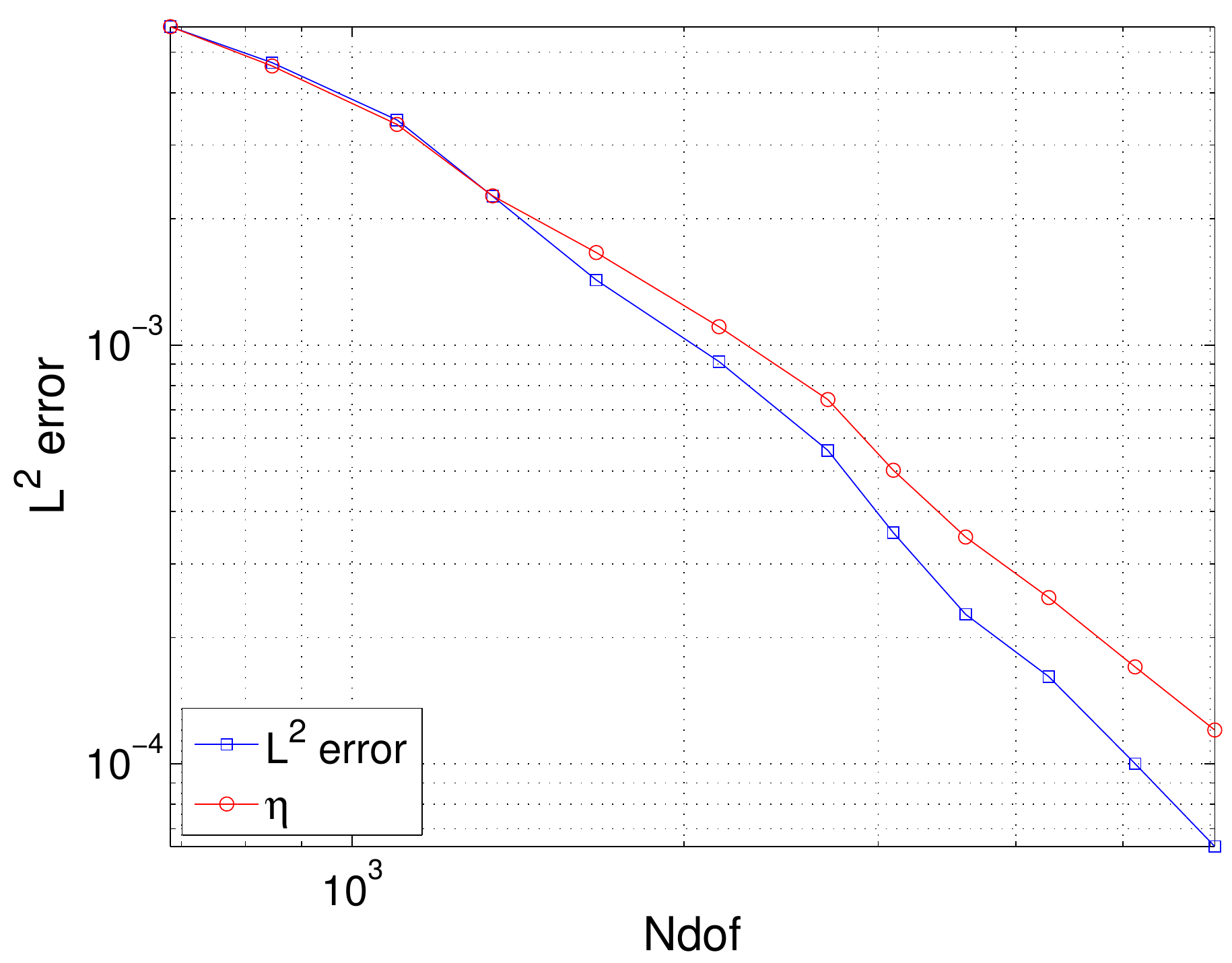}}&
\resizebox{0.4\textwidth}{!}{\includegraphics{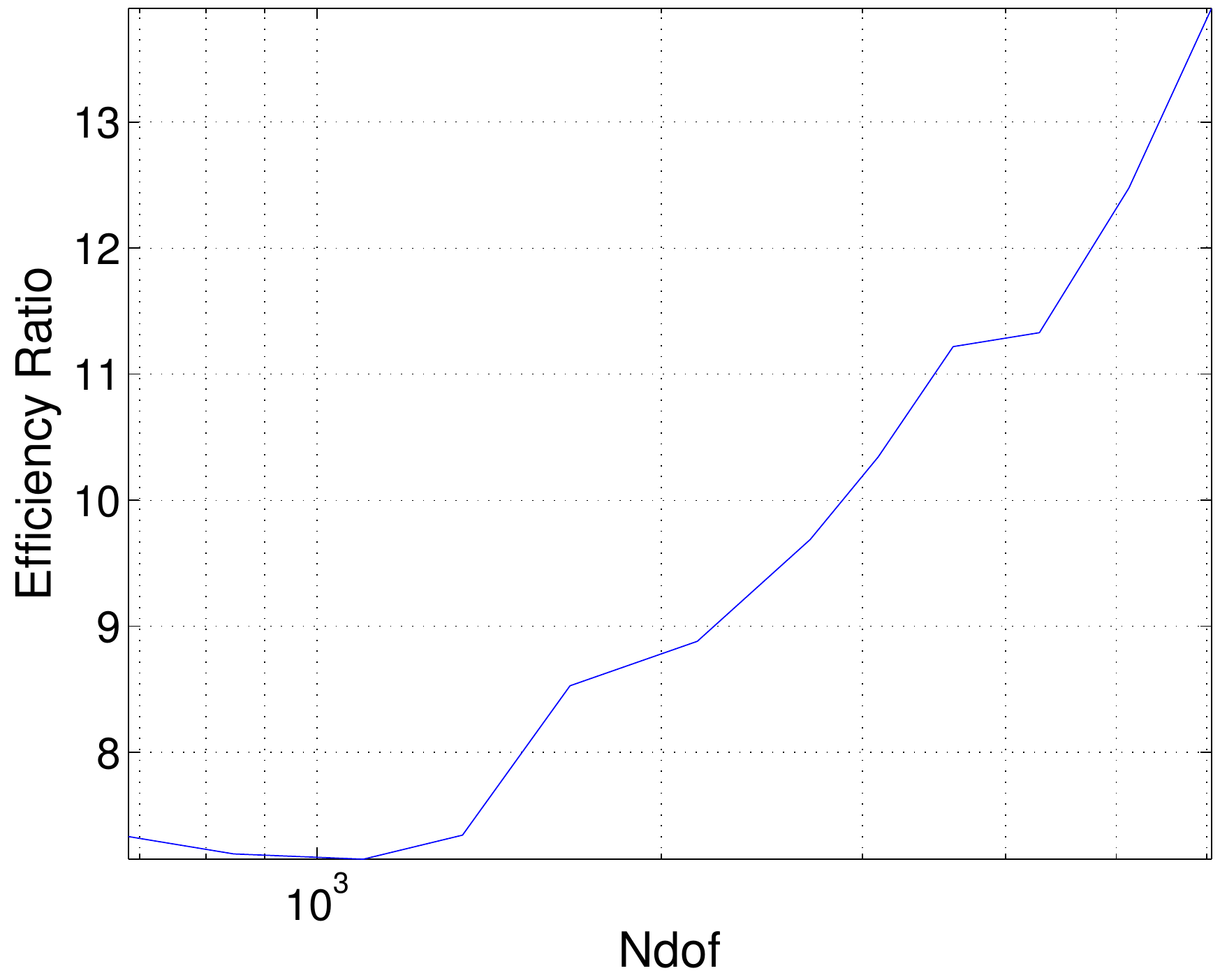}}\\
\resizebox{0.4\textwidth}{!}{\includegraphics{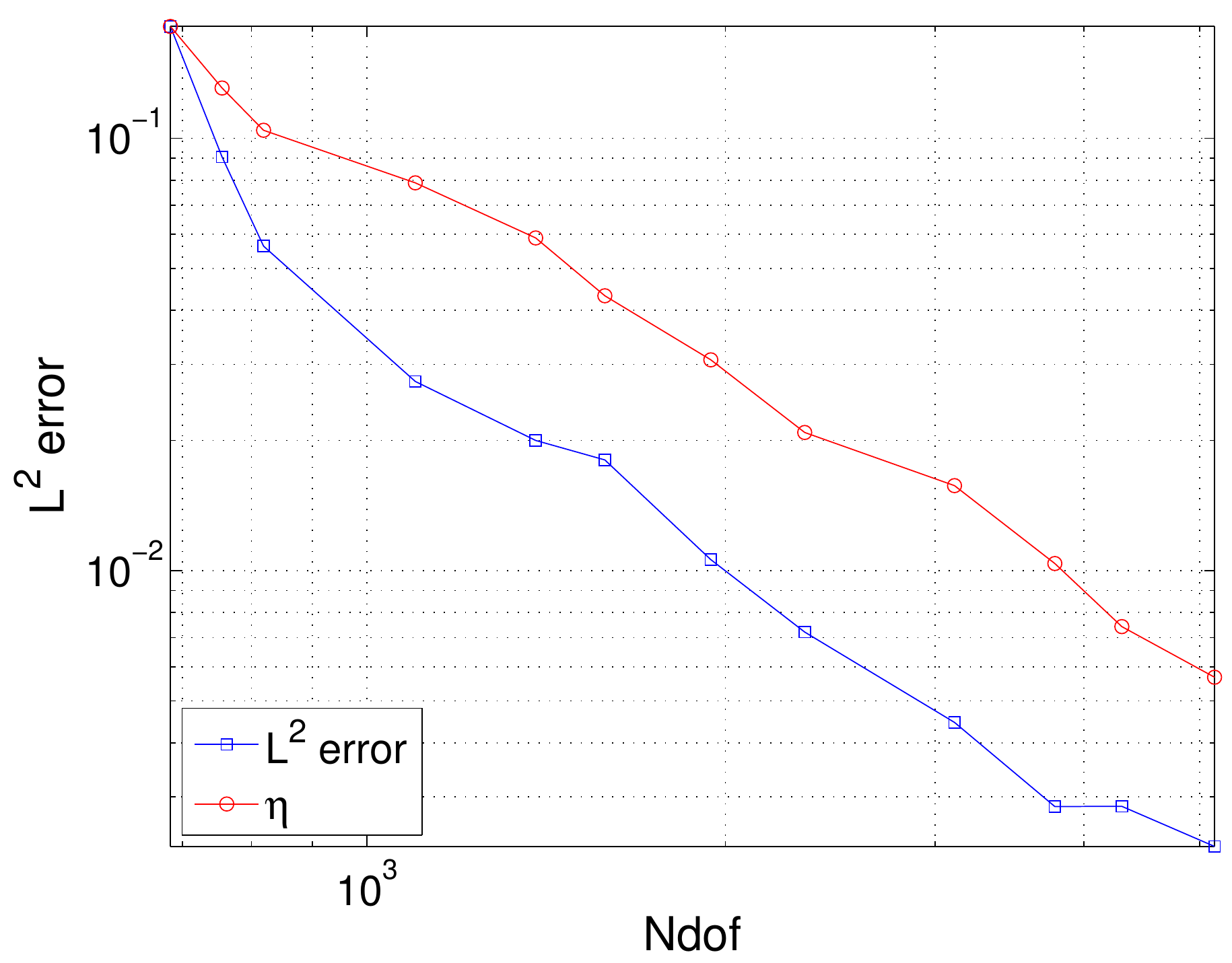}}&
\resizebox{0.4\textwidth}{!}{\includegraphics{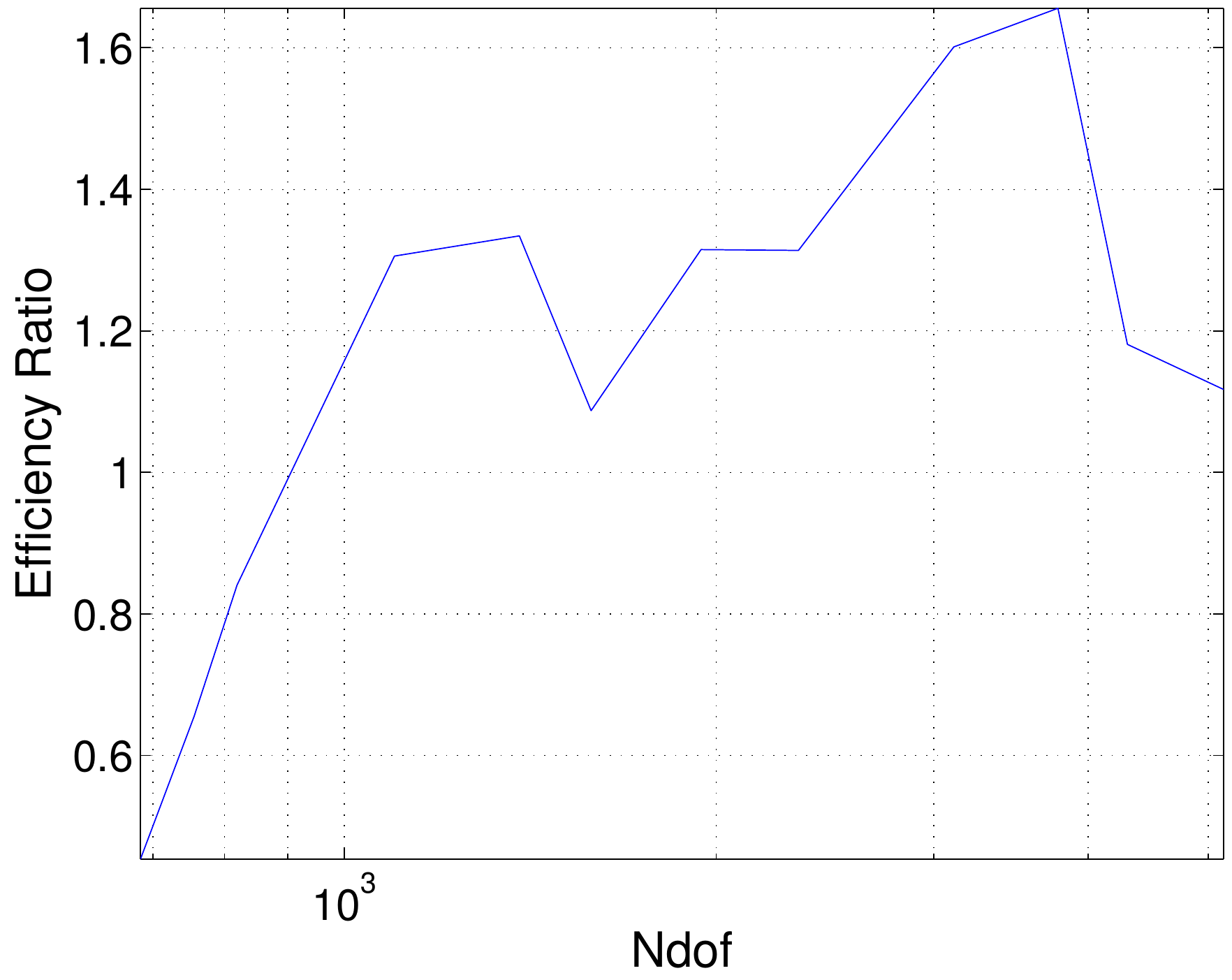}}
\end{tabular}
\caption{Results for $p_K=9$ and  $s=1/2$ on the L-shape domain. Top: smooth solution.  Bottom: singular solution. 
The columns of thus figure have the same layout as Fig.~\ref{L_smooth}. }
\label{peq9h}
\end{center}
\end{figure}

\section{Conclusion}
\label{conclusion.sec}

We have derived two new a posteriori error indicators for the PWDG method based.  One is based on standard theory and the second is based on the  observation that plane wave
basis functions can approximate piecewise linear finite elements on a fine mesh.   Using the usual Doerfler marking strategy the estimators drive mesh adaptivity that gives  convergence for a smooth solution as well as 
coping with singularities and evanescent modes.  The indicators give apparently 
reliable estimates for the $L^2$ norm but even for the improved indicators  the efficiency tends to deteriorate as the mesh is refined.  

The indicators have a parameter $s$ that depends on the solution domain.  A safe choice is $s=0$, but better efficiency is obtained by taking larger $s$, and numerically $s=1/2$ appears to be a good choice.

Usually error is estimated in the energy  norm and we will investigate a posteriori error
indicators for the broken $H^1$ norm in a future publication.
\section{Acknowledgements}
The research of SK and PM is supported in part by NSF grant number DMS-1216620.
The research of TW is supported in part by NSF grant number DMS-1216674.

\bibliographystyle{siam}

\bibliography{PwDG1}

\end{document}